\documentclass[aos]{imsart}

\RequirePackage{amsthm,amsmath,amsfonts,amssymb,amscd,mathrsfs}
\RequirePackage[numbers]{natbib}
\RequirePackage[colorlinks,citecolor=blue,urlcolor=blue]{hyperref}
\RequirePackage{graphicx}

\startlocaldefs

\newtheorem{theorem}{Theorem}[section]
\newtheorem{Proposition}[theorem]{Proposition}
\newtheorem{Lemma}[theorem]{Lemma}
\theoremstyle{remark}
\newtheorem{Definition}[theorem]{Definition}
\newtheorem{Condition}[theorem]{Condition}
\newtheorem{Remark}[theorem]{Remark}

\def \R {{\mathbb R}}
\def \C {{\mathbb C}}
\def \Cm {{\mathbb C}}
\def \II{{\mathcal I}}
\newcommand{\Nm}{\mathbb{N}}
\newcommand{\mD}{\mathcal{D}}
\newcommand{\mK}{\mathcal{K}}
\newcommand{\mL}{\mathcal{L}}
\newcommand{\id}{\mathrm{Id}}
\newcommand{\norm}[1]{\lVert #1 \rVert}
\def\t{\tilde{\tau}}
\newcommand{\wtH}{\widetilde{H}}
\newcommand{\ran}{\text{ran}}

\endlocaldefs

\begin{document}

\begin{frontmatter}
\title{Statistical guarantees for Bayesian uncertainty quantification in non-linear  inverse problems \\ with Gaussian process priors}
\runtitle{UQ for Nonlinear Inverse Problems}

\begin{aug}
\author[A]{\fnms{Fran\c{c}ois} \snm{Monard}\ead[label=e1]{fmonard@ucsc.edu}},
\author[B]{\fnms{Richard} \snm{Nickl}\ead[label=e2,mark]{nickl@maths.cam.ac.uk}}
\and
\author[B]{\fnms{Gabriel P.} \snm{Paternain}\ead[label=e3,mark]{g.p.paternain@dpmms.cam.ac.uk}}
\address[A]{Department of Mathematics,
University of California Santa Cruz,
\printead{e1}}

\address[B]{Department of Pure Mathematics and Mathematical Statistics,
University of Cambridge,
\printead{e2,e3}}
\end{aug}

\begin{abstract}
Bayesian inference and uncertainty quantification in a general class of non-linear inverse regression models is considered. Analytic conditions on the regression model $\{\mathscr G(\theta): \theta \in \Theta\}$ and on Gaussian process priors for $\theta$ are provided such that semi-parametrically efficient inference is possible for a large class of linear functionals of $\theta$. A general Bernstein-von Mises theorem is proved that shows that the (non-Gaussian) posterior distributions are approximated by certain Gaussian measures centred at the posterior mean. As a consequence posterior-based credible sets are valid and optimal from a frequentist point of view. The  theory is illustrated with two applications with PDEs that arise in non-linear tomography problems: an elliptic inverse problem for a Schr\"odinger equation, and inversion of non-Abelian $X$-ray transforms. New analytical techniques are deployed to show that the relevant Fisher information operators are invertible between suitable function spaces
\end{abstract}



\end{frontmatter}

\section{Introduction}

We are concerned here with a general class of non-linear inverse regression problems that arise with partial differential equations (PDEs). They involve a functional parameter $\theta$ one wishes to make inference on, a non-linear `forward map' $\theta \mapsto \mathscr G(\theta)$ describing a set of regression functions $\{\mathscr G(\theta): \theta \in \Theta\}$ defined on some domain $\mathcal X$, and statistical measurements 
\begin{equation}\label{discrete}
Y_i = \mathscr G(\theta)(X_i) + \sigma \varepsilon_i, \qquad i=1, \dots, N.
\end{equation}
Here the $(X_i)_{i=1}^N$ represent a finite `uniform' discretisation of $\mathcal X$ and the $\varepsilon_i$ are independent standard Gaussian variables scaled by a fixed noise level $\sigma>0$. 

The aim is to construct a statistically and computationally efficient algorithm that recovers $\theta$ from such data $(Y_i, X_i)_{i=1}^N$.  In applications, often more is required and one is further interested in \textit{data-driven performance guarantees for the output of the algorithm}. This task forms part of the evolving scientific paradigm of `uncertainty quantification' \cite{UQ}. In statistical terminology one is concerned with the \textit{construction of a confidence set} for aspects of the possibly infinite-dimensional parameter $\theta$. In common language this just expresses the desire to find valid `error bars' for  the output of the algorithm one has used. 

Various methods aiming to `quantify inferential uncertainty' for inverse problems involving PDEs are now available, particularly based on Bayesian posterior distributions arising from Gaussian process (and other) priors for $\theta$, as advocated in influential work by A. Stuart \cite{S10, DS16}. While these measures of uncertainty can be computed by MCMC methods (see \cite{KKSV00, KS04, CRSW13, RHL14, CMPS16, BGLFS17} and below), few statistical guarantees are available for the validity of such posterior inferences in typical PDE settings where $\mathscr G$ is non-linear and $\theta$ is modelled as a Gaussian process. The present paper attempts to shed some light on this issue.

The general results we obtain will be shown to apply to two prototypical `model problems' which are concerned with non-linear maps $f \mapsto u_f$ arising with solutions $u=u_f$ of a differential equation of the form 
\begin{equation}\label{toypde}
\mathscr D u - f u =0 \quad \text{ on } \quad  M,
\end{equation}
where $\mathscr D$ is a \textit{given} differential operator and $f$ an unknown potential defined on some domain $M$ in $\mathbb R^d$. The aim is to recover $f$ from certain measurements of $u_f$.

In our first example one takes for $\mathscr D$ an elliptic second order differential operator, in fact to simplify the exposition we only consider $\mathscr D = \Delta$ equal to the standard Laplacian. One then parameterises $f$ via a link function mapping a linear space $\Theta$ (to which Gaussian process priors can be assigned) into positive potentials $f=f_\theta>0$, and collects noisy measurements (\ref{discrete}) with $\mathcal X=M$ of the solution $\mathscr G(\theta)=u_{f_\theta}$ of the corresponding (time-independent) \textit{Schr\"odinger equation} (\ref{toypde}) with prescribed boundary values. Various  nonlinear inverse problems are of this form or can be reduced to one involving a Schr\"odinger-type equation \cite{KKL01}. To reduce the mathematical complexity of this first example, we assume that measurements \textit{throughout all of} $M$ are available, as is relevant, e.g., in \textit{photo-acoustic tomography} \cite{BU10, BR11}. 

In our second example we consider a more challenging problem where only boundary measurements (`scattering data') of the solution $u$ of (\ref{toypde}) are given. Here the differential operator $\mathscr D$ arises from the \textit{geodesic vector field} on the $2$-dimensional unit disk $M$ and one observes \textit{non-Abelian $X$-ray transforms} corresponding to the `influx' boundary values at $\mathcal X =\partial_+ SM$ of matrix-valued solutions $u_\theta$ of (\ref{toypde}) with $f=\theta$ a skew-symmetric matrix field. This non-linear geometric inverse problem appears in physical imaging problems such as \textit{neutron spin tomography}, see \cite{Hetal, Sales17} and has been studied in \cite{E, No, PSUGAFA, MNP19}. Mathematically the setting is fundamentally different from the Schr\"odinger case as the underlying PDE methods are not elliptic but of transport type. An important contribution of this article is to solve the analytical problem of inverting the Fisher information operator arising in this setting (see below for more details).

We will give rigorous frequentist ($N \to \infty$) guarantees for Bayesian uncertainty quantification methodology arising from sufficiently smooth Gaussian process priors for $\theta$ in such inverse problems. Specifically, conditions will be provided under which optimal asymptotic semi-parametric inference is possible for linear functionals $\langle \theta, \psi \rangle$ for smooth $\psi \in C^\infty,$ from data in (\ref{discrete}), and we verify these conditions for the preceding examples with the Schr\"odinger equation and non-Abelian $X$-ray transforms. As a consequence Bayesian credible sets for such parameters are shown to be valid frequentist confidence sets, providing objective large sample guarantees for uncertainty quantification. We numerically validate these theoretical findings for reasonable sample sizes $(N=600, 1000)$ in Section \ref{numbojumbo}.

 The idea behind our results is based on obtaining asymptotically exact Bernstein-von Mises type Gaussian approximations for the local fluctuations of the \textit{non-Gaussian} posterior measure near $\theta_0$. In traditional regular statistical models such approximations have a long history going back to Laplace \cite{L1812}, von Mises \cite{vM31}, Le Cam \cite{LC86} and van der Vaart \cite{vdV98}. In more complex settings with infinite-dimensional parameter spaces and inverse problems, such results are more recent and the present article contributes to the programme developed in \cite{CN13, CN14, CR15, R17, N1, MNP, NS19, GK19, NR20, CR19, CvdP20}.

Next to some standard regularity assumptions on $\mathscr G$, our results involve \textit{two key hypotheses} which are specific to a given inverse problem. The first condition we require is that posterior inference is \textit{globally consistent}, that is, that the posterior measure concentrates on a shrinking $\|\cdot\|_\infty$ neighborhood of the ground truth $\theta_0$ generating the data. Proving such results typically requires `global' stability estimates for the inverse problem and the techniques involved are thus quite different from the `local' techniques of the present paper. Consistency results of this kind were recently obtained in relevant PDE settings in \cite{MNP19, AN19, GN19} building on ideas from Bayesian nonparametric statistics \cite{vdVvZ08}. As we are dealing with difficult non-linear ill-posed inverse problems, the contraction rates obtained in our concrete model examples are comparably slow in `low regularity settings'.  Thus, in order to control the discretisation error and semi-parametric `bias' terms in our proofs, we will have to assume that the prior Gaussian process model employed is sufficiently regular (in a Sobolev sense).

The second key condition concerns the inverse of the so-called (`Fisher'-) information operator of the inverse problem. If we denote by $\mathbb I_{\theta_0} $ the linear operator obtained from linearising the non-linear map $\mathscr G$ near the ground truth parameter $\theta_0$ (one may think of it as a derivative $(\partial \mathscr G/\partial \theta)_{|\theta = \theta_0}$ in a suitable sense), then general statistical theory (reviewed in Section \ref{optimality} below) suggests that a canonical asymptotic approximation to the posterior measure for $\theta$ should arise from a Gaussian measure with covariance operator $\mathbb I_{\theta_0} (\mathbb I_{\theta_0}^* \mathbb I_{\theta_0})^{-1}$ where $\mathbb I^*_{\theta_0}$ is an appropriate adjoint of $\mathbb I_{\theta_0}.$ Moreover this operator provides a benchmark for the optimum any uncertainty quantification algorithm can achieve. What precedes can be made rigorous, however, only if the \textit{information (or normal) operator} $\mathbb I_{\theta_0}^* \mathbb I_{\theta_0}$ is surjective onto a large enough range, and if the mapping properties of its inverse allow for the composition of $\mathbb I_{\theta_0}$ with $(\mathbb I_{\theta_0}^* \mathbb I_{\theta_0})^{-1}$. In the settings above this is not at all clear a-priori and in fact generates new PDE questions in its own right. For the Schr\"odinger equation problem it was shown in \cite{N1} using elliptic theory that $\mathbb I_{\theta_0}^* \mathbb I_{\theta_0}$ indeed is invertible (in fact, its inverse equals a certain type of iterated Schr\"odinger operator). We extend here the results in \cite{N1} to allow for Gaussian priors and a more general discrete measurement setting (under suitable hypotheses). For the non-Abelian $X$-ray case, inversion of $\mathbb I_{\theta_0}^* \mathbb I_{\theta_0}$ is a more delicate problem
that we successfully solve in this paper using recent techniques from \cite{M}. We refer to Remark \ref{remark:themuaffair} for some context and perspectives on this result. At this point it suffices to point out that 
the statistical questions explored here and in \cite{MNP, MNP19} are drivers of new developments in geometric inverse problems.

This paper is organised as follows: The main results for the PDE models arising from (\ref{toypde}) are given in Section \ref{mainpde}, whereas the general theory for Bayesian inference in non-linear random design regression models is developed in Section \ref{general}. All proofs are given in subsequent sections, and the results on the information geometry of non-Abelian $X$-ray transforms are presented in Section \ref{sec:mainXray}.  Throughout, for $\mathcal X$ a suitable open subset of Euclidean space, we use standard notation for H\"older spaces $C^\beta(\mathcal X)$ of $[\beta]$-times ($[\cdot]$ denotes integer part) continuously differentiable functions whose partial derivatives of order $[\beta]$ satisfy a $\beta-[\beta]$-H\"older continuity condition on $\mathcal X$. We define the usual Sobolev spaces $H^\alpha(\mathcal X)$ of functions with $L^2(dx)$-derivatives up to order $\alpha$, defined for $\alpha \notin \mathbb N$ by interpolation. Finally, for $V$ a normed vector space, $C^\infty(\mathcal X, V)$ denotes all smooth $V$-valued functions defined on $\mathcal X$, and $C^\infty_c(\mathcal X, V)$ denotes the subspace of $C^\infty(\mathcal X, V)$ consisting of functions that are compactly supported in the interior of $\mathcal X$. In Section \ref{sec:mainXray} these definitions will also be used when $\mathcal X=M$ is a Riemannian manifold $M$ with boundary.

\section{Main results for PDE models}\label{mainpde}

\subsection{General observation setting, prior and posterior}

Let $(\mathcal X, \mathcal A)$ and $(\mathcal Z, \mathcal B)$ be measurable spaces equipped with measures $\lambda, \zeta$, respectively. We assume that $\lambda$ is a probability measure and that $\zeta$ a finite measure. Let further $V,W$ be finite-dimensional vector spaces of fixed finite dimensions $p_V, p_W \in \mathbb N$, with inner products $\langle \cdot, \cdot \rangle_W, \langle \cdot, \cdot \rangle_V$, respectively. Let
$$ L^\infty(\mathcal X), \quad L^2(\mathcal X) = L^2_\lambda(\mathcal X,V) \quad \text{ and } \quad L^\infty(\mathcal Z), \quad L^2(\mathcal Z) = L^2_\zeta(\mathcal Z, W)$$ denote the bounded measurable, and $\lambda$- or $\zeta$- square integrable, $V$ or $W$-valued functions defined on $\mathcal X, \mathcal Z$, respectively. Denote by $\|\cdot\|_{L_\zeta^2(\mathcal Z)}, \|\cdot\|_{L_\lambda^2(\mathcal X)}$ the usual $L^2$-norms on these spaces, and by  $\langle \cdot, \cdot \rangle_{L^2_\zeta(\mathcal Z)}, \langle \cdot, \cdot \rangle_{L_\lambda^2(\mathcal X)}$ the corresponding Hilbert space inner products; and write $\|\cdot\|_\infty$ for the supremum norm.

We will consider parameter spaces $\Theta$ that are (Borel-measurable) \textit{linear} subspaces of $L^\infty(\mathcal Z, W)$, on which  measurable `forward maps' 
\begin{equation}\label{fwdG}
\theta \mapsto \mathscr G(\theta), \qquad \mathscr G : \Theta \to L^2_\lambda(\mathcal X, V),
\end{equation}
are defined. Observations then arise in a general random design regression setup where one is given jointly i.i.d.~random variables  $(Y_i, X_i)_{i=1}^N$ of the form
\begin{equation} \label{model}
Y_i = \mathscr G(\theta)(X_i) + \varepsilon_i, \quad \varepsilon_i \sim^{i.i.d} N(0, \sigma^2 I_V), \quad \sigma>0, \quad i=1, \dots, N,
\end{equation}
where the $X_i$'s are random i.i.d.~covariates drawn from law $\lambda$ on $\mathcal X$. We assume that the covariance $I_V$ of each noise vector $\varepsilon_i \in V$ is diagonal for the inner product of $V$. Correlated Gaussian noise can be accommodated simply by adjusting the choice of inner product on $V$. Conditions on the `experiments' underlying our regression model enter our results only through the probability measure $\lambda$ generating the $X_i$'s. In common cases where $\lambda$ represents a uniform distribution on some bounded domain in Euclidean space, a deterministic design regression model with `equally spaced' design $X_i=x_i$ can be seen to be statistically equivalent to (\ref{model}), see \cite{R08}, and our analysis thus also informs such measurement setups. We opt to present the theory here in a \textit{random} design model as it allows for a unified probabilistic treatment of the numerical discretisation error in the proofs.

If the natural domain on which $\mathscr G$ is defined is not a linear space, one can employ `link functions'  that map $\Theta$ into the relevant domain. The new forward map then consists of the composition of that link function with the initial forward map. See Section \ref{schrott} below for an example. We insist that $\Theta$ be a linear space so that Gaussian process priors can be assigned to it.

To fix notation: The joint law of the random variables $(Y_{i}, X_i)_{i=1}^N$ in (\ref{model}) defines a product probability measure on $(V \times \mathcal X)^N$, and it will be denoted by $P_\theta^N = \otimes_{i=1}^N P_\theta^i$, where we note $P_\theta^i=P_\theta^1$ for all $i$. The infinite product probability measure $\otimes_{i=1}^\infty P_\theta^i$ describing the law of all possible infinite sequences of observations (in $(V \times \mathcal X)^\mathbb N$) will be denoted by $P_{\theta}^\mathbb N$. We also write shorthand
\begin{equation}\label{data}
D_N = \{Y_1, \dots, Y_N, X_1, \dots, X_N\},  \qquad N \in \mathbb N,
\end{equation}
 for the given data vector.

\smallskip

Now given a prior probability measure $\Pi$ on $\Theta$ to be specified, and assuming $\theta \sim \Pi$, we make the Bayesian model assumption that $$(Y_{i}, X_i)_{i=1}^N|\theta \sim P_\theta^N$$ which by Bayes' rule generates a conditional posterior distribution of $\theta|(Y_{i}, X_i)_{i=1}^N$ on $\Theta$ --  it will be denoted by $\Pi(\cdot|(Y_{i}, X_i)_{i=1}^N) \equiv \Pi(\cdot|D_N)$. The posterior distribution arises from a dominated family of probability measures (assuming joint measurability of the map $(\theta, x) \to \mathscr G(\theta)(x)$) and is hence given by
\begin{equation}\label{post}
\Pi(A|D_N)\equiv \Pi(A|Y_1, \dots, Y_N, X_1, \dots, X_N) = \frac{\int_A e^{\ell_N(\theta)}d\Pi(\theta)}{\int_\Theta e^{\ell_N(\theta)}d\Pi(\theta)},
\end{equation}
for any Borel set $A$ in $\Theta$. Here, by independence
\begin{equation} \label{likpo}
\ell_N(\theta) = \sum_{i \le N}\ell_i(\theta), \quad \text{ where }\quad \ell_i(\theta)  =  -\frac{1}{2\sigma^2} \|Y_{i} -\mathscr G(\theta)(X_i)\|_V^2,
\end{equation}
is, up to additive constants, the log-likelihood function of the observations. 

\subsection{Gaussian process priors for inverse problems} \label{wmp}

Gaussian priors are widely used in Bayesian inverse problems since \cite{KKSV00, KS04}, among others for uncertainty quantification purposes as discussed in the introduction. In the `non-parametric' setting advocated by Stuart \cite{S10}, when the parameter of interest is a function $\theta: \mathcal Z \to W$, the infinite-dimensional notion of a Gaussian prior is the one of a random map arising from a centred \textit{Gaussian process} (see, e.g., \cite{GN16, GvdV17} for background).

For example, if $\mathcal Z$ is a bounded smooth domain in $\mathbb R^d$, a \textit{Whittle-Mat\'ern process} with index set $\mathcal Z$ and regularity parameter $\alpha$ (cf. Example 11.8 in \cite{GvdV17})  arises as the stationary centred Gaussian process $G=\{G(z), \ z\in \mathcal Z\}$ with covariance kernel 
$$K(x,y)=\int_{\R^d}e^{-i\langle x-y,\xi\rangle_{\R^d}}\bar \mu(d\xi), 
	\quad \bar \mu(d\xi)=(1+\|\xi\|^2_{\R^d})^{-\alpha}d\xi, \quad x,y \in \mathcal Z.$$
From the results in Chapter 11 in \cite{GvdV17} we see that the \textit{reproducing kernel Hilbert space} (RKHS) of $(G(z): z \in \mathcal Z)$ equals the set of restrictions to $\mathcal Z$ of elements in the Sobolev space $H^\alpha(\mathbb R^d)$, which coincides, with equivalent norms, with the Sobolev space $H^\alpha(\mathcal Z)$ over $\mathcal Z$. Moreover, one shows (as in the proof of Lemma I.4 in \cite{GvdV17}) that $G$ has a version with paths belonging almost surely to the H\"older spaces $C^{\beta'}(\mathcal Z)$ for all $\beta'<\alpha-d/2$, and thus defines a Gaussian Borel probability measure on $\Theta=C(\mathcal Z)$ whenever $\alpha>d/2$ (and in fact in $C^\beta(\mathcal Z)$ for any $\beta<\beta'$).

\smallskip

A key challenge for implementation is of course the computation of the posterior distribution in such settings. When the forward map $\mathscr G$ is \textit{linear} then one can show that the posterior distribution (\ref{post}) will also be a Gaussian measure on $\Theta$ so that posterior sampling is fairly straightforward (see \cite{KS04} and, for concrete implementation with Whittle-Mat\'ern priors, e.g., \cite{MNP}). In the case where $\mathscr G$ is non-linear, so that the posterior is \textit{not Gaussian} any more, MCMC methods can be readily used as long as the forward map (and possibly its gradient) can be numerically evaluated, providing feasible statistical methodology for non-linear problems see, e.g., \cite{KKSV00, KS04, CRSW13, VSS13, RHL14, CMPS16, BGLFS17, MNP19} and also Section \ref{numbojumbo} below. Computational guarantees for convergence of such algorithms are also available in the high-dimensional and non-log-concave setting relevant here, see \cite{HSV14, NW20} and references therein.

\smallskip

Regarding statistical (frequentist) properties of posterior measures, the case of linear $\mathscr G$ is again fairly well understood due to the explicit Gaussian structure of the posterior distribution, we refer here only to \cite{KVV11, R13, ALS13, KLS16, MNP, GK19, GVD20} and references therein. The non-linear case, however, remains a formidable challenge. While consistency and contraction rates for Bayesian methods have been established very recently in some settings \cite{MNP19, AN19, GN19}, no guarantees are currently available for the task of uncertainty quantification investigated here (except for \cite{N1} to be discussed below).
  
To address this challenge we will prove Bernstein-von Mises theorems which entail that under suitable hypotheses the \textit{non-Gaussian} posterior measure $\Pi(\cdot|D_N)$ is approximated, in the sense of weak convergence, by a Gaussian distribution with a canonical covariance structure. Our results will hold in $P_{\theta_0}^N$-probability, where $\theta_0$ is the ground truth parameter generating the data (\ref{model}), and for all linear functionals $\langle \theta, \psi \rangle_{L^2}, \theta \sim \Pi(\cdot|D_N),$ with $\psi$ a test function. To limit technicalities we assume that both $\theta_0$ and $\psi$ are \textit{smooth} -- relaxing such conditions is possible (adapting arguments from \cite{N1}) but not the scope of the present paper. 

Rigorous statements will involve an arbitrary metric $d_{weak}$ for weak convergence of probability measures on $\mathbb R$ (see \cite{D02}). If $E^\Pi[\theta|D_N]$ is the posterior mean (a Bochner integral in $C(\mathcal Z)$), if $\Pi^\psi(\cdot|D_N)$ denotes the (through $D_N$ random) probability law of $$\sqrt N \langle \theta - E^\Pi[\theta|D_N], \psi \rangle_{L^2_\zeta(\mathcal Z)},  \qquad \theta \sim \Pi(\cdot|D_N),$$ and for a normal $N(0, \sigma_\psi^2)$ distribution with variance $\sigma^2_\psi$ to be specified, we will prove limit theorems of the form
\begin{equation} \label{weakprob}
d_{weak}(\Pi^\psi(\cdot|D_N), N(0, \sigma^2_\psi)) \to_{N \to \infty}^{P^N_{\theta_0}} 0.
\end{equation}
When (\ref{weakprob}) holds we shall often just say that $\sqrt N \langle \theta - E^\Pi[\theta|D_N], \psi \rangle_{L^2_\zeta(\mathcal Z)}\to^d N(0, \sigma^2_\psi)$ in $P_{\theta_0}^N$-probability, where $\to^d$ denotes convergence in distribution. We obtain general results of this kind in Section \ref{general} but first give their explicit consequences for the main examples (\ref{toypde}) of the Schr\"odinger equation and  non-Abelian $X$-ray transforms.

\subsection{Normal approximation for the Schr\"odinger equation} \label{schrott}

We now consider an inverse problem for a steady state Schr\"odinger equation. Such problems have applications in photo-acoustic tomography \cite{BU10, BR11} and have been studied recently in the Bayesian inference setting in \cite{N1}. For a bounded smooth domain $\mathcal X=\mathcal Z$ in $\mathbb R^d, d \in \mathbb N,$ with boundary $\partial \mathcal X$, let $\lambda=\zeta$ equal the Lebesgue measure on $\mathcal X$ normalised to one. Then consider solutions $u_f$ of the elliptic boundary value problem 
\begin{equation}\label{schrodinger}
	\begin{cases}
	\frac{1}{2} \Delta u -fu =0 \quad \textnormal{on }  \mathcal X,\\
	u=g\quad \textnormal{on } \mathcal \partial \mathcal X,
	\end{cases}
	\end{equation}
	where $f: \mathcal X \to (0,\infty),$ is a positive potential, where $\Delta$ is the Laplacian, and where $g: \partial \mathcal X \to [g_{min}, \infty), g_{min}>0,$ are given smooth `boundary temperatures'. For $\theta \in C(\mathcal X)$ we will parameterise $f = \phi \circ \theta$ where $\phi: \mathbb R \to (f_{min},\infty), f_{min}\ge 0,$ is a smooth bijective `regular link' function chosen as in \cite{NvdGW18}, satisfying in particular $\phi(0)=1$ and $\phi'>0$. We shall write $\phi(\theta), \phi'(\theta)$ for $\phi \circ \theta$ and $\phi' \circ \theta$, respectively, when no confusion can arise. In the notation from earlier in this section we set $$\mathscr G(\theta) \equiv u_{\phi \circ \theta} \in L_\lambda^2(\mathcal X),  \quad V=W=\mathbb R,$$ where we note that for $f=\phi \circ \theta, \theta \in C^\beta(\mathcal X), \beta>0$, a unique $C^2$-solution $u_f$ of (\ref{schrodinger}) exists by standard results for elliptic PDEs \cite{GT98}. Measurements in (\ref{model}) are thus collected \textit{throughout} the domain $\mathcal X$ -- results for the case where only boundary measurements at $\partial \mathcal X$ are available (`Calder\'on type problems') will require a different approach as the inverse problem is then statistically `severely-ill-posed' (see \cite{AN19}).

Now draw $\theta'$ from an $\alpha$-regular Whittle-Mat\'ern Gaussian process (cf.~Subsection \ref{wmp}) supported in $C^\beta(\mathcal X)$ for $0<\beta<\alpha-d/2$, and let the prior $\Pi=\Pi_N$ be the law on $\Theta \equiv C^\beta(\mathcal X)$ of the random function 
\begin{equation} \label{schrottprior}
\frac{\theta'(x)}{N^{d/(4\alpha+2d)}}, \quad x \in \mathcal X.
\end{equation} 
The $N$-dependent rescaling provides additional regularisation of the posterior distribution which is crucial to deal with the global non-linearity of the inverse problem in the proofs (cf.~also Remark 3.5 in \cite{MNP19}).

To state the following theorem, define the space $C^{\infty, 2}(\mathcal X)$ consisting of real-valued functions $f \in C^\infty(\mathcal X)$ such that the partial derivatives $(D^jf)_{| \partial \mathcal X}=0$ vanish for all multi-indices $j$ of order $0 \le |j| \le 2$. Evidently $C^\infty_c(\mathcal X) \subset C^{\infty,2}(\mathcal X)$. We also introduce the Schr\"odinger operator $$\mathbb S_{f}[w] = \frac{1}{2} \Delta w - fw, \quad  w \in C^2(\mathcal X),$$ appearing in the expression for the asymptotic variance. The following theorem extends related results in \cite{N1} to Gaussian process priors, and to the more realistic measurement setting (\ref{discrete}), if the true parameter $\theta_0$ and test function $\psi$ define appropriate elements of $C^\infty(\mathcal X)$. 

\begin{theorem}\label{mainschrott}
Consider the prior $\Pi_N$ from (\ref{schrottprior}) with integer regularity $\alpha$ satisfying 
\begin{equation}\label{ohlord}
\frac{\alpha}{2\alpha+d}\frac{\alpha-d}{\alpha+2-d/2} >\frac{1}{3}.
\end{equation}
Let $\theta \sim \Pi(\cdot|D_N)$ where $\Pi(\cdot|D_N)$ is the posterior measure (\ref{post}) on $\Theta$ arising from observations $D_N$ in model (\ref{model}) with $\mathscr G(\theta)$ the solution of the Schr\"odinger equation (\ref{schrodinger}), $f=\phi \circ \theta$, and where $\phi: \mathbb R \to (f_{min}, \infty), f_{min} \ge 0,$ is a regular link function. Denote the posterior mean by $\bar \theta_N = E^\Pi[\theta|D_N]$, and let $\psi \in C^{\infty,2}(\mathcal X)$. Assume $f_0 = \phi \circ \theta_0$ for some $\theta_0 \in C^\infty(\mathcal X)$ such that $\inf_{x \in \mathcal X}f_0(x) > f_{min}$. Then we have as $N \to \infty$,
$$\sqrt N \langle \theta - \bar \theta_N, \psi \rangle_{L_\lambda^2(\mathcal X)} | D_N  \to^d N(0,  \sigma^2(f_0, \psi)) \quad \text{ in } P_{\theta_0}^N- \text{probability},$$ and moreover that $$\sqrt N \langle \bar \theta_N - \theta_0, \psi \rangle_{L_\lambda^2(\mathcal X)} \to^d N(0, \sigma^2(f_0, \psi))$$ where the asymptotic variance is given by
\begin{equation} \label{varschrott}
\sigma^2(f_0, \psi) = \Big\|\mathbb S_{f_0}\Big[\frac{\psi}{u_{f_0} \phi'(\theta_0)}\Big]\Big\|^2_{L^2_\lambda(\mathcal X)}.
\end{equation}
\end{theorem}

The boundary conditions $u=g>0$ on $\partial \mathcal X$ and regularity assumption $\theta_0 \in C^\infty(\mathcal X)$ ensure that the inverse of the underlying information operator (which is an elliptic order-4 type operator, see (\ref{invopod}) below) exists and maps $C^{\infty,2}(\mathcal X)$ into $C^\infty(\mathcal X)$. This fact is used crucially in the proofs and also implies finiteness of $\sigma^2(f_0, \psi)$ in (\ref{varschrott}).

In the proofs we establish a non-parametric contraction rate $\bar \delta_N \to 0$ of the posterior measure about $\theta_0$ in $\|\cdot\|_\infty$-distance. The rate $\bar \delta_N$ improves if the Gaussian process prior model is more regular. To control non-linear semi-parametric bias terms in the Bernstein-von Mises approximation we require $N \bar \delta_N^3 = o(1)$ in our proofs, giving rise to the condition (\ref{ohlord}). For instance when $d=2$ this requires $\alpha >10$. This can be weakened by obtaining a faster rate than $\bar \delta_N$ (the optimal rate is obtained in \cite{N1} for more restrictive measurements and non-Gaussian priors), but we do not pursue this issue here as we require $\theta_0 \in C^\infty(\mathcal X)$ at any rate (for the mapping properties of the information operator).

\subsection{Normal approximation for non-Abelian $X$-ray transforms}\label{xraynorm}

We now present results comparable to those from the previous subsection for the non-Abelian $X$-ray transform as considered in \cite{PSUGAFA, MNP19}. Applications to neutron spin tomography can be found in \cite{Hetal, Sales17}, see also Section 1.2 in \cite{MNP19}

We let $M\subset \mathbb{R}^2$ be the closed unit disk with boundary $\partial M$. We consider lines in the plane (i.e. geodesics) parametrized by $\gamma(t)=x+tv$, where $x\in \mathbb{R}^{2}$ and $v$ is a direction on the unit circle $S^{1}$.
We only want those lines intersecting our region $M$ of interest and further introduce the influx and outflux boundaries as
\begin{align*}
    \partial_{+} SM &= \{(x,v)\in\partial M\times S^{1}:\;x\cdot v\leq 0\}, \\
    \partial_{-} SM &= \{(x,v)\in\partial M\times S^{1}:\;x\cdot v\geq 0\},    
\end{align*}
where $\cdot$ is the standard dot product in the plane. If we take $(x,v)\in \partial _{+}SM$, then the line $\gamma(t)=x+tv$ will exit the disk in time
\[\tau(x,v):=-2\,x\cdot v.\]

Let $\Phi:M\to \C^{n\times n}$ be a continuous matrix field. Given a line segment (geodesic) $\gamma:[0,\tau]\to M$ with endpoints $\gamma(0), \gamma(\tau)\in \partial M$, we consider $\C^{n\times n}$-valued functions $U=U(t), 0 \le t \le \tau,$ solving the matrix ODE
\begin{align*}
   \frac{d}{dt} U(t) + \Phi(\gamma(t)) U(t) = 0, \qquad U(\tau) = \id.
\end{align*}
We define the scattering data of $\Phi$ on $\gamma$ to be $C_\Phi(\gamma):= U(0)$.
This problem, backward in time for convention here, is well-posed and leads to a unique definition of $U(0)$, containing information about $\Phi$ along the geodesic $\gamma$. Note that when $\Phi$ is scalar, we obtain $\log U(0) = \int_0^{\tau} \Phi (\gamma(t))\ dt$, which is the classical X-ray/Radon transform of $\Phi$ along the ray $\gamma$. Considering the collection of all such data makes up the {\em non-Abelian X-ray transform} of $\Phi$, viewed here as a map 
\begin{align}
    C_\Phi\colon \partial_+ SM\to \C^{n\times n},
    \label{eq:Cphi}
\end{align}
and the goal is to recover $\Phi$ from $C_\Phi$. Inverting Abelian and non-Abelian X-ray transforms are examples of inverse problems in integral geometry, an active field permeating several tomographic imaging methods, see also the recent topical review \cite{IlMo}.  We are most interested here in the case where $\Phi$ takes values in the Lie algebra $\mathfrak{so}(n)$ of skew-symmetric matrices associated to the special orthogonal group $SO(n)$. In this case the scattering data $C_{\Phi}$ maps into $SO(n)$ and the map $\Phi \mapsto C_\Phi$ is known to be injective \cite{E,No,PSUGAFA}. Also, for $n=3$ this is the relevant problem for neutron spin tomography \cite{Hetal, Sales17}. 

Since $M$ is the unit disk, we can parametrise its boundary (the unit circle) $\partial M$ with an angular variable 
$\phi$; similarly the vectors $v$ pointing inside $M$ can be parametrized with an angular variable $\varphi\in [-\pi/2,\pi/2]$ (fan-beam coordinates). The influx boundary $\partial_{+}SM$ can hence be equipped with a normalized area form $\lambda$, $d\lambda:=d\phi\,d\varphi/2\pi^2$.
The other common measure in use is $\cos\varphi d\lambda$ (the symplectic measure) and as we comment below in Remark \ref{remark:themuaffair} the ramifications of choosing one over the other in terms of the Fisher information operator go quite deeply.
In this paper we work exclusively with $\lambda$ as in \cite{MNP19}.

The non-Abelian X-ray transform can be cast into the general statistical model setting from (\ref{model}) as follows: We set $\mathcal Z=M$ endowed with its volume element $\zeta=dx$, and $\mathcal X = \partial_+ SM$ with $\lambda$ defined above. The vector spaces $V=W$ can be taken to equal the space of $n \times n$ real matrices with Frobenius inner product $\langle \cdot, \cdot \rangle_F$. The standard element-wise basis $e_{jk}=\delta_{jk}, 1 \le j,k \le n,$ of $V$ then allows to realise the random vector $\varepsilon \sim N(0, I_V)$ as the i.i.d.~sequence $\varepsilon_{j,k}\sim N(0,1), 1 \le j, k \le n,$ considered in the noise model in \cite{MNP19}. Next we let $\Theta = \times_{j=1}^{\text{\rm dim}(\mathfrak{so}(n))} C(M)$ denote the space of all continuous maps defined on $M$ taking values in $\mathfrak {so}(n)$. Identifying $\theta =\Phi$, the non-linear forward map is then $\mathscr G(\theta)=C_\theta=C_\Phi$ from (\ref{eq:Cphi}).

The linearisation $\mathbb I_{\theta_0}$ of $\mathscr G$ at $\theta_0$ provides a bounded linear map from $L_\zeta^2(M)$ to $L^2_\lambda(\partial_+SM)$ with adjoint $\mathbb I_{\theta_0}^*: L^2_\lambda(\partial_+SM) \to L^2_\zeta(M)$, see Section \ref{sec:mainXray}. There it is further shown that for $\theta_0 \in C^\infty_c(M, \mathfrak{so}(n))$ the information operator $\mathbb I_{\theta_0}^*\mathbb I_{\theta_0}$  is invertible on $C^\infty(M)=C^\infty(M, V)$, in particular,
\begin{equation}
\psi \in C^\infty(M, \mathfrak{so}(n)) \quad \implies \quad \tilde \psi = (\mathbb I_{\theta_0}^* \mathbb I_{\theta_0})^{-1}\psi \in C^\infty(M, \mathfrak{so}(n)).
\label{eq:inversionfisher}
\end{equation}

To construct a prior $\Pi$ on $\Theta$ we follow \cite{MNP19} and construct a $\mathfrak{so}(n)$ valued matrix Gaussian random field on $M$ by taking i.i.d.~copies of Gaussian process priors $B_j: j = 1, \dots, \text{\rm dim}(\mathfrak {so}(n))$. For each component $B_j$, we first draw an $\alpha$-regular ($\alpha \in \mathbb N$) planar Whittle-Mat\'ern Gaussian process $\theta'$ on $M$ (cf.~Subsection \ref{wmp}), with law on $C(M)$ denoted by $\Pi'$. Then we choose as prior for $B_j$ the law of $$\theta_j = \frac{\theta_j'}{N^{1/(2\alpha+2)}}, \quad \theta_j' \sim \Pi',$$ the rescaling playing a comparable role to (\ref{schrottprior}). The product prior probability measure on $\Theta = \times_{j=1}^{\text{\rm dim}(\mathfrak{so}(n))} C(M)$ arising from these coordinate distributions will be denoted by $\Pi_N$.  The following theorem holds for arbitrary smooth test functions $\psi : M \to \mathfrak{so}(n)$. As the prior and posterior are measures concentrated in $\mathfrak{so}(n)$ valued matrix fields, it is natural to require the same range constraint on the test function $\psi$ appearing in the dual pairing $\langle \theta, \psi\rangle_{L^2}$. Further remarks paralleling those following Theorem \ref{mainschrott} about the conditions on $\theta_0, \alpha$ apply to the next theorem as well. 

\begin{theorem}\label{mainpnt}
Consider the preceding Gaussian prior $\Pi_N$ with integer $\alpha>8$. Let $\theta$ be drawn from the posterior distribution $\Pi(\cdot|D_N)$ from (\ref{post}) on $\Theta$ arising from observations $D_N$ in model (\ref{model}), where $\mathscr G(\theta)$ is the non-Abelian $X$-ray transform. Denote the posterior mean by $\bar \theta_N = E^\Pi[\theta|D_N]$, and let $\psi \in C^\infty(M, \mathfrak{so}(n))$. Assume $\theta_0 \in C_c^\infty(M, \mathfrak {so}(n))$. Then we have as $N \to \infty$ and in $P_{\theta_0}^N$-probability, the weak convergence
$$\sqrt N \langle \theta - \bar \theta_N, \psi \rangle_{L^2_\zeta(M)} | D_N  \to^d N(0, \|\mathbb I_{\theta_0} (\mathbb I_{\theta_0}^* \mathbb I_{\theta_0})^{-1}\psi\|^2_{L^2_\lambda(\partial_+ SM)})$$ and moreover that $$\sqrt N \langle \bar \theta_N - \theta_0, \psi \rangle_{L^2_\zeta(M)} \to^d N(0, \|\mathbb I_{\theta_0} (\mathbb I_{\theta_0}^* \mathbb I_{\theta_0})^{-1}\psi\|^2_{L^2_\lambda(\partial_+ SM)}).$$  
\end{theorem}

\begin{Remark} {\rm The inversion of $\mathbb I_{\theta_0}^* \mathbb I_{\theta_0}$ as stated in \eqref{eq:inversionfisher} has its own independent interest and it is one of the innovations of the present paper. In general, for geodesic X-ray transforms, the inversion of the Fisher information operator is a delicate problem and its solution depends on the measure chosen on the influx boundary $\partial_{+}SM$ as this choice determines the adjoint $\mathbb{I}_{\theta_0}^*$. There are two commonly used measures
and in both cases the Fisher information operator becomes an elliptic pseudo-differential operator of order $-1$ in the interior of $M$. However, its boundary behaviour is sensitive to the choice of measure and given the non-local nature of $\mathbb I_{\theta_0}^* \mathbb I_{\theta_0}$ one must understand finer mapping properties that include boundary effects. In \cite{MNP} we considered (in the Abelian case) the Fisher information operator for the symplectic measure, i.e. the natural measure on the space of geodesics (also the measure naturally produced by Santal\'o's formula). In this
case, it turns out that $\mathbb I_{\theta_0}^* \mathbb I_{\theta_0}$ extends as a pseudo-differential operator to a slightly larger manifold containing $M$ and one can make use of transmission properties as developed by H\"ormander and Grubb \cite{Gr}. The upshot of this analysis is the need to incorporate a blow up at the boundary of type $\rho^{-1/2}$ (where $\rho$ is distance to the boundary) when proving Bernstein von-Mises theorems.
In contrast, the second choice of measure which is given by the canonical volume form $\lambda$ on the influx boundary -and the one chosen in this paper- exhibits different behaviour and $\mathbb I_{\theta_0}^* \mathbb I_{\theta_0}$ does not extend
as a pseudo-differential operator to any neighbourhood of $M$. To study the behaviour near the boundary
in the case of the disk we take advantage of the recent developments in \cite{M} which deliver non-standard Sobolev scales with suitable  degenerations at the boundary. The inversion in \eqref{eq:inversionfisher} is the first result of its kind and hints at a more general picture valid on any non-trapping manifold with strictly convex boundary and no conjugate points.}
\label{remark:themuaffair}
\end{Remark}

\subsection{Application to uncertainty quantification}\label{UQ}

Bayesian uncertainty quantification for functionals $\langle \theta, \psi \rangle_{L^2_\zeta(\mathcal Z)},$ is based on level $1-\xi$ Bayesian credible sets 
\begin{equation}
C_N = \{v \in \mathbb R: |v- \langle \bar \theta, \psi \rangle_{L_\zeta^2(\mathcal Z)}| \le R_N \}, \quad \Pi(C_N|D_N)=1-\xi, \quad 0<\xi<1,
\end{equation}
where $\bar \theta = E^\Pi[\theta|D_N]$ is the posterior mean. Construction of the interval $C_N$ requires only computation of that mean and of the quantiles $R_N$ of the posterior distribution, both of which can be calculated approximately along a chain of MCMC samples (see also Section \ref{numbojumbo}). In particular the asymptotic variances appearing in Theorems \ref{mainschrott} and \ref{mainpnt} need not be estimated.

Now using Theorems \ref{mainschrott} and \ref{mainpnt} with $0 \neq \psi \in C^\infty$, and arguing as in Remark 2.9 in \cite{MNP} one shows that the credible interval $C_N$ has valid frequentist coverage of the true parameter $\theta_0$ in the sense that, as $N \to \infty$, $$P_{\theta_0}^N(\langle \theta_0, \psi \rangle_{L_\zeta^2(\mathcal Z)} \in C_N) \to 1-\xi,\qquad \sqrt N R_N \to^{P^N_{\theta_0}} Q^{-1}(1-\xi),$$ with $Q(t)=\Pr(|Z|\le t), t \in \mathbb R,$ where $Z$ is the (for $\psi \neq 0$ non-degenerate) limiting normal distribution occurring in Theorems \ref{mainschrott} or \ref{mainpnt}. In particular the diameter of this confidence interval is optimal in an asymptotic minimax sense, see Section \ref{optimality} for details.

\subsection{Numerical illustration}\label{numbojumbo}

We illustrate our theory by numerical experiments for non-Abelian $X$-ray transforms, following the implementation detailed in Section 4 of \cite{MNP19} with $\mathfrak {so}(3)$ replaced by the (isomorphic) $\mathfrak{su}(2)$. We fix the Euclidean metric on the unit disk, and represent the disk as an unstructured mesh with 886 vertices. We choose an $\mathfrak{su}(2)$-valued matrix field $\Phi = a\ \sigma_1 + b\ \sigma_2 + c\ \sigma_3$ as in \cite{MNP19} with $\sigma_1, \sigma_2,\sigma_3$ the Pauli basis matrices, and $a,b,c$ smooth scalar components characterised by their values at the 886 vertices, see Fig. \ref{fig:functions}. 
\begin{figure}[htpb]
    \centering
    \includegraphics[width=\textwidth]{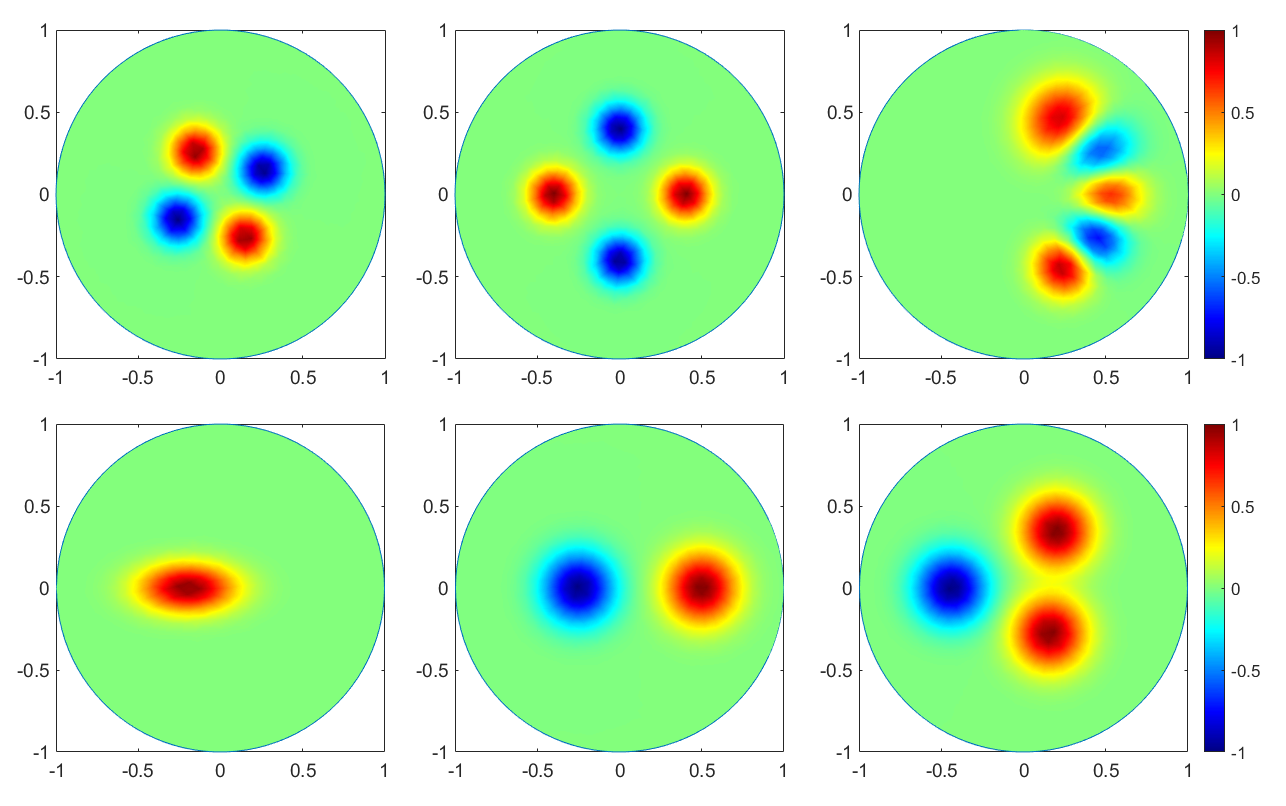}
    \caption{The six functions used (top row: $a$, $b$, $c$; bottom row: $d$, $e$, $f$), all drawn on the same color scale.}
    \label{fig:functions}
\end{figure}

For $N=600$, then $N=1000$, we compute $C_\Phi$ over $N$ geodesics drawn at random, whose entries are then corrupted by additive noise with $\sigma = 0.1$. The prior is set to be of Mat\'ern type with parameters $\alpha=3$ (and length-scale parameter $\ell=0.2$, see \cite{MNP19} for full details).  

The preconditioned Crank-Nicolson (pCN) algorithm is then used to compute $N_s = 10^5$ iterations of a Markov chain $\{\Phi_n\}_{n\le N_s}$ targeting the posterior distribution of $\Phi|D_N$, cf.~Sec.4, \cite{MNP19}. As the purpose here is to explore and display the main features of the posterior, the initial condition is chosen as the ground truth $\Phi$, which shortens the burn-in phase. The sequence $\{\Phi_n\}_{n=1000}^{N_s}$ represents a family of posterior draws.

We fix three $\mathfrak{su}(2)$-valued test functions 
\begin{align*}
    \Psi_1 = d\ \sigma_1 + e\ \sigma_2 + f\ \sigma_3,\quad \Psi_2 = e\ \sigma_1 + f\ \sigma_2 + d\ \sigma_3,\quad \Psi_3 = f\ \sigma_1 + d\ \sigma_2 + e\ \sigma_3,
\end{align*}
where the functions $d,e,f$ appear on Fig. \ref{fig:functions}, and we are interested in the statistics of the smooth aspects $\langle \Phi, \Psi_1\rangle_{L^2}$, $\langle \Phi, \Psi_2\rangle_{L^2}$, $\langle \Phi, \Psi_3\rangle_{L^2}$ of the posterior measure. Figure \ref{fig:histo} displays histograms of the tracked quantities $\{ \langle \Phi_n, \Psi_j\rangle_{L^2}, \ 0\le n\le N_s, \ j\in \{1,2,3\} \}$ along each chain, illustrating both approximate posterior normality and concentration as $N$ increases as predicted by Theorem \ref{mainpnt}. Note that although all three test functions used have the same $L^2(M)$ norm, the predicted asymptotic variances $\|\mathbb I_{\theta_0} (\mathbb I_{\theta_0}^* \mathbb I_{\theta_0})^{-1}\Psi_j\|^2_{L^2_\lambda(\partial_+ SM)}$ should differ for $1\le j\le 3$, as observed on Figure \ref{fig:histo}. The empirical posterior standard deviations further corroborate the frequentist validity of the uncertainty quantification provided by these credible sets established in Section \ref{UQ}.

\begin{figure}[htpb]
    \centering
    \includegraphics[width=\textwidth]{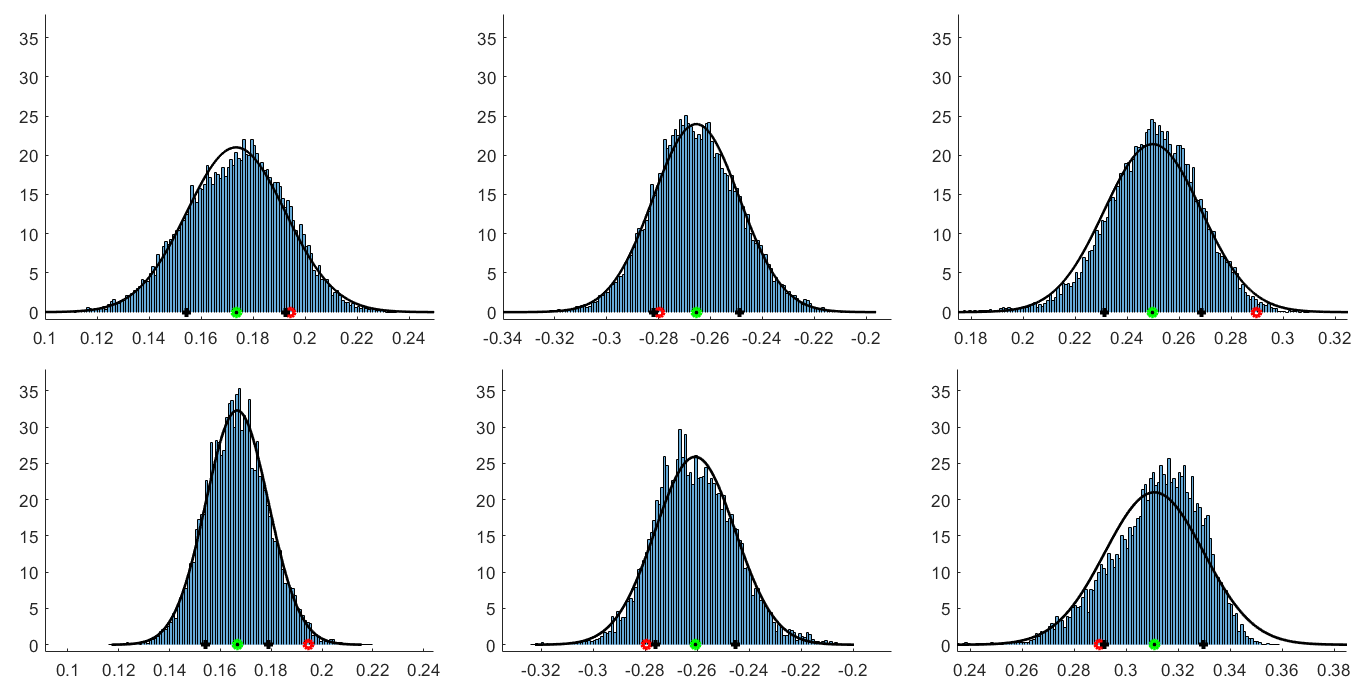}
    \caption{Histograms of the MCMC chains for (left to right): $\langle \Phi_n, \Psi_1\rangle$, $\langle \Phi_n, \Psi_2\rangle$, $\langle \Phi_n, \Psi_3\rangle$, rescaled as probability densities, and Gaussians with empirical mean $\hat{m}$ and variance $\hat{\sigma}^2$ superimposed. Axes scales are uniform across all plots.  Red dot: true value; green dot: mean; black dots: $\hat{m}\pm 1 \hat{\sigma}$. Top to bottom: $N=600$, $N=1000$. The spreads decrease from the top row to the bottom row, most noticeably by $25\%$ on the left plot.}
    \label{fig:histo}
\end{figure}

\section{BvM in regression models with Gaussian process priors} \label{general}

In this section we provide general conditions under which Bernstein-von Mises type approximations can be proved for posterior distributions arising from Gaussian process priors in the general nonlinear regression model (\ref{model}). Theorems \ref{mainschrott} and \ref{mainpnt} will be deduced by verifying these conditions.

\subsection{Analytical hypotheses}

We start with the key hypotheses on the forward map $\mathscr G$ from (\ref{fwdG}). Recall that $\Theta$ is a parameter set arising as a linear subspace of $L^\infty(\mathcal Z, W)$. The first condition concerns the uniform boundedness as well as the global Lipschitz continuity of $\mathscr G$ on $\Theta$ both for $L^2$ and $\|\cdot\|_\infty$ norms. While restrictive, such assumptions are often satisfied due to `compactification' or `energy preservation' properties of the PDEs describing the forward map $\mathscr G$. The second condition requires that $\mathscr G$ is differentiable at the `true value' $\theta_0 \in \Theta$ in a suitable sense.

\begin{Condition}\label{blip}
There exists a fixed constant $C>0$ such that we have $$\|\mathscr G(\theta)\|_\infty \le C, \quad\text{ and }\|\mathscr G(\theta)-\mathscr G(\theta')\|' \le C \|\theta-\theta'\|, \quad \text{ for all } \theta, \theta' \in \Theta,$$ where either $\|\cdot\|'=\|\cdot\|=\|\cdot\|_\infty$ or $\|\cdot\|'=\|\cdot\|_{L^2_\lambda(\mathcal X)}$ and $\|\cdot\|=\|\cdot\|_{L^2_\zeta(\mathcal Z)}$.
\end{Condition}
\begin{Condition}\label{quadratic}
For $\theta_0 \in \Theta$ and any $h \in \Theta$ suppose that as $\|h\|_\infty \to 0$, $$\|\mathscr G(\theta_0 + h) - \mathscr G(\theta_0) - D\mathscr G_{\theta_0}[h]\|_{L^2_\lambda(\mathcal X, V)} \equiv \rho_{\theta_0}[h]=o(\|h\|_\infty)$$ for some operator $$\mathbb I_{\theta_0} \equiv D\mathscr G_{\theta_0}: (\Theta, \langle \cdot, \cdot \rangle_{L^2_\zeta(\mathcal Z, W)}) \to L^2_\lambda(\mathcal X,V)$$ that is a continuous linear map. Moreover we assume that $\mathbb I_{\theta_0}$ is also continuous as a map from $(\Theta, \|\cdot\|_\infty) \to L^\infty(\mathcal X)$. 
\end{Condition}

When considering inference on linear functionals $\langle \psi, \theta \rangle_{L^2_\zeta(\mathcal Z)}$ of $\theta$, the invertibility of  the `information' (or normal) operator $\mathbb I_{\theta_0}^* \mathbb I_{\theta_0}$ induced by $\mathbb I_{0}$ in directions $\psi$ will be required. Here $\mathbb I_{\theta_0}^*: L^2_\lambda(\mathcal X,V) \to \overline{(\Theta, \langle \cdot, \cdot \rangle_{L^2_\zeta(\mathcal Z, W)})}$ denotes the adjoint map of $\mathbb I_{\theta_0}$, and we will employ the following `source type' condition on $\psi$.
\begin{Condition}\label{infokey}
Given $\psi \in \Theta$ and $\mathbb I_{\theta_0}$ from Condition \ref{quadratic}, suppose there exists $\tilde \psi=\tilde \psi_{\theta_0} \in \Theta$ such that  $\mathbb I_{\theta_0}^* \mathbb I_{\theta_0} \tilde \psi = \psi$, that is, $\langle \mathbb I_{\theta_0}^* \mathbb I_{\theta_0} \tilde \psi -\psi, h \rangle_{L^2_\zeta(\mathcal Z, W)}=0$ for all $h \in \Theta$.
\end{Condition}

We now turn to the choice of Gaussian process priors and their reproducing kernel Hilbert spaces (RKHS). As is common in Bayesian non-parametric statistics \cite{vdVvZ08, GvdV17}, we will require assumptions on the small deviation asymptotics of the prior measure $\Pi$. While the displayed probability in the following condition still involves the map $\mathscr G$, one can readily use (\ref{blip}) to simplify the condition to one involving only the prior small probabilities of $\|\theta-\theta_0\|_{L^2_\zeta(\mathcal Z)}$.
\begin{Condition}\label{prior}
The  priors $\Pi=\Pi_N$ consist of Gaussian Borel probability measures on the measurable linear subspace $\Theta$ of $L^\infty(\mathcal Z)$.  The RKHS of $\Pi_N$ is given by the linear subspace $\mathcal H_N$ of $\Theta$, with RKHS inner product $\langle \cdot, \cdot \rangle_{\mathcal H_N}$. Suppose further that $\sup_N E^{\Pi_N}\|\theta\|_{L^2(\mathcal Z)}^4<\infty$ and that for some sequence $\delta_N \to 0$ satisfying $e^{-N \delta_N^2}N^2 \to_{N \to \infty} 0$, some $\bar d>0$ and all $N$ large enough, $$\pi(\delta_N):=\Pi_N (\|\mathscr G(\theta) - \mathscr G(\theta_0)\|_{L_\lambda^2(\mathcal X)} \le \delta_N) \ge \exp\{-\bar dN \delta_N^2\}.$$
\end{Condition}

Note that norms of tight Gaussian probability measures always have all higher moments finite, but we require this bound to be uniform in $N$, hence the condition.

The next condition concerns an initial result about global contraction properties of the posterior measure near the true value $\theta_0 \in \Theta$. For non-linear inverse problems with Gaussian process priors such results have recently been obtained in \cite{MNP19, AN19, GN19}. 
\begin{Condition}\label{DN} For a prior $\Pi_N$ as in Condition \ref{prior}, consider the posterior distribution $\Pi(\cdot|D_N)$ in (\ref{post}) arising from data $D_N$ in the model (\ref{model}). Let the `ground truth' $\theta_0 \in \Theta$ generate data $D_N \sim P_{\theta_0}^N$, and let $(\mathcal R, \|\cdot\|_\mathcal R)$ be a normed linear measurable subspace of $L^\infty(\mathcal Z)$. Assume that as $N \to \infty$ and for real sequences $\bar \delta_N\to 0, M_N \ge 1$, such that $\sqrt N \bar \delta_N \to \infty$,
\begin{equation}\label{linfcon}
\Pi\big(\theta:\|\theta\|_\mathcal R \le M_N, \|\theta-\theta_0\|_\infty\le \bar \delta_N|D_N\big)= 1- o_{P_{\theta_0}^N} (\eta_N).
\end{equation}
Here $\eta_N=e^{-(L+1)N \delta_N^2}$ with $L = 2(2C^2+1)+\bar d$ where $C$ is as in Condition \ref{blip} and $\bar d, \delta_N$ as in Condition \ref{prior}.
\end{Condition}
Following ideas in \cite{MNP19}, verification of Condition \ref{DN} can be based on i) a global \textit{stability (or inverse continuity) estimate} for the map $\theta \mapsto \mathscr G(\theta)$, ii) a `forward' contraction rate result for the posterior law of $\mathscr G(\theta)$ about $\mathscr G(\theta_0)$ and iii) the fact that for \textit{rescaled} Gaussian priors, posteriors automatically concentrate (with high probability) on suitable bounded sets in regularisation spaces $\mathcal R= C^{\beta}$ for some $\beta$ related to the path regularity of the Gaussian process. A general result providing bounds for ii) and iii) can be found in Theorem 14 in the Appendix of \cite{GN19}. Stability estimates i) are more problem specific -- for the applications from Subsections \ref{schrott} and \ref{xraynorm} we rely on Lemma 28 in \cite{NvdGW18} and Corollary 2.3 in \cite{MNP19}, respectively.

\smallskip

The preceding `regularised parameter spaces' 
\begin{equation} \label{keyset}
\Theta_{N}= \{\theta \in \Theta: \|\theta\|_\mathcal R \le M_N, \|\theta-\theta_0\|_\infty\le \bar \delta_N\}
\end{equation}
play a key role in our proofs via the following quantitative condition that allows to control the non-linearity of the likelihood function of the model (\ref{model}), the discretisation errors arising from statistical sampling, and the sensitivity of $\Pi_N$ with respect to small perturbations in $\tilde \psi$-directions. Let $\mathscr J_N$ be an upper bound for the following (`Dudley'-type) integral of the Kolmogorov metric entropy of $\Theta_N$;
\begin{equation}\label{kolmodude}
\mathscr J_N(s, t) \ge \int_0^{s} \sqrt{\log 2 N(\Theta_N, \|\cdot\|_\infty, t \epsilon)} d\epsilon,\quad s,t>0,
\end{equation}
where $N(\Theta_N, \|\cdot\|_\infty, \epsilon)$ are the usual $\epsilon$-covering numbers of the set $\Theta_N$ for the $\|\cdot\|_\infty$-distance (i.e., the minimal number of $\epsilon$-balls for $\|\cdot\|_\infty$ required to cover $\Theta_N$).

\begin{Condition}\label{quant}
Suppose that $\tilde \psi=\tilde \psi_{\theta_0}$ from Condition \ref{infokey} belongs to $\mathcal H_N \cap \mathcal R$ and that it satisfies, for $\delta_N$ from Condition \ref{prior} and $\|\cdot\|_{\mathcal H_N}$ the norm induced by $\langle \cdot, \cdot \rangle_{\mathcal H_N}$,
\begin{equation}\label{rkhsbds}
\lim_{N \to \infty} \delta_N \|\tilde \psi \|_{\mathcal H_N} =0.
\end{equation}
Moreover, for $\bar \delta_N$ as in Condition \ref{DN}, suppose that as $N \to \infty$,
\begin{equation}\label{simple}
\sqrt N \bar \delta_N^2 \mathscr J_N(1, \bar \delta_N^2) \to 0.
\end{equation}
Further for $\sigma_N$ a sequence such that for all $N$ large enough and all $t \in \mathbb R$ fixed, $$\sigma_N \ge \sup_{\theta \in \Theta_N} \rho_{\theta_0}[\theta-\theta_0-(t/\sqrt N) \tilde \psi],$$ assume that as $N \to \infty$,
\begin{equation}\label{mbdc}
\max\Big(N (\sigma^2_N+\sigma_N \bar \delta_N), \sqrt N \mathscr J_N(\sigma_N, 1), \bar \delta_N \sqrt{\log N} \mathscr J^2_N(\sigma_N, 1\big)/\sigma^2_N\Big) \to 0.
\end{equation}
\end{Condition}

 In prototypical situations where $\mathcal R$ equals a fixed ball in a H\"older space $C^\beta(\mathcal Z)$ for a $d$-dimensional domain $\mathcal Z$, and when the approximation in Condition \ref{quadratic} is quadratic ($\rho_{\theta_0}(h)=O(\|h\|_\infty^2)$), it can be shown (see Section \ref{calcul}) that Conditions (\ref{simple}) and (\ref{mbdc}) reduce to the much simpler conditions 
\begin{equation}\label{trirat}
N \bar \delta_N^3 \to 0,\quad \text{ and }~\beta>2d.
\end{equation} 
The requirements on $\alpha$ in Theorems \ref{mainschrott} and \ref{mainpnt} ultimately arise from (\ref{trirat}) for the initial uniform contraction rate $\bar \delta_N$ of the posterior distribution. 

\subsection{Bernstein-von Mises theorems}

Our first main theorem shows that the posterior distribution in our non-linear inverse problem is asymptotically Gaussian when integrated against fixed  test functions $\psi  \in \Theta$, and when centred at
\begin{equation}\label{effcent}
\hat \Psi_N = \langle \theta_0, \psi \rangle_{L_\zeta^2(\mathcal Z, W)} + \frac{1}{N} \sum_{i=1}^N \langle \mathbb I_{\theta_0}\tilde \psi_{\theta_0}(X_i), \varepsilon_i \rangle_V,
\end{equation}
We recall that the next limit is to be understood in the sense of (\ref{weakprob}).

\begin{theorem}\label{main}
Let $\theta \sim \Pi(\cdot|D_N)$ be a posterior draw and assume Conditions \ref{blip} -\ref{quant} are satisfied. Then we have as $N \to \infty$ and in $P_{\theta_0}^N$-probability,
$$\sqrt N \big( \langle \theta, \psi \rangle_{L_\zeta^2(\mathcal Z)}  - \hat \Psi_N \big) | D_N \to^d N(0, \|\mathbb I_{\theta_0} \tilde \psi_{\theta_0}\|_{L^2_\lambda(\mathcal X, V)}^2).$$
\end{theorem}

To use an approximation as the last one for uncertainty quantification (as in Section \ref{UQ}), we need to choose a \textit{feasibly computable centring statistic} instead of the (infeasible) $\hat \Psi_N$. A desirable choice, both for inference and computation purposes via MCMC, is the mean $$ \langle \bar \theta_N, \psi \rangle_{L^2_\zeta(\mathcal Z, W)},~\text{ where }~ \bar \theta_N = E^\Pi[\theta|D_N],$$ of the posterior distribution. [By uniform boundedness of $\mathscr G$, (\ref{post}) and Condition \ref{prior}, the Bochner integral $E^\Pi[\theta|D_N]$ can be shown to exists for any given data vector $D_N$.] 

\begin{theorem}\label{mainmean}
In the setting of Theorem \ref{main}, if $\bar \theta_N = E^\Pi[\theta|D_N]$ denotes the posterior mean, then we have as $N \to \infty$,
$$\sqrt N \langle \theta - \bar \theta_N, \psi \rangle_{L^2_\zeta(\mathcal Z)} | D_N  \to^d N(0, \|\mathbb I_{\theta_0} \tilde \psi_{\theta_0}\|_{L^2_\lambda(\mathcal X,V)}^2)\quad  \text{ in } P_{\theta_0}^N-\text{ probability}.$$ Moreover, as $N \to \infty$, we also have $$\sqrt N \langle \bar \theta_N - \theta_0, \psi \rangle_{L_\zeta^2(\mathcal Z, W)} \to^d N(0, \|\mathbb I_{\theta_0} \tilde \psi_{\theta_0}\|_{L^2_\lambda(\mathcal X, V)}^2).$$
\end{theorem}

\subsection{LAN expansion and asymptotic optimality}\label{optimality}

We finally establish the local asymptotically normal (LAN) expansion of our model and deduce from it the semi-parametric information bound (cf.~\cite{vdV91, vdV98}) for inference on $\langle \theta, \psi \rangle_{L^2_\zeta(\mathcal Z)}$. This implies the optimality of Theorem \ref{mainmean} as long as $\mathbb I_{\theta_0}$ is injective, as is the case in our model examples. [In the Schr\"odinger case, injectivity of $\mathbb I_{\theta_0}$ from (\ref{score}) follows from the uniqueness of solutions of (\ref{schrottiso}) and positivity of $u_{f_0}, \phi'(\theta_0)$, while the injectivity of $\mathbb I_{\theta_0}$ in the setting of Theorem \ref{mainpnt} is proved in \cite{PSUGAFA}.]

 \begin{Proposition} \label{LANexp}
Suppose Conditions \ref{blip} and \ref{quadratic} hold true. Then the log-likelihood ratio process in the model (\ref{model}) satisfies, for every fixed $h \in L^\infty(\mathcal Z)$ and as $N \to \infty$, the asymptotic expansion
\begin{equation}\label{LAN}
\log \frac{dP_{\theta_0+ h/\sqrt N}^N}{dP_{\theta_0}^N}(D_N) = W_N(h) - \frac{1}{2}\|\mathbb I_{\theta_0}[h]\|_{L^2_\lambda(\mathcal X,V)}^2 + o_{P_{\theta_0}^N}(1)
\end{equation}
for random variables 
\begin{equation}\label{WN}
W_N \equiv  \frac{1}{\sqrt N} \sum_{i=1}^N \langle \mathbb I_{\theta_0}h(X_i), \varepsilon_i \rangle_V \to_{N \to \infty}^d N(0, \|\mathbb I_{\theta_0}[h]\|_{L^2_\lambda(\mathcal X,V)}^2).
\end{equation}
Assuming also Condition \ref{infokey} and that $\mathbb I_{\theta_0}: \Theta \to L^2_\lambda(\mathcal X)$ is injective, the semi-parametric information bound for optimal inference on the functional $\langle \theta, \psi \rangle_{L^2_\zeta(\mathcal Z)}$ based on observations $D_N$ is given by 
\begin{equation}\label{crlb}
\|\mathbb I_{\theta_0} \tilde \psi_{\theta_0}\|_{L^2_\lambda(\mathcal X, V)}^2.
\end{equation}
\end{Proposition}
\begin{proof}
An expansion for $\ell_N(\theta_0 +h/\sqrt N)-\ell_N(\theta_0)$ under $P^N_{\theta_0}$ can be obtained as in the proof of Proposition \ref{lanlap} (replacing $\theta, t \tilde \psi_{\theta_0}$ by $\theta_0$ and $h$, respectively), from which the LAN expansion (\ref{LAN}) can be derived without difficulty, and (\ref{WN}) follows directly from the central limit theorem. To find the information lower bound for estimating the functional $\kappa(\theta)=\langle \theta, \psi \rangle_{L_\zeta^2(\mathcal Z)}$ we need to find the Riesz representer $\tilde \kappa$ of $$\kappa :(\Theta, \langle \cdot, \cdot \rangle_{LAN}) \to \mathbb R, \quad\text{where }\quad  \langle \cdot, \cdot \rangle_{LAN} = \langle \mathbb I_{\theta_0}[\cdot], \mathbb I_{\theta_0}[\cdot] \rangle_{L_\lambda^2(\mathcal X)}$$ is a Hilbert space inner product since $\mathbb I_{\theta_0}$ is linear and injective. But Condition \ref{infokey} implies
\begin{equation}
\langle h, \tilde \psi_{\theta_0}  \rangle_{LAN} = \langle h, \mathbb I_{\theta_0}^*\mathbb I_{\theta_0}[\tilde \psi_{\theta_0}] \rangle_{L^2_\zeta(\mathcal Z)} = \langle h, \psi \rangle_{L_\zeta^2(\mathcal Z)} = \kappa(h), \quad h \in \Theta,
\end{equation}
hence $\tilde \kappa =\tilde \psi_{\theta_0} \in \Theta,$ and arguing as in, e.g., Sec.~7.5 in \cite{N1}, the information lower bound is given by $\langle \tilde \kappa, \tilde \kappa \rangle_{LAN}$, as desired.
\end{proof}

\begin{Remark} \label{Optimality} \normalfont
By convergence of moments established in the proof of Theorem \ref{mainmean}, $$N E_{\theta_0}^N \langle \bar \theta_N - \theta_0, \psi \rangle^2_{L^2_\zeta(\mathcal Z)} \to  \|\mathbb I_{\theta_0} \tilde \psi\|_{L^2_\lambda(\mathcal X)}^2$$ as $N \to \infty$, and this is optimal in the minimax sense by the preceding proposition, as then, by the semi-parametric asymptotic minimax theorem \cite{vdV98},
$$\lim_N \inf_{\tilde \psi_N: (V \times \mathcal X)^N \to \mathbb R} ~\sup_{\theta: \|\theta-\theta_0\|_{L^2(\mathcal Z) \le 1/\sqrt N}}NE^N_{\theta} (\tilde \psi_N - \langle \theta, \psi \rangle_{L^2(\mathcal Z)})^2 = \|\mathbb I_{\theta_0} \tilde \psi\|_{L^2_\lambda(\mathcal X)}^2.$$ In particular, no confidence region can have a smaller uniform asymptotic diameter as the one constructed in Section \ref{UQ}.
\end{Remark}

\section{Proofs of Theorems \ref{main} and \ref{mainmean}}

We set $\sigma^2=1$ to simplify notation. We follow ideas from \cite{CN14, CR15, N1, NS19} and  prove a Bernstein-von Mises theorem by proving convergence of the moment generating functions (Laplace transforms) of $\sqrt N \big(\langle \theta, \psi \rangle_{L^2(\mathcal Z)}|D_N - \hat \Psi_N\big)$ with centring as in (\ref{effcent}), which implies weak convergence (in probability), and thus Theorem \ref{main}. This follows by obtaining LAN-type approximations of suitable likelihood-ratios within the support of a suitably `localised' posterior distribution. The stochastic linearisation as well as the discretisation error are controlled by tools from empirical process theory  in Subsection \ref{eprocbd}. That one can centre at the posterior mean instead of $\hat \Psi_N$ (i.e., Theorem \ref{mainmean}) will be proved in Sec.~\ref{tripos}.

\subsection{Localisation of the posterior measure}

We first record a standard stochastic lower bound on the posterior denominator commonly used in Bayesian nonparametric statistics. 
 \begin{Lemma} \label{oldsong} Assume Condition \ref{prior} holds for some $\delta_N, \bar d$ and let $C$ be the constant from Condition \ref{blip}. Then $P_{\theta_0}^N(\mathcal C_N) \to 1$ as $N \to \infty$ where $$\mathcal C_N = \left\{\int_{\Theta} e^{\ell_N(\theta) - \ell_N(\theta_0)}d\Pi(\theta)  \ge e^{-LN \delta^2_N} \right\}, \quad L=2(2C^2+1) + \bar d.$$ Moreover, if $T_N$ is a measurable subset of $\Theta$ such that $$\Pi(T_N) \le e^{-D_0 N \delta_N^2} \quad\text{ for some }\quad D_0>L ,$$ then as $N \to \infty$, $$\Pi(T_N|D_N) =O_{P_{\theta_0}^N}(e^{-(D_0-L) N \delta_N^2})=o_{P_{\theta_0}^N}(1).$$ 
\end{Lemma}
\begin{proof}
We apply Lemma 7.3.2 in \cite{GN16} with $\mathcal P$ there equal to our $\Theta$ and with $p=dP^1_\theta, p_0=dP^1_{\theta_0}$ the model densities for variables $X=(Y_1,X_1)$ on $(V \times \mathcal X)$ generating the i.i.d.~data (\ref{model}) for respective choices of $\theta$, and with probability measure $\nu = \Pi(\cdot \cap B)/\Pi(B)$ on the set $B=B_\epsilon \subset \Theta$ defined in that lemma. If we define sets $$B_{(N)} = \{\theta \in \Theta: \|\mathscr G(\theta)- \mathscr G(\theta_0)\|_{L^2_\lambda(\mathcal X)} \le \delta_N \}$$ then $B_{(N)} \subset B_\epsilon$ for $\epsilon=\sqrt{2C^2+1}\delta_N$ since, noting that $-\log (p/p_0)= \ell_1(\theta_0)-\ell_1(\theta)$ in the notation (\ref{likpo}), standard computations with likelihood ratios (e.g., Lemma 23 in \cite{GN19}, or p.224 in \cite{GvdV17}) and Condition \ref{blip} imply
 $$E_{\theta_0}^1[\ell_1(\theta_0)-\ell_1(\theta)] = \frac{1}{2} \|\mathscr G(\theta)-\mathscr G(\theta_0)\|_{L_\lambda^2(\mathcal X)}^2,$$ $$E_{\theta_0}^1[\ell_1(\theta_0)-\ell_1(\theta)]^2 \le (2C^2+1) \|\mathscr G(\theta)-\mathscr G(\theta_0)\|_{L_\lambda^2(\mathcal X)}^2.$$
We hence obtain from that lemma (with $c=1$), as $N \to \infty$,
 $$P_{\theta_0}^N\left(\int_{B_\epsilon} e^{\ell_N(\theta) - \ell_N(\theta_0)}d\Pi(\theta)  \ge e^{-2(2C^2+1)N \delta^2_N} \Pi(B_\epsilon) \right) \le \frac{1}{(2C^2+1)N\delta^2_N} \to 0.$$ Now the first limit follows since $\Theta \supset B_\epsilon$ and since $\Pi(B_\epsilon) \ge \Pi(B_{(N)}) \ge \pi(\delta_N) \ge e^{-\bar dN\delta_N^2}$ by Condition \ref{prior}. Finally, we see on the event $\mathcal C_N$ that
$$\Pi(T_N|D_N)= \frac{\int_{T_N} e^{\ell_N(\theta)-\ell_N(\theta_0)}d\Pi(\theta)}{\int_{\Theta} e^{\ell_N(\theta)-\ell_N(\theta_0)}d\Pi(\theta)}
\le e^{LN \delta^2_N} Z_N$$ where $Z_N:= \int_{T_N} e^{\ell_N(\theta)-\ell_N(\theta_0)}d\Pi(\theta)=O_{P_{\theta_0}^N}(e^{-D_0N \delta^2_N})$ by Markov's inequality since Fubini's theorem and $E_{\theta_0}^Ne^{\ell_N(\theta)-\ell_N(\theta_0)}=1$ imply
$E_{\theta_0}^N Z_N \le \Pi(T_N)\le e^{-D_0 N \delta_N^2}$.
  \end{proof}
  
 Now since $\tilde \psi$ from Condition \ref{infokey} defines an element of the RKHS $\mathcal H_N$ of $\Pi_N$ by Condition \ref{quant}, if $\theta \sim \Pi_N$ then by properties of RKHS the variable $\langle \theta, \tilde \psi \rangle_{\mathcal H_N}$ has distribution $N(0, \|\tilde \psi\|^2_{\mathcal H_N})$. Hence if we define $$T_N = \Big\{\theta: \frac{|\langle \theta, \tilde \psi \rangle_{\mathcal H_N}|}{\|\tilde \psi\|_{\mathcal H_N}}> \sqrt{2L+1} \sqrt N \delta_N \Big\},$$ then the tail inequality for standard normal random variables implies that $\Pi(T_N)\le e^{-(2L+1) N \delta_N^2}$ and hence the previous lemma applies, so that for $\Theta_N$ from (\ref{keyset}) and
 \begin{equation}\label{key2}
 \bar \Theta_N := \Theta_N \cap T_N^c \quad \text{ we have }\quad \Pi(\bar \Theta^c_N|D_N) =O_{P_{\theta_0}^N} (e^{-(L+1) N \delta_N^2})=o_{P_{\theta_0}^N}(1)
 \end{equation}
 as $N \to \infty$, using also Condition \ref{DN}. In the proofs that follow we consider $\theta \sim \Pi^{\bar \Theta_N}(\cdot|D_N)$ where the posterior (\ref{post}) is taken to arise from prior probability measure $$\Pi^{\bar \Theta_N}\equiv \frac{\Pi(\cdot \cap \bar \Theta_N)}{\Pi(\bar \Theta_N)}$$ equal to $\Pi$ restricted to $\bar \Theta_N$ from (\ref{keyset}) and renormalised. Indeed, Condition \ref{DN} and standard arguments (e.g., p.142 in \cite{vdV98}) then imply, for $\|\cdot\|_{TV}$ the total variation distance on probability measures on $\Theta$, that as $N \to \infty$
\begin{equation}\label{constv}
\|\Pi(\cdot|D_N) - \Pi^{\bar \Theta_N}(\cdot|D_N)\|_{TV} \le 2 \Pi(\bar \Theta^c_N|D_N) \to^{P_{\theta_0}^N} 0,
\end{equation}
and then also $d_{weak}(\Pi(\cdot|D_N), \Pi^{\bar \Theta_N}(\cdot|D_N)) \to^{P_{\theta_0}^N} 0$ for any metric $d_{weak}$ for weak convergence. It hence suffices to prove Theorem \ref{main} for $\Pi^{\bar \Theta_N}(\cdot|D_N)$ instead of $\Pi(\cdot|D_N)$.

\subsection{Uniform LAN approximation of the posterior Laplace transform}

\begin{Proposition}\label{lanlap}
For $\theta, \psi \in \Theta$ and $\tilde \psi = \tilde \psi_{\theta_0}$ from Condition \ref{infokey}, define $$\theta_{(t)}= \theta - \frac{t}{\sqrt N} \tilde \psi_{\theta_0},\quad t \in \mathbb R.$$ Let $\hat \Psi_N$ be as in (\ref{effcent}) and $\bar \Theta_N$ as in (\ref{key2}). Then we have for every fixed $t \in \mathbb R$ and a sequence  $R_N=o_{P_{\theta_0}^N}(1)$ that as $N \to \infty$ 
\begin{equation*}
E^{\Pi^{\bar \Theta_N}} \big[\exp\{t \sqrt N \big(\langle \theta, \psi \rangle_{L^2(\mathcal Z)} - \hat \Psi_N \big)\} | D_N \big] = e^{\frac{t^2}{2} \|\mathbb I_{\theta_0} \tilde \psi\|_{L^2(\mathcal X)}^2} \times \frac{\int_{\bar \Theta_N}  e^{\ell_N(\theta_{(t)})}d\Pi(\theta)}{\int_{\bar \Theta_N} e^{\ell_N(\theta)}d\Pi(\theta)} \times e^{R_N}.
\end{equation*}
\end{Proposition}

\begin{proof}
For $W_N$ as in (\ref{WN}) with $h=\tilde \psi$, the posterior Laplace transform equals
\begin{equation*}
E^{\Pi^{\bar \Theta_N}} \big[e^{t \sqrt N (\langle \theta, \psi \rangle_{L^2 (\mathcal Z)} - \hat \Psi_N )} | D_N \big] = \frac{\int_{\bar \Theta_N}  e^{t \sqrt N \langle \theta-\theta_0, \psi \rangle_{L^2(\mathcal Z)} - tW_N +\ell_N(\theta)-\ell_N(\theta_{(t)})+\ell_N(\theta_{(t)})}d\Pi(\theta)}{\int_{\bar \Theta_N} e^{\ell_N(\theta)}d\Pi(\theta)}
\end{equation*}
The main step in the proof is a uniform in $\theta \in \bar \Theta_N$ perturbation expansion of the log-likelihood ratios under $P_{\theta_0}^N$, recalling (\ref{likpo}) and $\sigma=1$,
\begin{align*}
&\ell_N(\theta)-\ell_N(\theta_{(t)}) \\
& = -\frac{1}{2} \sum_{i=1}^N \big(\|Y_i - \mathscr G(\theta)(X_i)\|_V^2 - \|Y_i - \mathscr G(\theta_{(t)})(X_i)\|_V^2 \big) \\
&= -\frac{1}{2} \sum_{i=1}^N \big(\|\mathscr G(\theta_0)(X_i) - \mathscr G(\theta)(X_i) + \varepsilon_i\|_V^2 - \|\mathscr G(\theta_0)(X_i)- \mathscr G(\theta_{(t)})(X_i)+\varepsilon_i\|_V^2 \big) \\
&= -\sum_{i=1}^N \Big( \langle \varepsilon_i, \mathscr G(\theta_0)(X_i) - \mathscr G(\theta)(X_i) \rangle_V -  \langle \varepsilon_i, \mathscr G(\theta_0)(X_i) - \mathscr G(\theta_{(t)})(X_i) \rangle_V  \Big) \\
&\quad \quad -\frac{1}{2} \sum_{i=1}^N \Big( \|\mathscr G(\theta_0)(X_i) - \mathscr G(\theta)(X_i)\|_V^2 -  \|\mathscr G(\theta_0)(X_i) - \mathscr G(\theta_{(t)})(X_i)\|_V^2  \Big) \equiv I+II.
\end{align*} 
About term I, we `linearise' the map $\mathscr G$ at $\theta_0$ in each inner product to obtain
\begin{align*}
I& =  \sum_{i=1}^N \langle \varepsilon_i, D\mathscr G_{\theta_0}(X_i)[\theta-\theta_{(t)}]\rangle_V  \\
& \quad + \sum_{i=1}^N \langle \varepsilon_i, \mathscr G(\theta)(X_i) - \mathscr G(\theta_0)(X_i) - D\mathscr G_{\theta_0}(X_i)[\theta-\theta_0]\rangle_V \\ 
& \quad - \sum_{i=1}^N \langle \varepsilon_i, \mathscr G(\theta_{(t)})(X_i) - \mathscr G(\theta_0)(X_i) - D\mathscr G_{\theta_0}(X_i)[\theta_{(t)}-\theta_0]\rangle_V \\
& = \frac{t}{\sqrt N}\sum_{i=1}^N \langle \varepsilon_i, D\mathscr G_{\theta_0}(X_i)[\tilde \psi]\rangle_V + R_{(0)}(\theta) - R_{(t)}(\theta) = t W_N +R_{(0)}(\theta) - R_{(t)}(\theta) ,
\end{align*}
noting that $\theta_{(0)}=\theta$ and where the `remainder empirical processes' are given by $$R_{(t)} \equiv \sum_{i=1}^N \langle \varepsilon_i, \mathscr G(\theta_{(t)})(X_i) - \mathscr G(\theta_0)(X_i) - D\mathscr G_{\theta_0}(X_i)[\theta_{(t)}-\theta_0]\rangle_V.$$ We show in Lemma \ref{remgp} below that for all $t \in \mathbb R$ fixed, 
\begin{equation}\label{remgauss}
\sup_{\theta \in \Theta_N} |R_{(t)}(\theta)| = o_{P_{\theta_0}^N}(1)
\end{equation}
so that these terms form a part of the sequence $R_N$.

For term II we write $E^X$ for the expectation under the $X_i$'s only so that
\begin{align*}
&-\frac{1}{2} \sum_{i=1}^N \Big( \|\mathscr G(\theta_0)(X_i) - \mathscr G(\theta)(X_i)\|_V^2 - E^X \|\mathscr G(\theta_0)(X_i) - \mathscr G(\theta)(X_i)\|_V^2 \\
&\quad -  \|\mathscr G(\theta_0)(X_i) - \mathscr G(\theta_{(t)})(X_i)\|_V^2 + E^X \|\mathscr G(\theta_0)(X_i) - \mathscr G(\theta)(X_i)\|_V^2  \Big) \\
&- \frac{N}{2} \|\mathscr G(\theta_0) - \mathscr G(\theta)\|_{L^2(\mathcal X)}^2 + \frac{N}{2} \|\mathscr G(\theta_0) - \mathscr G(\theta_{(t)})\|_{L^2(\mathcal X)}^2.
\end{align*}
The  sums in the first two lines are empirical processes and are shown in Lemma \ref{rememp} below to be $o_{P_{\theta_0}^N}(1)$ uniformly in $\theta \in \Theta_N$ for every fixed $t$, and can thus also be absorbed into $R_N$.

For the terms in the last line of the last display, we can further decompose
\begin{align*} 
\|\mathscr G(\theta_0)-\mathscr G(\theta_{(t)})\|_{L^2(\mathcal X)}^2 &=\|\mathscr G(\theta_{(t)})-\mathscr G(\theta_0)-D\mathscr G_{\theta_0}[\theta_{(t)}-\theta_0]+D\mathscr G_{\theta_0}[\theta_{(t)}-\theta_0]\|_{L^2(\mathcal X)}^2 \\
&= \|D\mathscr G_{\theta_0}[\theta_{(t)}-\theta_0]\|^2_{L^2(\mathcal X)} \\
& ~~~~+ 2 \langle D\mathscr G_{\theta_0}[\theta_{(t)}-\theta_0], \mathscr G(\theta_{(t)}) - \mathscr G(\theta_0) - D\mathscr G_{\theta_0}[\theta_{(t)}-\theta_0]\rangle_{L^2(\mathcal X)} \\
&~~~~+ \| \mathscr G(\theta_{(t)}) - \mathscr G(\theta_0) - D\mathscr G_{\theta_0}[\theta_{(t)}-\theta_0]\|_{L^2(\mathcal X)}^2  
\end{align*}
including also the case $\theta=\theta_{(0)}$ by convention for $t=0$. Now using Conditions \ref{quadratic}, \ref{quant} and the Cauchy-Schwarz inequality the last two remainder terms are bounded by a constant multiple of
\begin{align*}
 \sup_{\theta \in \Theta_N} \big[\rho_{\theta_0}^2(\theta_{(t)}-\theta_0) + \|\theta_{(t)}-\theta_0\|_{L^2} \rho_{\theta_0}(\theta_{(t)}-\theta_0)  \big] \lesssim \sigma_N^2 + \sigma_N \bar \delta_N = o(1/N).
\end{align*}
The remaining terms in the expansion are 
\begin{align*}
&\frac{N}{2}\big(\big\|D\mathscr G_{\theta_0}[\theta-\theta_0 -\frac{t}{\sqrt N} \tilde \psi]\big\|_{L^2(\mathcal X, V)}^2 - \big\|D\mathscr G_{\theta_0}[\theta-\theta_0]\big\|_{L^2(\mathcal X, V)}^2  \big) \\
&=- t \sqrt N \langle D\mathscr G_{\theta_0}[\theta-\theta_0], D\mathscr G_{\theta_0}[\tilde \psi]\rangle_{L^2(\mathcal X, V)} + \frac{t^2}{2} \|D\mathscr G_{\theta_0}[\tilde \psi]\big\|_{L^2(\mathcal X, V)}^2 \\
&= - t\sqrt N \langle \theta-\theta_0, \mathbb I_{\theta_0}^* \mathbb I_{\theta_0} \tilde \psi \rangle_{L^2(\mathcal Z, W)} + \frac{t^2}{2} \|\mathbb I_{\theta_0}[\tilde \psi]\big\|_{L^2(\mathcal X, V)}^2
\end{align*}
which, combined with Condition \ref{infokey}, the bounds from term $I$ and the identity in the first display in this proof,  implies the result.
\end{proof}

\subsection{Stochastic bounds on remainder terms and discretisation error} \label{eprocbd}

The following two key lemmas use tools from infinite-dimensional probability to bound the collections of empirical processes appearing as remainder terms in the proof of Proposition \ref{lanlap}. While that proposition considers localisation to the sets $\bar \Theta_N$, the following bounds actually hold uniformly in the larger classes $\Theta_N$ from (\ref{keyset}).

\begin{Lemma}\label{remgp}
We have (\ref{remgauss}).
\end{Lemma}
\begin{proof}
For $t$ fixed define new functions $g_\theta: \mathcal X \to V$ as $$g_\theta = \mathscr G(\theta_{(t)})(\cdot) - \mathscr G(\theta_0)(\cdot) - D\mathscr G_{\theta_0}(\cdot)[\theta_{(t)}-\theta_0].$$ Then the remainder term from (\ref{remgauss}), viewed as a stochastic process indexed by $\theta \in\Theta_N$, equals a centred (since $E\varepsilon_i=0$) \textit{empirical process} for the jointly i.i.d.~variables $(X_i, \varepsilon_i)$ of the form
\begin{equation*} 
|R_{(t)}(\theta)| \equiv \big|\sum_{i=1}^N \langle \varepsilon_i, g_\theta(X_i)\rangle_V \big| \equiv \big|\sum_{j=1}^{p_V} \sum_{i=1}^N \varepsilon_{i,j} g_{\theta,j}(X_i)\big| \le \sum_{j=1}^{p_V} \big|\sum_{i=1}^N \varepsilon_{i,j} g_{\theta,j}(X_i)\big|.
\end{equation*}
Here $g_{\theta, j}$ are the entries of the vector field $g_\theta \in V$, and the $\varepsilon_{i,j}$ are all i.i.d.~$N(0,1)$ variables. We will now bound the supremum over $\Theta_N$ of the each of the last $p_V$ summands by using a moment inequality for the empirical process $\{\sum_{i=1}^N f_\theta(Z_i): f \in \mathcal F\}$ where, for every $1 \le j \le p$ fixed (and with $e$ denoting a real variable in this proof in slight abuse of notation), $$f_\theta \in \mathcal F \equiv \mathcal F_j =\{f_\theta(z) = e g_{\theta, j}(x): \theta \in \Theta_N\}, \quad z=(e,x) \in \mathbb R \times \mathcal X,$$ and $Z_1, \dots, Z_N$ are i.i.d.~copies of the variables $Z=(\varepsilon, X) \sim N(0,1) \times \lambda =P$. 

We will apply Theorem 3.5.4 in \cite{GN16} but to do so need to calculate some preliminary bounds: First, by independence of $X, \varepsilon$, the `weak' variances of $\mathcal F$ are of order $$\sup_{\theta \in \Theta_N}Ef^2_\theta(Z) = \sup_{\theta \in \Theta_N} Eg_\theta^2(X) \le \sup_{\theta \in \Theta_N}  \rho^2_{\theta_0}(\theta_{(t)}-\theta_0) \leq \sigma_N^2$$ by Conditions \ref{quadratic} and \ref{quant}. Next, by Condition \ref{blip}, the $L^\infty$-norm mapping properties of $D\mathscr G_{\theta_0}$ (Condition \ref{quadratic}) and the definition of $\Theta_N$ we have $$\sup_{\theta \in \Theta_N}\|g_{\theta,j}\|_\infty \lesssim \|\theta_{(t)} - \theta_0\|_\infty \lesssim \bar \delta_N (1 + \|\tilde \psi\|_\infty) \lesssim \bar \delta_N.$$ As a consequence the preceding empirical process has point-wise envelopes $$\sup_{\theta \in \Theta_N}|f_\theta(e, x)| \lesssim |e| \bar \delta_N \equiv F_N(e, x) \quad \forall (e,x) \in \mathbb R \times \mathcal X,$$ in particular $F_N>0$ $P$-a.s.~and $$\|F\|_{L^2(P)}^2 := \int_{\mathbb R \times \mathcal X} F^2_N(z)dP(z) \lesssim \bar \delta_N^2, \quad \|F\|_{L^2(Q)}^2 :=  \int_{\mathbb R \times \mathcal X} F^2_N(z)dQ(z) = \bar\delta^2_N s_Q^2,$$ where, for any (discrete, finitely supported) probability measure $Q$ on $\mathbb R \times \mathcal X$, we have set $s_Q^2:= \int_{\mathbb R \times \mathcal X} e^2dQ(e,x)$. Finally, we have again from Condition \ref{blip} and \ref{quadratic}, and for any $\theta, \theta' \in \Theta$ and some fixed constant $c_0$ that
\begin{align*}
\|f_\theta-f_{\theta'}\|_{L^2(Q)}&:=  \sqrt{\int_\mathbb R  \int_\mathcal  X e^2 (g_{\theta, j}(x)-g_{\theta',j}(x))^2 dQ(e,x)} \\
&\le s_Q \|g_{\theta, j} - g_{\theta', j}\|_\infty \\
&\le s_Q (\|\mathscr G(\theta_{(t)}) - \mathscr G(\theta_{(t)}')\|_\infty + \|\mathbb I_{\theta_0}[\theta_{(t)}-\theta_{(t)}']\|_\infty) \\
& \leq c_0 \|F_N\|_{L^2(Q)} \|\theta-\theta'\|_\infty/\bar \delta_N.
\end{align*}
We conclude that any $\bar \delta_N \epsilon/c_0$-covering of $\Theta_N$ for the norm $\|\cdot\|_\infty$  induces a $\|F_N\|_{L^2(Q)}\epsilon$-covering of $\mathcal F$ for the $L^2(Q)$ norm, and so $J(\mathcal F, F, s)$ in (3.169) in \cite{GN16} is bounded by a constant multiple of our $\mathscr J_N(s, \bar \delta_N)$ (using also Lemma 3.5.3a in \cite{GN16}). With these preparations, we can now apply Theorem 3.5.4 in \cite{GN16} where for our choice of envelope $F_N$ we can take $\|U\|_{L^2(P)}$ in that theorem bounded by a constant multiple of $\sqrt {\log N} \bar \delta_N$ (using independence of $X,\varepsilon$ and also Lemma 2.3.3 in \cite{GN16}). The upper bound (3.171) in \cite{GN16} then implies that $$E\sup_{\theta \in \Theta_N}\big|\sum_{i=1}^N f_\theta(Z_i)\big| \lesssim \sqrt N \max\left[\bar \delta_N \mathscr J_N(\sigma_N/\bar \delta_N, \bar \delta_N), \frac{\sqrt {\log N} \bar \delta_N^3  \mathscr J^2_N(\sigma_N/\bar \delta_N, \bar \delta_N)}{\sqrt N \sigma_N^2} \right]$$ which in turn,  using the substitution $\bar \delta_N \epsilon = \rho$ in (\ref{kolmodude}), is bounded by a constant multiple of the maximum of the second and third terms appearing in (\ref{mbdc}). Hence the remainder terms from (\ref{remgauss}) converge to zero in expectation, and then also in probability (by Markov's inequality). [Let us finally note that, strictly speaking, the application of Theorem 3.5.4 in \cite{GN16} requires $0 \in \mathcal F$ and $\mathcal F$ countable: If $\|\theta_0\|_{\mathcal R} < M_N$ then $g_\theta=0$ for $\theta=\theta_0 -(t/\sqrt N)\tilde \psi \in \Theta_N$ and $N$ large enough, so $0 \in \mathcal F$. Otherwise we can recenter $f_\theta$ at $f_{\theta_*}$ for some arbitrary $\theta_*$ and use a standard (one-dimensional) moment bound for $E|\sum_{i=1}^Nf_{\theta_*}(Z_i)| \le \sqrt N \sigma_N \to 0$. One then applies the previous argument to the class $\mathcal F - f_{\theta_*}$, so that the same overall bound holds true also in this case. Finally, by continuity of $\theta \mapsto g_{\theta,j}$ on the totally bounded set $\Theta_N$, the supremum of the empirical process can be realised over a countable dense subset of $\Theta_N$, so the assumption that $\mathcal F$ be countable can be met, too.]  
\end{proof}

\begin{Lemma} \label{rememp}
We have for any $t \in \mathbb R$ that $$\sup_{\theta \in \Theta_N} \Big|\sum_{i=1}^N \big( \|\mathscr G(\theta_0)(X_i) - \mathscr G(\theta_{(t)})(X_i)\|_V^2 - E^X \|\mathscr G(\theta_0)(X_i) - \mathscr G(\theta)(X_i)\|_V^2  \Big)\Big| = o_{P_{\theta_0}^N}(1)$$
\end{Lemma}
\begin{proof}
We will obtain a bound for the supremum of the empirical process $\{\sum_{i=1}^N (f(X_i)-Ef(X_i)): f \in \mathcal F\}$, this time with indexing class $$\mathcal F = \{f_\theta=\|\mathscr G(\theta_0)(\cdot) - \mathscr G(\theta_{(t)})(\cdot)\|_V^2: \theta \in \Theta_N\}.$$ Using Condition \ref{blip}, the envelopes of $\mathcal F$ can be taken to be 
\begin{align*}
\sup_{\theta \in \Theta_N} \|\mathscr G(\theta_0) - \mathscr G(\theta_{(t)})\|_\infty^2 & \lesssim \sup_{\theta \in \Theta_N} \|\theta_0 - \theta_{(t)}\|_\infty^2 \lesssim \bar \delta_N^2 \equiv F,
\end{align*}
and we also have, since $\|\mathscr G(\theta)\|_\infty \le C$ by Condition \ref{blip}, that $$\|f_\theta - f_{\theta'}\|_\infty \lesssim \|\theta-\theta'\|_\infty \quad \forall \theta, \theta' \in \Theta_N.$$ This implies, similar to the proof in the previous lemma, that a $c_0 \bar \delta_N^2 \epsilon$-covering of $\Theta_N$ for the $\|\cdot\|_\infty$-norm (and $c_0$ a small but fixed constant) induces a $\|F\|_{L^2(Q)}\epsilon$-covering of $\mathcal F$ for the $L^2(Q)$-norm ($Q$ any probability measure), and that the functional $J(\mathcal F, F, s)$ in (3.169) in \cite{GN16} is bounded by a constant multiple of our $\mathscr J(s, \bar\delta_N^2)$. The convergence to zero required in the lemma now follows from Theorem 3.5.4 in \cite{GN16}, in fact Remark 3.5.5 after it, the requirement (\ref{simple}) from Condition \ref{quant}, and Markov's inequality.
\end{proof}

\subsection{Gaussian change of variables}

We now control the ratio of Gaussian integrals appearing in Proposition \ref{lanlap}.

\begin{Proposition} \label{cameron} As $N \to \infty$ we have for any fixed $t \in \mathbb R$ that
$$ \frac{\int_{\bar \Theta_N}  e^{\ell_N(\theta_{(t)})}d\Pi(\theta)}{\int_{\bar \Theta_N} e^{\ell_N(\theta)}d\Pi(\theta)}  \to^{P_{\theta_0}^N} 1.$$
\end{Proposition}
\begin{proof}
If we denote by $\Pi_t$ the Gaussian law of $\theta_{(t)}=\theta - (t/\sqrt N) \tilde \psi$, then the Cameron-Martin theorem (e.g., Theorem 2.6.13 in \cite{GN16}) provides the formula for the Radon-Nikodym density of $$\frac{d\Pi_t}{d\Pi}(\theta) =\exp\Big\{\frac{t}{\sqrt N} \langle \theta, \tilde \psi \rangle_{\mathcal H_N} - \frac{t^2}{2N} \|\tilde \psi\|_{\mathcal H_N}^2 \Big\}, \quad \theta \sim \Pi, \quad \tilde \psi \in \mathcal H_N.$$  The ratio in the proposition thus equals
$$ \frac{e^{ - \frac{t^2}{2N} \|\tilde \psi\|_{\mathcal H_N}^2} \int_{\bar \Theta_{N,t}}  e^{\ell_N(\vartheta)} e^{\frac{t}{\sqrt N} \langle \vartheta, \tilde \psi \rangle_{\mathcal H_N}} d\Pi(\vartheta)}{\int_{\bar \Theta_N} e^{\ell_N(\theta)}d\Pi(\theta)},\quad\text{ where }\quad \bar \Theta_{N,t}=\{\theta_{(t)}: \theta \in \bar \Theta_N\}.$$
Uniformly in $\theta \in T_N^c \subset \bar \Theta_N$ from (\ref{key2}) we have as $N \to \infty$ that $|(t/\sqrt N) \langle \theta, \tilde \psi \rangle_{\mathcal H_N}| \lesssim \delta_N \|\tilde \psi\|_{\mathcal H_N} \to 0$ by the requirement (\ref{rkhsbds}) in Condition \ref{quant}, which also implies that $(t^2/ N) \|\tilde \psi\|^2_{\mathcal H_N}=o(1)$ since $1/\sqrt N = o(\delta_N)$. Now since
$$\frac{|t|}{\sqrt N}\sup_{\vartheta \in \bar \Theta_{N,t}} |\langle \vartheta, \tilde \psi \rangle_{\mathcal H_N}| \le \frac{|t|}{\sqrt N}\sup_{\theta \in T_N^c} |\langle \theta, \tilde \psi \rangle_{\mathcal H_N}| + \frac{t^2}{N}\|\tilde \psi\|_{\mathcal H_N}^2 $$ we deduce from what precedes that the last ratio of integrals equals
$$ e^{o(1)} \times \frac{\int_{\bar \Theta_{N,t}}  e^{\ell_N(\vartheta)} d\Pi(\vartheta)}{\int_{\bar \Theta_N} e^{\ell_N(\theta)}d\Pi(\theta)} = e^{o(1)} \times \frac{\Pi(\bar \Theta_{N,t}|D_N)}{\Pi(\bar \Theta_{N}|D_N)}.$$
The denominator converges to 1 in $P_{\theta_0}^N$-probability by (\ref{key2}), and so does then the numerator, using again (\ref{key2}) and that $t\|\tilde \psi\|_\infty/\sqrt N = o(\bar \delta_N)$ and  $t\|\tilde \psi\|_\mathcal R/\sqrt N = o(M_N)$ under the maintained assumptions. 
\end{proof}

Combining Propositions \ref{lanlap} and \ref{cameron} we have shown that  for all $t \in \mathbb R$, as $N \to \infty$,
\begin{equation}\label{mgfpre}
E^{\Pi^{\bar \Theta_N}} \big[\exp\{t \sqrt N \big(\langle \theta, \psi \rangle_{L_\zeta^2(\mathcal Z)} - \hat \Psi_N \big)\} | D_N \big] \to \exp \Big\{\frac{t^2}{2} \|\mathbb I_{\theta_0} \tilde \psi\|_{L_\lambda^2(\mathcal X)}^2 \Big\}
\end{equation}
in $P^N_{\theta_0}$-probability, and therefore, using also (\ref{constv}), for $\theta \sim \Pi(\cdot|D_N)$,
\begin{equation} \label{prelim}
\sqrt N  \big(\langle \theta|D_N, \psi \rangle_{L_\zeta^2(\mathcal Z)} - \hat \Psi_N \big) \to^d N(0, \|\mathbb I_{\theta_0} \tilde \psi\|_{L^2_\lambda(\mathcal X)}^2)
\end{equation}
 by the in $P_{\theta_0}^N$-probability version of the usual implication that convergence of Laplace transforms implies convergence in distribution (see the appendices of \cite{N1} or \cite{CR15}). This completes the proof of Theorem \ref{main}.
  
 \subsection{Convergence of the posterior mean}\label{tripos}
 
 The proof combines ideas from \cite{CN13, N1, MNP, MNP19}. The key lemma is the following stochastic bound on the posterior second moments.
 
 \begin{Lemma}\label{UIdom}
 Under the hypotheses of Theorem \ref{mainmean} we have
 $$NE^\Pi\big[ (\langle \theta, \psi \rangle_{L_\zeta^2(\mathcal Z)} - \hat \Psi_N )^2 |D_N\big] = O_{P_{\theta_0}^N}(1)$$
 \end{Lemma}
 \begin{proof}
The left hand side in the last display is bounded by
$$2NE^\Pi\big[ \langle \theta - \theta_0, \psi \rangle_{L_\zeta^2(\mathcal Z)}^2|D_N] + 2 N(\hat \Psi_N - \langle \theta_0, \psi \rangle_{L_\zeta^2(\mathcal Z)})^2$$ and in view of (\ref{effcent}), the second term in the last decomposition is bounded in $P^N_{\theta_0}$-probability by the central limit theorem applied to $W_N$ from (\ref{WN}) with $h = \tilde \psi_{\theta_0}$ (one also applies the continuous mapping theorem for $x \mapsto x^2$ and Prohorov's theorem to deduce from convergence in distribution of $NW_N^2$ that it is uniformly tight.) 

It hence remains to bound the first term in the last decomposition. Define $A_N = \{\|\theta-\theta_0\|_\infty \le \bar \delta_N\} \subset \Theta$  and write the first quantity in the last display as (two times)
\begin{equation} \label{tightterms}
 N E^\Pi\big[ \langle \theta-\theta_0, \psi \rangle_{L_\zeta^2(\mathcal Z)}^2 1_{A_N}  |D_N\big] + N E^\Pi\big[ \langle \theta-\theta_0, \psi \rangle_{L_\zeta^2(\mathcal Z)}^2 1_{A^c_N}  |D_N\big] = I+II.
\end{equation}
To deal with term II, we apply the Cauchy-Schwarz inequality to obtain the bound 
\begin{equation*}
N \sqrt{E^\Pi\big[ \langle \theta -\theta_0, \psi \rangle_{L^2(\mathcal Z)}^4|D_N]} \sqrt{\Pi(\|\theta-\theta_0\|_\infty > \bar \delta_N|D_N)}
\end{equation*}
and we now show that this term is bounded in $P_{\theta_0}^N$-probability: Using Condition \ref{DN}, Lemma \ref{oldsong}, Markov's inequality and $E_{\theta_0}^N e^{\ell_N(\theta)-\ell_N(\theta_0)}=1$ we indeed have
\begin{align*}
& P_{\theta_0}^N\Big(E^\Pi\big[ \langle \theta -\theta_0, \psi \rangle^4|D_N] \Pi(\|\theta-\theta_0\|_\infty > \bar \delta_N|D_N) > N^{-2} \Big) \\
& \le P_{\theta_0}^N\Big(E^\Pi\big[ \langle \theta -\theta_0, \psi \rangle^4|D_N] e^{-(L+1)N \delta_N^2}> N^{-2} \Big) +o(1) \\
& = P_{\theta_0}^N\left(\frac{\int_\Theta \langle \theta -\theta_0, \psi \rangle^4 e^{\ell_N(\theta)-\ell_N(\theta_0)} d\Pi(\theta) }{\int_\Theta  e^{\ell_N(\theta)-\ell_N(\theta_0)} d\Pi(\theta)}>e^{(L+1)N \delta_N^2} N^{-2}, \mathcal C_N \right) +o(1) \\
& \le  \|\psi\|_{L^2(\mathcal Z)}^4 e^{-N \delta_N^2} N^2 \int_\Theta \|\theta-\theta_0\|_{L^2(\mathcal Z)}^4E_{\theta_0}^N e^{\ell_N(\theta)-\ell_N(\theta_0)} d\Pi(\theta)  + o(1) \\
& \lesssim N^2 e^{-N\delta_N^2} +o(1) \to 0
\end{align*}
as $N \to \infty$, by hypothesis on $\delta_N, \Pi_N$. Collecting what precedes implies that the term $II$ in (\ref{tightterms}) is indeed $O_{P_{\theta_0}}^N(1)$. 

The next step is to bound the term $I$ in (\ref{tightterms}). Recalling that $\Pi^{\bar \Theta_N}[\cdot|D_N]$ denotes the posterior distribution arising from prior restricted and renormalised to $\bar \Theta_N$, we decompose
\begin{align*}
&N E^{\Pi}\big[ \langle \theta-\theta_0, \psi \rangle_{L^2(\mathcal Z)}^2 1_{A_N}|D_N\big]  =  N E^{\Pi^{\bar \Theta_N}}\big[ \langle \theta-\theta_0, \psi \rangle_{L^2(\mathcal Z)}^2 1_{A_N} |D_N\big] \\
&~~~~~+ N E^{\Pi}\big[ \langle \theta-\theta_0, \psi \rangle_{L^2(\mathcal Z)}^2 1_{A_N} |D_N\big]  - N E^{\Pi^{\bar \Theta_N}}\big[ \langle \theta-\theta_0, \psi \rangle_{L^2(\mathcal Z)}^2 1_{A_N} |D_N\big] = A+B.
\end{align*}
For term $A$, using $x^2 \le 2e^x, x \ge 0$, the definition of $\hat \Psi_N$ from (\ref{effcent}) and $W_N=O_{P_{\theta_0}^N}(1)$ with $h = \tilde \psi_{\theta_0}$ from (\ref{WN}), the limit (\ref{mgfpre}) at $t=1$ implies that for all $N$ large enough and some $r_N=o_{P_{\theta_0}^N}(1)$,
$$A  \le  2e^{W_N+r_N} e^{\frac{1}{2} \|\mathbb I_{\theta_0} \tilde \psi\|_{L^2(\mathcal X, V)}^2},$$ and hence this term is stochastically bounded.

Finally, by definition of the events $A_N$, the term $|B|$ can be written as
\begin{align*}
&N \Big|\int_{A_N} \langle \theta-\theta_0, \psi \rangle_{L^2(\mathcal Z)}^2 [d\Pi(\theta|D_N) - d\Pi^{\bar \Theta_N}(\theta|D_N)] \Big|\\
& \le N \bar \delta^2_N \|\psi\|^2_{L^1(\mathcal Z)} \|\Pi(\cdot|D_N) - \Pi^{\bar \Theta_N}(\cdot|D_N)\|_{TV} \\
& \lesssim N \bar \delta^2_N \Pi(\bar\Theta^c_N|D_N) \lesssim N \bar \delta^2_N O_{P_{\theta_0}^N}(e^{-(L+1)N \delta_N^2}) = o_{P_{\theta_0}^N}(1),
\end{align*}
where we have used (\ref{constv}) and (\ref{key2}), completing the proof of the lemma.
\end{proof}
 
Now to prove the theorem note that by (\ref{prelim}) and (\ref{weakprob}) we have for $$Z_n|D_N \equiv \sqrt N  (\langle \theta, \psi \rangle_{L^2_\zeta(\mathcal Z)}- \hat \Psi_N) |D_N, \quad Z \sim N(0, \|\mathbb I_{\theta_0} \tilde \psi\|_{L^2_\lambda(\mathcal X)}^2)$$ and $d_{weak}$ any metric for weak convergence of laws $\mathcal L(\cdot)$ on $\mathbb R$,
\begin{equation}\label{wlim}
d_{weak}\left( \mathcal L\big(Z_N |D_N\big), \mathcal L(Z)) \right) \to_{N \to \infty}^{P_{\theta_0}^N} 0.
\end{equation}
The idea of the proof of follow is that the previous lemma implies (by uniform integrability) convergence of moments in the last limit (\ref{wlim}), and thus that, since $EZ=0$, the posterior mean equals $\hat \Psi_N$ up to a stochastic term of order $o(1/\sqrt N)$. However, as the probability measures $\mathcal L(Z_N|D_N)$ to which this argument is applied are \textit{random} via the data $D_N$, the proof requires some care. We will employ a contradiction argument: To prove Theorem \ref{mainmean}, it suffices by Theorem \ref{main}, Slutsky's lemma and (\ref{WN}) with $h= \tilde \psi_{\theta_0}$ to prove that as $N \to \infty$,
\begin{equation}\label{whatwewant}
\sqrt N\big(\langle E^\Pi[\theta|D_N], \psi \rangle_{L^2_\zeta(\mathcal Z)}-\hat \Psi_N \big) \to^{\Pr} 0,
\end{equation}
 where we write $\Pr$ for the probability measure $P^\mathbb N_{\theta_0}$ on the underlying measurable space $(\Omega, \mathcal S):=((V \times \mathcal X)^\mathbb N, \mathcal S)$ supporting all data variables $(D_N, N \in \mathbb N)$. Suppose the last limit does not hold true. Then there exists $\Omega' \in \mathcal S$ of positive probability $\Pr(\Omega')>\tau$ and $\zeta'>0$ such that along a subsequence of $N$ (still denoted by $N$) we have 
\begin{equation}\label{conti}
|\sqrt N\big(\langle E^\Pi[\theta|D_N(\omega)], \psi \rangle_{L_\zeta^2(\mathcal Z)}-\hat \Psi_N(\omega) \big)| \ge \zeta' >0 \quad\text{for } \omega \in \Omega'. 
\end{equation}
Now since convergence in $\Pr$-probability implies $\Pr$-almost sure convergence along a subsequence, we can extract a further subsequence of $N$ such that (\ref{wlim}) holds almost surely, that is, on an event $\Omega_0 \subset \Omega$ such that $\Pr(\Omega_0)=1$. For each fixed $\omega \in \Omega_0$ we can use the Skorohod imbedding (Theorem 11.7.2 in \cite{D02}) to construct (if necessary on a new probability space) new real random variables $\tilde Z_N, \tilde Z$ such that their laws satisfy $$\mathcal L(\tilde Z_N) = \mathcal L \big(Z_N|D_N(\omega) \big), \quad \mathcal L(\tilde Z)= \mathcal L(Z), \quad \tilde Z_N \to_{N \to \infty}^{a.s.} \tilde Z,$$ and we also know by Lemma \ref{UIdom} that $E\tilde Z_N^2=E[Z_N^2|D_N(\omega)]=O(1)$ for all $\omega \in \Omega_0' \subset \Omega_0$ of probability $\Pr(\Omega_0')>1-\tau$ as close to one as desired. But this implies that the $
(\tilde Z_N: N \in \mathbb N)$ are uniformly integrable real random variables so that almost sure convergence implies convergence of  first moments (\cite{D02}, Theorem 10.3.6), that is $$E|Z_n|D_N(\omega) - Z| =E|\tilde Z_N - \tilde Z| \to_{N \to \infty} 0$$ for all $\omega \in \Omega_0'$. In particular then, using also Fubini's theorem,
\begin{equation}\label{momlim}
\sqrt N  \big(\langle E^\Pi [\theta|D_N(\omega)], \psi \rangle_{L^2_\zeta} - \hat \Psi_N(\omega) \big) =E^\Pi \big[\sqrt N  \big(\langle \theta, \psi \rangle - \hat \Psi_N \big) |D_N(\omega)\big] \to EZ=0
\end{equation} 
for $\omega \in \Omega_0'$. But if the last limit holds for all $\omega \in \Omega_0'$ with probability $\Pr(\Omega_0')>1-\tau$ we have a contradiction to (\ref{conti}) (as then $\Pr(\Omega) \ge \Pr(\Omega') + \Pr(\Omega_0') > 1 -\tau +\tau =1$), completing the proof of (\ref{whatwewant}) and thus of the theorem.

\section{Proofs for non-Abelian $X$-ray and Schr\"odinger equation}\label{examples}

The proofs proceed by verifying the hypotheses of Theorems \ref{main} and \ref{mainmean}.

\subsection{Proof of Theorem \ref{mainschrott}}

We follow ideas laid out in \cite{N1} for a more restrictive class of priors and a simpler noise model. In particular in our setting $\Theta$ is unbounded and we therefore need to explicitly track the growth of various constants in the PDE estimates used in \cite{N1}. These have been obtained in the recent article \cite{NvdGW18} in the study of a related problem, and we will refer repeatedly to \cite{NvdGW18} in the proofs that follow.

\smallskip

A key role is played by the linear $L^2(\mathcal X)$-self-adjoint `inverse Schr\"odinger' integral operator $\mathbb V_f, f>0$ smooth, furnishing unique solutions $u_{f,\psi}=\mathbb V_f[\psi]$ of the PDE 
\begin{equation}\label{schrottiso}
\mathbb S_f(u_{f,\psi}) = \psi\quad  \text{ on } \mathcal X, \quad\text{ s.t. }\quad u_{f,\psi}=0\quad \text{ on } \partial \mathcal X, \quad\text{ for all }\psi \in C(\mathcal X),
\end{equation}
where we recall the Schr\"odinger operator $\mathbb S_f(h) = \frac{\Delta}{2}h - fh$. We also have for $\psi \in C^2_0(\mathcal X):= C^2(\mathcal X) \cap \{f_{|\partial \mathcal X}=0\}$ that 
\begin{equation}\label{schrottiso2}
\mathbb V_f[\mathbb S_f[\psi]] = \psi \quad \text{ on } \mathcal X.
\end{equation}
See Chapter 3 in \cite{CZ95} (or also Proposition 22 in \cite{N1}) for these facts. We will also repeatedly use below that the linear operator $\mathbb V_f$ is Lipschitz-continuous on $L^p_\lambda(\mathcal X)$ for $p=2,\infty$, with Lipschitz constant independent of $f$, see e.g., Lemma 25 in \cite{NvdGW18} for a proof.

\smallskip

\textbf{Condition \ref{blip}}: Let us write $\theta=\phi^{-1}\circ f, \theta'=\phi^{-1}\circ h$ for $\theta, \theta' \in \Theta$ so that $$\mathscr G(\theta) - \mathscr G(\theta') = u_{f}-u_h = \mathbb V_f[(f-h)u_h].$$ Using $L^p$-continuity of $\mathbb V_f$ and that composition with regular link functions is Lipschitz for $L^p$-norms (Lemma 29 in \cite{NvdGW18}), 
\begin{equation}\label{forwardest}
\|\mathbb V_f[(f-h)u_h]\| \lesssim \|u_h\|_\infty \|f-h\| \lesssim \|\theta - \theta'\|
\end{equation} 
both for $\|\cdot\|$ equal to the $L^2(\mathcal X)$ and the $L^\infty(\mathcal X)$-norm, and with constants independent of $f$. Here we have used also that 
\begin{equation}\label{forwardbd}
\|u_h\|_\infty \le c\|g\|_\infty, \quad 0 \le h \in C^\beta,
\end{equation}
 for a fixed constant $c>0$, as follows, e.g., from the Feynman-Kac representation of $u_h$ (see (5.35) in \cite{NvdGW18}). Then (\ref{forwardbd}) also implies the first inequality in Condition \ref{blip}.

\smallskip

\textbf{Conditions \ref{quadratic} and \ref{infokey}}: If $f_0=\phi(\theta_0), f_h=\phi(\theta_0+h)$, then Proposition 4 in \cite{N1} and again regularity of the link function $\phi$ imply, for $\mathbb V_f$ the inverse Schr\"odinger operator,
$$\|u_{f_h} - u_{f_0} - \mathbb V_{f_0}[u_{f_0}(f_h-f_0)]\|_{L^2(\mathcal X)} =O(\|f_h -f_0\|_\infty^2) = O(\|h\|_\infty^2).$$ Then by the chain rule for  $\phi \circ \theta$ and continuity of the operator $\mathbb V_{f_0}$ on $C(\mathcal X)$,
\begin{equation}\label{quadschrott}
\|u_{f_h} - u_{f_0} - \mathbb V_{f_0}[u_{f_0} \phi' (\theta_0)h]\|_{L^2(\mathcal X)} =O(\|h\|_\infty^2)
\end{equation}
which shows that the linearised `score' operator $\mathbb I_{\theta_0}: L^2(\mathcal X) \to L^2(\mathcal X)$ equals 
\begin{equation}\label{score}
\mathbb I_{\theta_0} = \mathbb V_{f_0}[u_{f_0}\phi' (\theta_0) \cdot] \quad\text{ with adjoint }\quad  \mathbb I_{\theta_0}^*=u_{f_0}\phi' (\theta_0)  \mathbb V_{f_0}[\cdot].
\end{equation}
 We see that $\mathbb I_{\theta_0}$ is a continuous operator on both $L^2(\mathcal X)$ and $L^\infty(\mathcal X)$ since $\mathbb V_{f_0}$ is and since both $u_{f_0}$ and $\phi'(\theta_0)$ are bounded functions. Now as in Section 4.2 in \cite{N1} we can define
\begin{equation}\label{invopod}
\tilde \psi \equiv (\mathbb I_{\theta_0}^* \mathbb I_{\theta_0})^{-1}(\psi) \equiv \frac{\mathbb S_{f_0} \mathbb S_{f_0} \big[\frac{\psi}{u_{f_0}\phi'(\theta_0)}\big]}{u_{f_0}\phi'(\theta_0)},~~~ \psi \in C^\infty(\mathcal X),
\end{equation}
where we note that $\min(u_{f_0}, \phi' (\theta_0)))>0$ throughout $\mathcal X$ by $g \ge g_{min}$ and the Feynman-Kac formula (cf.~(5.36) in \cite{NvdGW18}) and since $\theta_0 \in C^\infty(\mathcal X)$ is bounded. Moreover since $f_0$ is smooth by assumption we also have $u_{f_0} \in C^\infty(\mathcal X)$ (as in Lemma 27 in \cite{NvdGW18}, for instance). Then, for all $\psi \in C^{\infty,2}(\mathcal X)$ one checks directly from the definitions and the product rule that $\mathbb S_{f_0} \big[\psi/(u_{f_0}\phi'(\theta_0))]\in C^2_0(\mathcal X)$.  We can thus apply (\ref{schrottiso2}) to obtain  $$\mathbb I_{\theta_0} \tilde \psi =  \mathbb V_{f_0} \Big[\mathbb S_{f_0} \mathbb S_{f_0} \big[\frac{\psi}{u_{f_0}\phi'(\theta_0)}\big] \Big] =  \mathbb S_{f_0} \big[\frac{\psi}{u_{f_0}\phi'(\theta_0)}\big]$$ and another application of (\ref{schrottiso2}) implies $\mathbb I^*_{\theta_0} \mathbb I_{\theta_0} \tilde \psi= \psi$ and hence Condition \ref{infokey}, in particular $(\mathbb I_{\theta_0}^* \mathbb I_{\theta_0})^{-1}$ is a proper inverse mapping $C^{\infty,2}(\mathcal X)$ into $C^\infty(\mathcal X)$. What precedes also explains the form of the asymptotic variance in Theorem \ref{mainschrott}.

\smallskip

\textbf{Conditions \ref{prior} and \ref{DN}}:  We will use results in \cite{GN19} for general non-linear inverse problems. Using the bounds (\ref{forwardest}) and (\ref{forwardbd}) the conditions formulated at the beginning of Section A in \cite{GN19} can be verified for the PDE arising from the Schr\"odinger equation with $\kappa=\gamma=0$. Lemma 16 in \cite{GN19} (which for $\kappa=0$ permits to replace $H^\alpha_c$ by $H^\alpha$ in its Condition 3) then verifies the lower bound for $\pi(\delta_N)$  in Condition \ref{prior} for the rescaled  prior $\Pi_N$ with RKHS $$\mathcal H_N=(H^\alpha(\mathcal X), \sqrt N \delta_N \|\cdot\|_{H^\alpha(\mathcal X)})~\text { where } \delta_N = N^{-\alpha/(2\alpha+d)}.$$ Moreover, since $E\|\theta'\|_{L^2}^4<\infty$ the moment condition is also verified. To verify Condition \ref{DN}, we will choose as regularisation space $\mathcal R = C^\beta(\mathcal X)$ equipped with the $C^\beta$-norm for any $\max(2, d/2)<\beta<\alpha-d/2$. We apply Theorem 14 in \cite{GN19} to the effect that we can find $L_0,M>0$ large enough depending on $L$ such that the set $$\tilde \Theta_N = \{\theta \in \mathcal R: \|u_{\phi(\theta)} - u_{\phi(\theta_0)}\|_{L^2} \le L_0 \delta_N;~ \|\theta\|_{C^\beta} \le M \}$$ satisfies $$\Pi(\tilde \Theta_N^c|D_N) =o_{P_{\theta_0}^N}(\eta_N), \quad\eta_N = e^{-(L+1) N \delta_N^2}.$$  We next show that for all $N$ large enough $$\tilde \Theta_N \subset \Theta_N = \{\theta \in \mathcal R: \|\theta - \theta_0\|_{\infty} \le \bar \delta_N;~ \|\theta\|_{C^\beta} \le M \}$$ and hence Condition \ref{DN}, for convergence rate $$\bar \delta_N \equiv N^{-r(\alpha)}\quad \text{ for any } r(\alpha) < \frac{\alpha}{2\alpha+d} \cdot \frac{\beta-\frac{d}{2}}{\beta+2},\quad \alpha>\beta-d/2>0.$$  Indeed, just as in Lemma 28 in \cite{NvdGW18}, using the Sobolev imbedding theorem, standard interpolation inequalities for Sobolev spaces (e.g., (5.9) in \cite{NvdGW18}) and regularity estimates for the Schr\"odinger equation (e.g., Lemma 27 in \cite{NvdGW18}), we have
\begin{align*}
\|f-f_0\|_\infty & \lesssim \|u_f - u_{f_0}\|_{C^2} \lesssim \|u_f - u_{f_0}\|_{H^{2+d/2+\epsilon}} \\
& \lesssim \|u_f - u_{f_0}\|_{L^2}^\theta \|u_f -u_{f_0}\|_{H^{\beta+2}}^{1-\theta} \\
& \lesssim \delta_N^\theta (\|f\|_{C^\beta}+\|f_0\|_{C^\beta}) =o(\bar \delta_N)
\end{align*}
where $\theta = (\beta-d/2-\epsilon)/(\beta+2)$. By our hypotheses on $\beta$ the sequence $\bar \delta_N$ converges to zero and since $f_0>f_{min}$ we then also have $\inf_{x \in \mathcal X}f(x) >f_{min}$ for all $N$ large enough. Then composition with $\phi^{-1}$ is Lipschitz on $(f_{min}, \infty)$  so that $\|\theta- \theta_0\|_\infty \lesssim \|f-f_0\|_\infty$ and we finally deduce the inclusion $\tilde \Theta_N \subset \Theta_N$ follows for all large enough $N$.

\smallskip

\textbf{Conditions \ref{prior} and \ref{quant}}: The conditions (\ref{simple}) and (\ref{mbdc}) are checked, in Subsection \ref{calcul}. The RKHS-norm of the rescaled Whittle-Mat\'ern prior from (\ref{schrottprior}) equals
\begin{equation}\label{RKHSnorm2}
\delta_N \|\tilde \psi\|_{\mathcal H_N} = \sqrt N \delta_N^2 \|\tilde \psi\|_{H^\alpha(\mathcal X)} \to 0
\end{equation}
as $N \to \infty$ since $\tilde \psi \in C^\infty(\mathcal X) \subset \mathcal R \cap H^\alpha(\mathcal X)$ (cf.~after (\ref{invopod})) and $\alpha>d/2$, verifying (\ref{rkhsbds}). 

\subsection{Proof of Theorem \ref{mainpnt}}

We again verify the general Conditions \ref{blip}-\ref{quant}.

\smallskip

\textbf{Condition \ref{blip}}: The Lipschitz estimate for $L^2$ and $L^\infty$ norms follows from Theorem 2.2 (case $k=0$) in \cite{MNP19} . The uniform boundedness of the forward map is clear since $\mathscr G(\theta)$ takes values in the compact group $SO(n)$. 

\smallskip

\textbf{Conditions \ref{quadratic} and \ref{infokey}}: The quadratic approximation for the linearisation is checked in Lemma \ref{lemma:approx} with $\rho_{\theta_0}(h) \lesssim \|h\|_\infty^2.$ For the required mapping properties of $\mathbb I_{\theta_0}$ on $L^2$ and on $L^\infty$ see Remark \ref{conxray} . Theorem \ref{theorem:inversefisherinfo} allows us to define $\tilde \psi_{\theta_0} = (\mathbb I^*_{\theta_0} \mathbb I_{\theta_0})^{-1}\psi$ which determines another element of $C^\infty(M,\mathfrak{so}(n))\subset H^\alpha(M)$.

\smallskip

\textbf{Conditions \ref{prior} and \ref{DN}}:  The verification of this condition is based on results in \cite{MNP19}, with our prior satisfying Condition 3.1 there. The lower bound for $\pi(\delta_N)$ is given in Lemmas 5.15 and 5.16 in \cite{MNP19} with $\delta_N = N^{-\alpha/(2\alpha+2)},$ and the finiteness of fourth moments of the prior is also clear. Next, it is shown in Theorem 5.19 in \cite{MNP19}, that we can take for $\mathcal R$ a $C^\beta$-H\"older-space, $M_N=M<\infty$, and for any integer $\beta'$ s.t.~$1<\beta'<\beta<\alpha-1$, 
\begin{equation}\label{bardelta}
\bar \delta_N = N^{-\frac{\alpha}{2\alpha+2} \frac{(\beta'-1)^2}{(\beta')^2}} = N^{-r(\alpha)},
\end{equation}
since the $L^\infty(M)$-rate can be bounded by the $H^{1+\epsilon}(M)$-rate (Sobolev imbedding) which in turn can be bounded by the $L^2$-rate to the power $(\beta-1-\epsilon)/\beta$ in view of the usual interpolation inequality for Sobolev norms. Also, we can choose $\eta_N$ as desired (noting that the conclusion of Theorem 5.19 in \cite{MNP19} in fact holds for any $C>0$ large enough provided $m', m''$ are large enough).

\smallskip

\textbf{Condition \ref{quant}}: The conditions (\ref{simple}) and (\ref{mbdc}) are checked in Subsection \ref{calcul}. For the prior-related conditions, we notice that the isomorphism theorem in Section \ref{sec:mainXray} implies $\tilde \psi_{\theta_0} \in C^\infty(M) \subset \mathcal R \cap H^\alpha(M)$ and so as $N \to \infty$, since $\alpha>1$,
\begin{equation}\label{RKHSnorm}
\delta_N \|\tilde \psi_{\theta_0}\|_{\mathcal H_N} = \sqrt N \delta_N^2 \|\tilde \psi_{\theta_0}\|_{H^\alpha(M)} \to 0.
\end{equation}

\subsection{About conditions (\ref{simple}) and (\ref{mbdc})}\label{calcul}

We finally check the quantitative conditions (\ref{simple}) and (\ref{mbdc}) for $\alpha-d/2> \beta>2d$ large enough -- the proofs are the same for both inverse problems and in fact only depend on the fact that $\Theta_N$ is a  subset of a $C^\beta$-ball and that its $L^\infty$-rate of contraction about $\theta_0$ is $\bar \delta_N=N^{-r(\alpha)}, r(\alpha)>0$, as well as on the quadratic approximation $\rho_{\theta_0}(h) = O(\|h\|^2_\infty)$ in Condition \ref{quadratic}: The covering numbers of a $\beta$-H\"older ball in dimension $d$ are of the order $$\log N(\Theta_N, \|\cdot\|_\infty, \epsilon) \lesssim \Big(\frac{1}{\epsilon}\Big)^{d/\beta},\quad \beta>0,$$  see (4.184) in \cite{GN16} for the case when the H\"older functions are defined over $[0,1]^d$, and this bound applies to our setting by a standard extension arguments (and regarding $M, \mathcal X$ as subsets of $[0,1]^d$, with $d=2$ in the former case). Also, by the preceding proofs we can take
$$\rho_{\theta_0}(\theta-\theta_0 + (t/\sqrt N ) \tilde \psi) \lesssim \bar \delta_N^2 \equiv \sigma_N.$$  We first note that the quantity in (\ref{simple}) is bounded by
\begin{equation}\label{simple2}
\sqrt N \bar \delta_N^{2} \int_0^1 (\bar \delta_N^2\epsilon)^{-d/(2\beta)} d\epsilon \lesssim \sqrt N  \bar \delta_N^{2 - \frac{d}{\beta}}
\end{equation}
since $\beta>d/2$. We will eventually show that the last bound converges to zero as $N \to \infty$, which also implies $N \sigma^2_N \lesssim N \bar \delta_N^4 \to 0$. The middle term in the maximum in (\ref{mbdc}) can similarly be bounded by 
\begin{align*}
\sqrt N  \mathscr J_N(\sigma_N, 1) &= \sqrt N  \int_0^{\sigma_N} \epsilon^{-d/(2\beta)} d\epsilon \lesssim \sqrt N \bar \delta_N^{2-\frac{d}{\beta}},
\end{align*}
and hence is of the same order as the one in (\ref{simple2}). For the third member in the maximum (\ref{mbdc}) we have, by a similar calculation,
\begin{align}\label{poisson}
 \frac{\bar \delta_N \sqrt{\log N}}{\sigma_N^2} \mathscr J_N^2\big(\sigma_N, 1 \big) &\lesssim \sqrt{\log N} \bar \delta_N^{1-\frac{2d}{\beta}}.
\end{align}
We can conclude from what precedes that it suffices to show that
\begin{equation}
\max\big(\sqrt N  \bar \delta_N^{2 - \frac{d}{\beta}}, N \bar \delta_N^3, \sqrt{\log N} \bar \delta_N^{1-\frac{2d}{\beta}} \big) \to 0
\end{equation}
as $N \to \infty$. This requires $\beta>2d$ and then simplifies to the basic requirement $N \bar \delta_N^3 \to 0$. In both the Schr\"odinger and the $X$-ray case we have $\bar \delta_N = N^{-r(\alpha)}$ with precise exponent $r(\alpha)>0$ given in the preceding subsections, which thus simplifies to $r(\alpha)>1/3$. For the rate $\bar \delta_N$ obtained in the Schr\"odinger model this necessitates (\ref{ohlord}) to hold, while in the $X$-ray case the corresponding rate translates into the condition 
\begin{equation}\label{ohjesus}
\frac{\alpha}{2\alpha+2}\frac{(\alpha-2)^2}{(\alpha-1)^2} >1/3,
\end{equation}
satisfied for  $\alpha \ge 9$. Both requirements on $\alpha$ imply in particular that we can choose $\beta$ such that $2d<\beta<\alpha-d/2$ (with $d=2$ in the $X$-ray case).

\section{Analytical results for non-Abelian $X$-ray transforms}

\subsection{Main results} \label{sec:mainXray}

This section contains the definitions and statements for the main analytical results needed on the non-Abelian $X$-ray transform, whose proofs can be found in Sec. \ref{lin}, \ref{sec:forward} and \ref{sec:isomorphism}. In particular, we compute the linearization of the map $\Phi\mapsto C_\Phi$ defined in \eqref{eq:Cphi} and its associated Fisher information operator. We then prove forward mapping properties of these operators in a fairly general setting (convex, non-trapping Riemannian manifolds). Finally, we show in the case of the Euclidean disk that the Fisher information operator is a bijection in suitable spaces. 

\subsubsection{Linearization and forward mapping properties on convex, non-trapping manifolds} \label{ssec:linearization}

Consider $(M,g)$ a $d$-dimensional Riemannian manifold with boundary that is non-trapping (in the sense that every geodesic reaches $\partial M$ in finite time) and has strictly convex boundary (in the sense of having a positive definite second fundamental form $\Pi$).  For background on such manifolds and the definitions that follow we refer to \cite{Sharafutdinov,PSU_book}.
Let $SM$ denote the unit sphere bundle on $M$, i.e.
$$SM:=\{(x,v)\in TM : |v|_g=1\}$$
with footpoint projection $\pi:SM\to M$.
We define the volume form on $SM$ by $d\Sigma^{2d-1}(x,v)=dV^d(x)\wedge dS_x(v)$, where $dV^d$ is the volume form on $M$ and $dS_x$ is the volume form on the fibre $S_x$.
The boundary of $SM$ is 
$$\partial SM:=\{(x,v)\in SM : x\in \partial M\}.$$
On $\partial SM$ the natural volume form is $d\Sigma^{2d-2}(x,v)=dV^{d-1}(x)\wedge dS_x(v)$, where $dV^{d-1}$ is the volume form on $\partial M$. We distinguish two subsets of $\partial SM$ (influx and outflux boundaries)
$$\partial_{\pm}SM:=\{(x,v)\in \partial SM : \pm\langle v,\nu(x)\rangle_g\geq 0\},$$
where $\nu(x)$ is the inward unit normal vector on $\partial M$ at $x$. It is easy to see that
$$\partial_{0}SM:=\partial_+SM\cap\partial_-SM=S(\partial M).$$
Given $(x,v)\in SM$, we let $\tau(x,v)$ denote the first time where the geodesic determined by $(x,v)$ hits $\partial M$ and we set $\mu(x,v):=\langle \nu(x),v\rangle$ for $(x,v)\in \partial SM$.
We let $X$ denote the geodesic vector field.

Fixing $n\in \mathbb{N}$, in order to give the linearization of the map 
\begin{align*}
    C^\infty(M, \Cm^{n\times n})\ni \Phi\mapsto C_\Phi \in C^\infty(\partial_+ SM,\Cm^{n\times n})    
\end{align*}
defined in \eqref{eq:Cphi}, we first recall some definitions. Given $m$ an integer and $\Theta\in C^\infty(M,\Cm^{m\times m})$ a skew-hermitian matrix field, we define the {\it attenuated X-ray transform} with attenuation $\Theta$
\[ I_\Theta\colon C^\infty(M,\Cm^m) \to C^\infty(\partial_+ SM, \Cm^m) \]
through $I_\Theta f := u|_{\partial_+ SM}$, where $u:SM\to \Cm^m$ solves the transport equation 
\begin{align*}
    Xu + \Theta u = -f \qquad (SM), \qquad u|_{\partial_- SM} = 0.
\end{align*}
Such a transform extends as a bounded map 
\begin{align}
    I_\Theta\colon L^2(M, \C^m) \to L^2(\partial_+ SM\to \C^{m}, (\mu/\tau) d\Sigma^{2d-2}), 
    \label{eq:funcsetting}
\end{align}
and we denote $I_\Theta^*$ its adjoint in this functional setting (computed in \eqref{eq:adjoint} below). Note that this differs from the volume form $\mu d\Sigma^{2d-2}$ on $\partial_+ SM$ determined by Santal\'o's formula (the symplectic volume form). For the unit disc in $\mathbb{R}^2$, $\mu/\tau=1/2$, so the probability measure $(\mu/\tau) d\Sigma^{2}$  agrees with $\lambda$. In general, and thanks to Lemma \ref{lemma:tauextension} below, the measure
$(\mu/\tau) d\Sigma^{2d-2}$ determines an {\it equivalent} $L^{2}$-norm as $d\Sigma^{2d-2}$ since $\mu/\tau$ is smooth and bounded away from zero.

These attenuated X-ray transforms are now well-studied \cite{E,No,Novikov_nonabelian,PSUGAFA,MNP19,SU}, and their connection to the scattering map \eqref{eq:Cphi} is as follows: the linearization of the map \eqref{eq:Cphi} about a point $\Phi$ involves an attenuated X-ray transform whose integrands belong to $C^\infty(M,\Cm^{n\times n})$, with attenuation $\Theta(\Phi,\Phi)$, a matrix field described through the formula (pointwise on $M$)
\begin{align*}
    \Theta(\Phi,\Phi) \cdot U := \Phi U - U \Phi, \qquad U\in \Cm^{n\times n}.
\end{align*}
The matrix field $\Theta(\Phi,\Phi)$ is skew-hermitian on $\Cm^{n\times n}$ equipped with the hermitian inner product $(A,B)\mapsto \text{tr} (AB^*)$.


More precisely, we prove in Section \ref{lin} the following lemma. 

\begin{Lemma} \label{lemma:approx}
    Let $(M,g)$ be a non-trapping manifold with strictly convex boundary. Given $\Phi\in C(M,\mathfrak{u}(n))$ and upon setting 
    \begin{align}
	\mathbb I_{\Phi}(h):=I_{\Theta(\Phi,\Phi)}(h)C_{\Phi}	
	\label{eq:linearized}
    \end{align}
    for $h\in C(M,\C^{n\times n})$ we have
    \begin{align*}
	\norm{C_{\Phi+h}-C_{\Phi}-\mathbb I_{\Phi}(h)}_{L^{2}}\lesssim \norm{h}_{L^{\infty}}\norm{h}_{L^{2}},	
    \end{align*}
    where the norm on the left-hand side is the $L^2(\partial_+ SM\to \C^{m}, (\mu/\tau) d\Sigma^{2d-2})$ norm.
\end{Lemma}

In addition to \eqref{eq:linearized}, since $C_\Phi(x,v) \in U(n)$ for all $(x,v)\in \partial_+ SM$,
the Fisher information operator $\mathbb{N}_{\Phi}:=\mathbb I^*_{\Phi}\mathbb I_{\Phi}$ of the problem is directly related to the associated normal operator $I_{\Theta(\Phi,\Phi)}^{*}I_{\Theta(\Phi,\Phi)}$, namely: 
\begin{align}
    \mathbb{N}_{\Phi}:=\mathbb I^*_{\Phi}\mathbb I_{\Phi}=I_{\Theta(\Phi,\Phi)}^{*}I_{\Theta(\Phi,\Phi)}.
    \label{eq:fisherXray}
\end{align}
In particular, the forward mapping properties of $\mathbb{N}_\Phi$ are a special case of a more general result on the mapping properties of ``normal'' operators $I_\Theta^* I_\Theta$, which we prove in Section \ref{sec:forward}. 

\begin{theorem}\label{thm:forwardmapping}
    Let $(M,g)$ be a non-trapping manifold with strictly convex boundary, and let $\Theta\in C^\infty(M,\Cm^{m\times m})$. The operator $I_\Theta^* I_\Theta$ maps $C^\infty(M, \C^m)$ into itself.
\end{theorem}

From this result, it becomes straightforward to deduce that the Fisher information operator \eqref{eq:fisherXray} maps $C^\infty(M, \C^{n\times n})$ into itself. However, since $\Phi$ is often valued into a strict subalgebra of $\Cm^{n\times n}$, the last result below requires a Lie-algebra specific refinement. Let $G$ be any compact Lie group. Without loss of generality we may assume that $G\subset U(n)$, where $U(n)$ is the unitary group of $n\times n$ matrices and let $\mathfrak{g}$ be the Lie algebra of $G$. We are essentially interested in the case of $G=SO(n)$, where $\mathfrak{g}=\mathfrak{so}(n)$. Let us denote 
\begin{align}
    \C^{n\times n} = \mathfrak{g}\oplus \mathfrak{g}^{\perp}
    \label{eq:splitting}
\end{align}
the orthogonal splitting of $\Cm^{n\times n}$ for the Frobenius inner product. (When $\mathfrak{g}=\mathfrak{u}(n)$, $\mathfrak{g}^\perp$ is the space of hermitian matrices).

\begin{theorem} \label{thm:forwardmappingFisher}
    Let $(M,g)$ be a non-trapping manifold with strictly convex boundary, and let $\Phi\in C^\infty(M,\Cm^{n\times n})$. Then the following hold. 
    
    (1) The Fisher information operator $\mathbb{N}_\Phi$ \eqref{eq:fisherXray} maps $C^\infty(M, \C^{n\times n})$ into itself. 

    (2) If $\Phi\in C^\infty(M,\mathfrak{g})$, then in the splitting \eqref{eq:splitting}, the operator $\mathbb{N}_\Phi$ maps $C^\infty(M, \mathfrak{g})$ into itself and $C^\infty(M, \mathfrak{g}^\perp)$ into itself. 
\end{theorem}



\subsubsection{Isomorphism properties on the Euclidean disk}

In light of Theorem \ref{thm:forwardmapping}, the next question is then whether an isomorphism property holds. With the current tools available, such a question cannot be answered within the level of generality of the previous section. However, if the manifold $M$ is the Euclidean disk and the attenuation matrix $\Theta$ is compactly supported, then the normal operator $I_\Theta^* I_\Theta$ can be viewed as a relatively compact perturbation of the unattenuated case ($\Theta=0$), whose sharp mapping properties have recently been described in \cite{M}. This allows to prove in Section \ref{sec:isomorphism} an isomorphism property, using microlocal tools as well as Fredholm theory on a suitable scale of Hilbert spaces.

\begin{theorem}\label{thm:isomorphism}
    Suppose $M$ is the unit disk $\{(x,y)\in \R^2,\ x^2+y^2\le 1\}$, equipped with the Euclidean metric, and let $\Theta$ be a smooth, skew-hermitian $m\times m$ matrix field on $M$, with compact support in $M^{int}$. Then the map
    \begin{align*}
	I_\Theta^* I_\Theta \colon C^\infty(M, \C^{m}) \to C^\infty(M, \C^{m})
    \end{align*}
    is an isomorphism. 
\end{theorem}

Theorem \ref{thm:isomorphism} is an abridged version of Theorem \ref{thm:iso} below, where additional isomorphism properties on a special Sobolev scale (defined in Eqs. \eqref{eq:wtH0} and \eqref{eq:wtH}) are also given. 

Finally, we explain how Theorem \ref{thm:isomorphism} yields the Fisher information result that is needed for the proof of the Bernstein-von Mises theorem for the non-Abelian X-ray transform. Let $G$ be any compact Lie group and $\mathfrak{g}$ as in Section \ref{ssec:linearization}. 



\begin{theorem} 
    Let $M$ be the unit disk with the Euclidean metric and let $\Phi\in C^{\infty}_{c}(M,\mathfrak{g})$. Then 
    \[\mathbb{N}_{\Phi}=I_{\Theta(\Phi,\Phi)}^*I_{\Theta(\Phi,\Phi)}: C^{\infty}(M,\mathfrak{g})\to C^{\infty}(M,\mathfrak{g})\]
    is a bijection.
\label{theorem:inversefisherinfo}
\end{theorem}

\begin{proof} 
    Theorem \ref{thm:isomorphism} implies right away that 
    \[\mathbb{N}_{\Phi}:C^{\infty}(M,\C^{n\times n})\to C^{\infty}(M,\C^{n\times n})\]
    is a bijection. The further isomorphism property on $C^{\infty}(M,\mathfrak{g})$ is a direct consequence of item (2) in Theorem \ref{thm:forwardmappingFisher} and the fact that $C^{\infty}(M,\C^{n\times n})=C^{\infty}(M,\mathfrak{g})\oplus C^{\infty}(M,\mathfrak{g}^{\perp})$.
\end{proof}


\subsection{Linearizing $C_{\Phi}$. Proof of Lemma \ref{lemma:approx}}\label{lin}

Fix $(M,g)$ a compact non-trapping manifold with strictly convex boundary. We let $\varphi_{t}$ denote the geodesic flow of $g$; the integrals that appear below in the variable $t$ are all compositions of functions with $\varphi_{t}$; we avoid writting this explicitly in order to prevent notation cluttering. An {\it integrating factor} for $\Phi$ is a function $R_{\Phi}\in C(SM,GL(n,\C))$ which is differentiable along the geodesic vector field $X$ and $XR_{\Phi}+\Phi R_{\Phi}=0$.
If $\Phi$ is smooth, then it is not hard to see that smooth integrating factors always exist cf. \cite{PSU_book}.

Let $U_{\Phi}$ denote the unique integrating factor with $U_{\Phi}|_{\partial_{-}SM}=\id$. Then $C_{\Phi}:\partial_{+}SM\to GL(n,\C)$ is defined as
\[C_{\Phi}:=U_{\Phi}|_{\partial_{+}SM}.\]
We can also consider the unique integrating factor $u_{\Phi}$ with $u_{\Phi}|_{\partial_{+}SM}=\id$. It is immediate to check that $u_{\Phi}|_{\partial_{-}SM}=[C_{\Phi}]^{-1}\circ\alpha$, where $\alpha:\partial SM\to\partial SM$ denotes the scattering relation of the metric.

The next lemma will be useful for our purposes.

\begin{Lemma}\label{lemma:generalC}
    Let $R_{\Phi}$ and $R_{\Psi}$ be integrating factors for continuous matrix fields $\Phi$ and $\Psi$ respectively. Then 
    \begin{align*}
	C_{\Phi}-C_{\Psi}&=R_{\Phi}\left[\int_{0}^{\tau(x,v)}R_{\Phi}^{-1}(\Phi-\Psi)R_{\Psi}\,dt\right] (R^{-1}_{\Psi})\circ\alpha\\
	&=R_{\Phi}\left[I(R_{\Phi}^{-1}(\Phi-\Psi)R_{\Psi})\right] (R^{-1}_{\Psi})\circ\alpha
    \end{align*}
    where $I:C(SM)\to C(\partial_{+}SM)$ is the standard $X$-ray transform.
\end{Lemma}

\begin{proof} 
    We first note that if $R$ solves $XR+\Phi R=0$, then any other integrating factor has the form $RF^{\sharp}$, where $F^{\sharp}$ is the first integral  (i.e. $XF^{\sharp}=0$) determined by $F\in C(\partial_{+}SM,GL(n,\C))$. Thus $R_{\Phi}=U_{\Phi}F^{\sharp}$ and from this we deduce
    \begin{equation}
	C_{\Phi}=R_{\Phi}(R_{\Phi}^{-1}\circ\alpha).
	\label{eq:CintermsofR}
    \end{equation}
    
    Next we observe that a computation gives
    \[X(R_{\Phi}^{-1}R_{\Psi})=R_{\Phi}^{-1}(\Phi-\Psi)R_{\Psi}.\]
    Integrating this along a geodesic between boundary points gives
    \[\int_{0}^{\tau(x,v)}R_{\Phi}^{-1}(\Phi-\Psi)R_{\Psi}\,dt=-R_{\Phi}^{-1}R_{\Psi}(x,v)+R_{\Phi}^{-1}R_{\Psi}\circ\alpha(x,v),\]
    for $(x,v)\in \partial_{+}SM$. The lemma follows from this and \eqref{eq:CintermsofR}.
\end{proof}

\begin{Definition} Given $\Phi,\Psi\in C(M,\C^{n\times n})$ and $h\in C(M,\C^{n\times n})$,
consider the unique matrix solution to $Xu+\Phi u-u\Psi=-h$ with $u|_{\partial_{-}SM}=0$.
We define the attenuated X-ray transform of $h$ with attenuation $\Theta(\Phi,\Psi)$ as
\[I_{\Theta(\Phi,\Psi)}(h):=u|_{\partial_{+}SM}.\]
\end{Definition}

 In terms of arbitrary integrating factors $R_{\Phi}$ and $R_{\Psi}$ we can give an integral expression for $I_{\Theta(\Phi,\Psi)}$ as
\begin{equation}
I_{\Theta(\Phi,\Psi)}(h)=R_{\Phi}\left[\int_{0}^{\tau(x,v)}R_{\Phi}^{-1}hR_{\Psi}\,dt\right]R^{-1}_{\Psi}.
\label{eq:formulaendo}
\end{equation}
Indeed, consider the unique matrix solution to $Xu+\Phi u-u\Psi=-h$ with $u|_{\partial_{-}SM}=0$.
By definition $u|_{\partial_{+}SM}=I_{\Theta(\Phi,\Psi)}(h)$. We compute
\begin{align*}
    X(R_{\Phi}^{-1}uR_{\Psi})&=R_{\Phi}^{-1}\Phi u R_{\Psi}+R_{\Phi}^{-1} Xu R_{\Psi}-R_{\Phi}^{-1} u \Psi R_{\Psi}\\
    &=-R_{\Phi}^{-1}h R_{\Psi}.
\end{align*}
Integrating along a geodesic between boundary points we get
\[R^{-1}_{\Phi}I_{\Theta(\Phi,\Psi)}(h) R_{\Psi}=\int_{0}^{\tau(x,v)}R_{\Phi}^{-1}hR_{\Psi}\,dt\]
and hence \eqref{eq:formulaendo} follows.

\begin{Remark}{\rm Lemma \ref{lemma:generalC} already contains the pseudo-linearization identity from \cite[Lemma 5.5]{MNP19}. Indeed, using $u_{\Phi}$ and $u_{\Psi}$ as integrating factors, the lemma and \eqref{eq:CintermsofR} give
\begin{align}
C_{\Phi}-C_{\Psi}&=\left[\int_{0}^{\tau(x,v)}u_{\Phi}^{-1}(\Phi-\Psi)u_{\Psi}\,dt\right]C_{\Psi}.\label{eq:psdo}\\
&=I_{\Theta(\Phi,\Psi)}(\Phi-\Psi)\,C_{\Psi}.\label{eq:psdoi}
\end{align}
}
\end{Remark}

To find the linearization of $C_{\Phi}$, let $\Phi_{s}$ be a curve of matrix-valued maps such that $\Phi_{0}=\Phi$ and $h:=\partial_{s=0}\Phi_{s}$.
Differentiating the equation $XU_{\Phi_{s}}+\Phi_{s}U_{\Phi_{s}}=0$ at $s=0$ we obtain
\[XH+hU_{\Phi}+\Phi H=0\]
where $H:=\partial_{s=0}U_{\Phi_{s}}$. Note that $H|_{\partial_{+}SM}=dC_{\Phi}(h)$.
Then the matrix $W:=HU^{-1}_{\Phi}$ satisfies
\[XW+\Phi W-W\Phi=-h.\]
Hence
\[W|_{\partial_{+}SM}=I_{\Theta(\Phi,\Phi)}(h)\]
and thus
\begin{equation}
dC_{\Phi}(h)=I_{\Theta(\Phi,\Phi)}(h)C_{\Phi}.
\label{eq:linC}
\end{equation}
We can now combine this with \eqref{eq:psdoi} to obtain
\begin{equation}
C_{\Phi+h}-C_{\Phi}-dC_{\Phi}(h)=(I_{\Theta(\Phi+h,\Phi)}(h)-I_{\Theta(\Phi,\Phi)}(h))C_{\Phi}.
\label{eq:lin2}
\end{equation}
We now use this identity to prove Lemma \ref{lemma:approx}.

\begin{proof}[Proof of Lemma \ref{lemma:approx}] 

    From \eqref{eq:psdo} and \eqref{eq:psdoi} we know that
    \[I_{\Theta(\Phi,\Psi)}(h)=\int_{0}^{\tau}u_{\Phi}^{-1}hu_{\Psi}\,dt.\]
    Thus
    \[I_{\Theta(\Phi+h,\Phi)}(h)-I_{\Theta(\Phi,\Phi)}(h)=\int_{0}^{\tau}(u_{\Phi+h}^{-1}-u_{\Phi}^{-1})hu_{\Phi}\,dt.\]
    Since $u_{\Phi}$ takes values in the unitary group, we can estimate using the Frobenius norm
    \[|(I_{\Theta(\Phi+h,\Phi)}(h)-I_{\Theta(\Phi,\Phi)}(h))C_{\Phi}|_{F}(x,v)\leq \int_{0}^{\tau}|(u^{*}_{\Phi+h}-u_{\Phi}^{*})h|_{F}\,dt.\]
    Using that $\tau\leq C_{0}\mu(x,v)$ (cf. Lemma \ref{lemma:tauextension} below) and Cauchy-Schwarz
    \[|(I_{\Theta(\Phi+h,\Phi)}(h)-I_{\Theta(\Phi,\Phi)}(h))C_{\Phi}|^{2}_{F}(x,v)\leq C_{0}\int_{0}^{\tau}|(u^{*}_{\Phi+h}-u_{\Phi}^{*})h|^{2}_{F}\,dt\mu(x,v)\]
    for $(x,v)\in \partial_{+}SM$. Integrating now over $\partial_{+}SM$ and using Santal\'o's formula we derive
    \[\norm{(I_{\Theta(\Phi+h,\Phi)}(h)-I_{\Theta(\Phi,\Phi)}(h))C_{\Phi}}_{L^{2}}\lesssim \norm{(u^*_{\Phi+h}-u^{*}_{\Phi})h}_{L^{2}}.\]
    
    Using \cite[Equation (5.8)]{MNP19} we have (strictly speaking the proof in \cite{MNP19} is for $U_{\Phi}$ but the same proof applies to $u^{*}_{\Phi}$)
    \[\norm{u^{*}_{\Phi+h}-u^{*}_{\Phi}}_{L^{2}}\lesssim \norm{h}_{L^{2}}\]
    and putting everything together using \eqref{eq:lin2}
    \begin{align*}
	\norm{C_{\Phi+h}-C_{\Phi}-dC_{\Phi}(h)}_{L^{2}}\lesssim \norm{h}_{L^{\infty}}\norm{h}_{L^{2}}.	
    \end{align*}
\end{proof}

\begin{Lemma}\label{lem:fisher}
    We have
    \[\mathbb{N}_{\Phi}:=\mathbb I^*_{\Phi}\mathbb I_{\Phi}=I_{\Theta(\Phi,\Phi)}^{*}I_{\Theta(\Phi,\Phi)}.\]
\end{Lemma}

\begin{proof}
    Since the matrix $C_{\Phi}$ is unitary we have
    \begin{align*}
	\langle \mathbb I_{\Phi}(\cdot),\mathbb I_{\Phi}(\cdot)\rangle_{L^{2}}=\langle I_{\Theta(\Phi,\Phi)}(\cdot),I_{\Theta(\Phi,\Phi)}(\cdot)\rangle_{L^{2}}	
    \end{align*}
    and the lemma follows.
\end{proof}

\begin{Remark}\label{conxray}{\rm Since the attenuated X-ray transform $I_{\Theta(\Phi,\Phi)}$ extends as a bounded map from $L^{2}(M)\to L^{2}(\partial_{+}SM)$, the same is true for $\mathbb I_{\Phi}$. Boundedness in $L^{\infty}$ for $\mathbb I_{\Phi}$ is also obvious from the integral expression
 \[I_{\Theta(\Phi,\Phi)}(h)=\int_{0}^{\tau}u_{\Phi}^{-1}hu_{\Phi}\,dt.\]
}
\end{Remark}


\subsection{Forward mapping properties. Proof of Theorems \ref{thm:forwardmapping} and \ref{thm:forwardmappingFisher}} \label{sec:forward}

Let $(M,g)$ be a non-trapping manifold with strictly convex boundary. We need the following facts (cf. \cite{PSU_book, Sharafutdinov}).

\begin{enumerate}
\item The function
\[\tilde{\tau}(x,v)=\left\{\begin{array}{ll}
\tau(x,v),\;\;(x,v)\in\partial_{+}SM,\\
-\tau(x,-v),\;\;(x,v)\in\partial_{-}SM\\
\end{array}\right.\]
belongs to $C^{\infty}(\partial SM)$. Actually $\tau:SM\to\mathbb{R}$ solves transport problem $X\tau=-1$
with $\tau|_{\partial_{-}SM}=0$ and the function $\tilde{\tau}=\tau(x,v)-\tau(x,-v)$ belongs to $C^{\infty}(SM)$.
\item The scattering relation $\alpha:\partial SM\to \partial SM$ is the diffeomorphism defined by 
$$\alpha(x,v)=\varphi_{\tilde{\tau}(x,v)}(x,v).$$
\item The scattering relation satisfies $\alpha^2=id$, based on the property $\tilde{\tau}\circ\alpha=-\tilde{\tau}$.
\end{enumerate}

For what follows it is convenient to consider $(M,g)$ isometrically embedded in a closed manifold $(N,g)$, so that the geodesic flow can run for all times.
Let $\rho\in C^{\infty}(N)$ be a boundary defining function for $\partial M$. That means that $\rho$ coincides with $M\ni x\mapsto d(x,\partial M)$ in a neighbourhood of $\partial M$, $\rho\geq 0$ on $M$ and $\partial M=\rho^{-1}(0)$. If we let $\nu$ be the inward unit normal, then  $\nabla \rho(x)=\nu(x)$ for all $x\in \partial M$.
Consider the function $h:\partial SM\times \R\to\R$ given by
\[h(x,v,t):=\rho(\pi\circ\varphi_{t}(x,v)).\]
Note
\begin{align*}
    h(x,v,0)=0, \qquad \left.\frac{d}{dt}\right|_{t=0} h(x,v,t)=\langle \nu(x),v\rangle, \qquad \left.\frac{d^2}{dt^2}\right|_{t=0} h(x,v,t)=\text{Hess}_{x}\rho(v,v).
\end{align*}

Hence there is a smooth function $R:\partial SM\times \R\to\R$ such that we can write
\begin{equation}\label{eq:taylorR}
h(x,v,t)=\langle \nu(x),v\rangle t+\frac{1}{2}\text{Hess}_{x}\rho(v,v) t^2+R(x,v,t)t^3.
\end{equation}
Since $h(x,v,\tilde{\tau}(x,v))=0$, it follows that
\begin{equation}\label{eq:tautilde}
\langle \nu(x),v\rangle +\frac{1}{2}\text{Hess}_{x}\rho(v,v) \tilde{\tau}+R(x,v,\t)\t^2=0.
\end{equation}
Note that $\t(x,v)=0$ iff $(x,v)\in \partial_{0}SM$.
Hence if we let
\[H(x,v,t):=\langle \nu(x),v\rangle +\frac{1}{2}\text{Hess}_{x}\rho(v,v)t+R(x,v,t)t^2\]
we see that $H$ is smooth, $H(x,v,\t(x,v))=0$ and
\[ \left.\frac{d}{dt}\right|_{t=0} H(x,v,t)=\frac{1}{2}\text{Hess}_{x}\rho(v,v).\]
But for $(x,v)\in \partial_{0}SM$, $\text{Hess}_{x}\rho(v,v)=-\Pi_{x}(v,v)<0$ and thus by the implicit function theorem, $\t$ is smooth in a neighbourhood of $\partial_{0}SM$. Since $\t$ is smooth in $\partial SM\setminus\partial_{0}SM$ this gives smoothness of $\t$ in $\partial SM$. A tweak of this argument gives the following lemma that is probably well-known to experts.
Recall that $\mu(x,v)=\langle\nu(x),v\rangle$ for $(x,v)\in \partial SM$.

\begin{Lemma} \label{lemma:tauextension}
    Let $(M,g)$ be a non-trapping manifold with strictly convex boundary. The function $\mu/\t$ extends to a smooth positive function on $\partial SM$ with values on $\partial_0 SM$ given by
    \begin{align*}
	\frac{\mu}{\t}(x,v) = \frac{\Pi_{x}(v,v)}{2}, \qquad (x,v)\in \partial_0 SM.
    \end{align*}
\end{Lemma}


\begin{proof} Using \eqref{eq:tautilde} we can write
\[\mu(x,v)= -\frac{1}{2}\text{Hess}_{x}\rho(v,v) \tilde{\tau}-R(x,v,\t)\t^2\]
and hence for $(x,v)\in \partial SM\setminus\partial_{0}SM$ near $\partial_{0}SM$ we can write
\begin{equation}\label{eq:p}
\mu/\t=-\frac{1}{2}\text{Hess}_{x}\rho(v,v)-R(x,v,\t)\t.
\end{equation}
But the right hand side of the last equation is a smooth function near $\partial_{0}SM$ since $R$ and $\t$ are; its
value at $(x,v)\in\partial_{0}SM$ is $\Pi_{x}(v,v)/2$. Finally, observe that $\mu$ and $\t$ are both positive for $(x,v)\in\partial_{+}SM\setminus\partial_{0}SM$ and both negative for $(x,v)\in\partial_{-}SM\setminus\partial_{0}SM$.
\end{proof}

\subsubsection{The maps $\Upsilon$ and $F$} We now introduce two important maps for what follows. 

Consider the map
\begin{align}
    \Upsilon:\partial_{+}SM\times [0,1]\to SM, \qquad \Upsilon(x,v,u):=\varphi_{u\tau(x,v)}(x,v).    
    \label{eq:Phi}
\end{align}
This map is smooth and it extends smoothly to
\[\Upsilon:\partial (SM)\times [0,1]\to SM\]
by setting $\Upsilon(x,v,u)=\varphi_{u\t(x,v)}(x,v)$. Note that $\Upsilon(x,v,0)=\text{Id}$, $\Upsilon(x,v,1)=\alpha(x,v)$ and $\Upsilon(\alpha(x,v),u) = \Upsilon(x,v,1-u)$.  In other words, if we let $\Gamma:\partial (SM)\times [0,1]\to\partial (SM)\times [0,1]$ be $\Gamma(x,v,u):=(\alpha(x,v),1-u)$, then $\Upsilon\circ \Gamma=\Upsilon$. The map $\Upsilon$ is a 2-1 cover with deck transformation $\Gamma$ away from $\partial_0 SM\times[0,1]$.

For brevity we shall denote $p:=\mu/\t\in C^{\infty}(\partial SM)$. We let $F:\partial SM\setminus\partial_{0}SM\times (0,1)\to \R$ be
\begin{align}
    F(x,v,u):=\frac{\rho(\pi\circ \varphi_{u\tilde{\tau}}(x,v))}{\tilde{\tau}^{2} u(1-u)}=\frac{h(x,v,u\t)}{\tilde{\tau}^{2} u(1-u)}>0.    
    \label{eq:F}
\end{align}

\begin{Proposition}\label{prop:Ffunc} The function $F$ extends to a smooth positive function $F:\partial SM\times [0,1]\to\R$ such that 
    \begin{enumerate}
	\item[(a)] $F(\alpha(x,v),u)=F(x,v,1-u)$;
	\item[(b)] $F(x,v,0)=p(x,v)$ and $F(x,v,1)=p\circ\alpha(x,v)$;
	\item[(c)] $F(x,v,u)=\Pi_{x}(v,v)$ for $(x,v,u)\in\partial_0 SM\times [0,1]$.
    \end{enumerate}
\end{Proposition}

\begin{proof} Using the definition of $F$ and \eqref{eq:taylorR} we can write
    \[F(x,v,u)=\frac{\mu(x,v) +\frac{1}{2}\text{Hess}_{x}\rho(v,v)u\t+R(x,v,u\t)u^2\t^2}{\t(1-u)}.\]
    Since $R$ and $\t$ are smooth, there is a smooth function $Q:\partial SM\times \R\to\R$ such that
    \[R(x,v,u\t(x,v))-R(x,v,\t(x,v))=(1-u)Q(x,v,u).\]
    Combining this with \eqref{eq:tautilde} we can write $F$ as
    \begin{equation}
	F(x,v,u)=-\frac{1}{2}\text{Hess}_{x}\rho(v,v)-R(x,v,\t)\t(1+u)+u^{2}Q(x,v,u)\t.
	\label{eq:Fisgood}
    \end{equation}
    The right hand side of this equation is a smooth function on $\partial SM\times \R$ thus showing that $F$
    extends to a smooth function on $\partial SM\times [0,1]$ as claimed.
    
    To check item (a), we check it first for $(x,v,u)\in \partial SM\setminus\partial_{0}SM\times (0,1)$.
    This is straightforward from the definition of $F$ and the fact that $\t\circ\alpha=-\t$. Since $\partial SM\setminus\partial_{0}SM\times (0,1)$ is dense in $\partial SM\times [0,1]$ item (a) follows.
    To check item (b) we use \eqref{eq:Fisgood} for $u=0$; it yields
    \[F(x,v,0)=-\frac{1}{2}\text{Hess}_{x}\rho(v,v)-R(x,v,\t)\t\]
    and from \eqref{eq:p} we see that it agrees with $p$. Combining this with item (b) we see that
    $F(x,v,1)=F(\alpha(x,v),0)=p\circ\alpha (x,v)$ as claimed.
    Item (c) follows from \eqref{eq:Fisgood} and the facts that $\t(x,v)=0$  and $\text{Hess}_{x}\rho(v,v)=-\Pi_{x}(v,v)$ for $(x,v)\in\partial_{0}SM$. Finally, the positivity of $F$ is a consequence of the positivity of $p$ and the second fundamental form $\Pi$. 
\end{proof}

\subsubsection{General mapping properties and proof of Theorems \ref{thm:forwardmapping} and \ref{thm:forwardmappingFisher}} Fix $m,p$ two arbitrary integers. Given a weight $w\in C^\infty(SM, \C^{m\times p})$ and for $f\in C^\infty(SM, \C^m)$, we define the weighted transform $\II^w\colon L^2(SM, \C^p) \to L^2(\partial_+ SM \to \C^m,\ \frac{\mu}{\tau} d\Sigma^{2d-2})$ as 
\begin{align*}
    \II^w f(x, v) := \int_0^{\tau(x,v)} w(\varphi_t(x,v)) f(\varphi_t(x,v))\ dt, \qquad (x,v) \in \partial_+ SM.
\end{align*}
An important space for what follows is given by 
\begin{align*}
    C^{\infty}_{\alpha} (\partial_+ SM) &:= \{ u \in C^\infty(\partial_+ SM), \quad u_\psi \in C^\infty(SM) \} \\
    &= \{ u \in C^\infty(\partial_+ SM), \quad A_+ u \in C^\infty(\partial SM) \},
\end{align*}
where for $u\in C^\infty(\partial_+ SM)$, we have defined $A_+ u \in C^\infty(\partial SM \backslash \partial_0 SM)$ as
\begin{align*}
    A_+ u (x,v) = \left\{
    \begin{array}{ll}
	u(x,v), & (x,v) \in \partial_+ SM, \\
	u(\alpha(x,v)), & (x,v) \in \partial_- SM.
    \end{array}
    \right.
\end{align*}
Such a space was first introduced in \cite{PU} as a 'natural' space of functions which are mapped into $C^\infty(M)$ through the traditional adjoint of the X-ray transform, and the second equality is a characterization proved in \cite{PU}. We extend this definition to vector-valued functions, namely $C^{\infty}_{\alpha} (\partial_+ SM, \C^m) := (C^{\infty}_{\alpha} (\partial_+ SM))^m$. With $\rho$ a boundary defining function for $M$ as above, we now show the following result. 

\begin{Proposition} \label{prop:mappinggeneral}
    Fix $m,p$ and a smooth weight $w$ as above. For every $s < 1$, the following mapping property holds: 
    \begin{align*}
	\II^w\colon \rho^{-s} C^\infty(SM,\C^p) \to \tau^{1-2s} C^\infty_{\alpha} (\partial_+ SM,\C^m).
    \end{align*}    
\end{Proposition}

\begin{proof}
    Given $f\in C^\infty(SM)$ and the function $F$ defined in \eqref{eq:F}, we consider the change of variable $t = \tau(x,v) u$, so that we may rewrite
    \begin{align*}
	I^w (\rho^{-s} f) (x,v) &= \tau(x,v)^{1-2s} \int_0^1 w(\Upsilon (x,v,u)) f(\Upsilon (x,v,u)) F^{-s}(x,v,u) \frac{du}{(u(1-u))^s}, \\
	&= \tau(x,v)^{1-2s} g(x,v),
    \end{align*}
    where
    \[g(x,v):=\int_0^1 w(\Upsilon (x,v,u)) f(\Upsilon(x,v,u)) F^{-s}(x,v,u) \frac{du}{(u(1-u))^s}\]
    and $\Upsilon$ is the map defined in \eqref{eq:Phi}. 
    
    All functions of $(x,v,u)$ involved in the definition of $g$ are defined and smooth for $(x,v)\in \partial SM$ (non-integer powers of $F$ are well-defined and smooth since $F$ is positive everywhere), and thus we may think of $g$ as $\tilde g|_{\partial_+ SM}$ for some $\tilde g$ whose definition is the same as above, but extended to $\partial SM$. 
    Since all the functions participating in the definition of $\tilde{g}$ satisfy the property $q(\alpha(x,v),u) = q(x,v,1-u)$, we have $\tilde g\circ \alpha = \tilde g = A_+ g$, and $\tilde g$ is smooth on $\partial SM$. In particular, the function $g$ belongs to $C_\alpha^\infty(\partial_+ SM, \C^m)$, which completes the proof. 
\end{proof}

The case of interest to us is when $s=0$, for which we obtain
\begin{align*}
    \II^w\colon C^\infty(SM,\C^p) \to \tau C^\infty_{\alpha} (\partial_+ SM,\C^m),
\end{align*}
and for $w\equiv 1$ and $m=p$, we will denote $\II^w = \II$. 

On to the attenuated X-ray transform $I_\Theta$ with $p=m$ and $\Theta\in C^\infty(M,\mathfrak{u}(m))$: fixing a smooth integrating factor $R\colon SM \to U(m)$ solution of $XR + \Theta R = 0$, we can write $I_\Theta f$ as 
\begin{align}
    I_\Theta f (x,v) = R (x,v) \II (R^{-1} f) (x,v), \qquad (x,v) \in \partial_+ SM.
    \label{eq:Rform}
\end{align}
In the functional setting \eqref{eq:funcsetting}, we then compute the adjoint: 
\begin{align*}
    (I_\Theta f, h)_{\frac{\mu}{\tau}} &= \int_{\partial_+ SM} \left\langle R(x,v) \II(R^{-1} f(x,v)), h(x,v)\right\rangle_{\C^m} \frac{\mu}{\tau}\ d\Sigma^{2d-2} \\
    &= \int_{\partial_+ SM} \left\langle \II(R^{-1} f) (x,v), \frac{1}{\tau} R^*(x,v) h(x,v)\right\rangle_{\C^m} \mu d\Sigma^{2d-2} \\
    &\stackrel{(s)}{=} \int_{SM} \left\langle R^{-1} f, \left( \frac{1}{\tau} R^* h\right) \circ \psi  \right\rangle_{\C^m}\ d\Sigma^{2d-1} \\
    &= \int_{M} \left\langle f(x), \int_{S_x} \left( (R^{-1})^*(x,v) \left( \frac{1}{\tau} R^* h\right) \circ \psi(x,v)  \right)\ dS_{x}(v) \right\rangle_{\C^m}\ dV^{d}(x),
\end{align*}
where Santal\'o's formula was used at step $(s)$. Note that we have used that the (componentwise) adjoint of $\II\colon L^2(SM) \to L^2(\partial_+ SM, \frac{\mu}{\tau} d\Sigma^{2d-2})$ is given by $\II^* h(x) = \frac{h}{\tau} (\psi(x,v))$, where $\psi\colon SM\to \partial_+ SM$ denotes the footpoint map, defined by $\psi(x,v) = \varphi_{-\tau(x,-v)}(x,v)$. 
This implies the following expression for the adjoint: 
\begin{align}
    I_\Theta^* h (x) = \int_{S_x} \left( (R^{-1})^*(x,v) \left( \frac{1}{\tau} R^* h\right) \circ \psi(x,v)  \right)\ dS_{x}(v).
    \label{eq:adjoint}
\end{align}

Notice that since $\Theta$ is skew-hermitian, we also have the pointwise relation $(R^{-1})^* (x,v) = R(x,v)$. We are now ready to compute associated normal operator $I_\Theta^* I_\Theta$:
\begin{align}
    I_\Theta^* I_\Theta f(x) &= \int_{S_x}  R(x,v) \left( \frac{1}{\tau} R^* I_\Theta f\right) \circ \psi(x,v)  \ dS_{x}(v) \nonumber\\
    &= \int_{S_x} R(x,v) \left( \frac{1}{\tau} R^* R \II(R^{-1} f)\right) \circ \psi(x,v) \ dS_{x}(v) \nonumber \\
    &= \int_{S_x} R(x,v) \left( \frac{1}{\tau} \II(R^{-1} f)\right) \circ \psi(x,v) \ dS_{x}(v),
    \label{eq:Nexpr}
\end{align}
where we have used that $R^* R = id_{m\times m}$ pointwise. We can now prove Theorem \ref{thm:forwardmapping}.

\begin{proof}[Proof of Theorem \ref{thm:forwardmapping}]
    Take $f$ smooth on $M$, then $R^{-1} f$ is smooth on $SM$, then by Proposition \ref{prop:mappinggeneral}, $\frac{1}{\tau} \II(R^{-1} f) \in C^{\infty}_\alpha(\partial_+ SM, \C^m)$. In particular, $\left( \frac{1}{\tau} \II(R^{-1} f)\right) \circ \psi(x,v)$ is smooth on $SM$, and so is its product with $R(x,v)$. Since $I_\Theta^* I_\Theta f$ is the fiberwise average of the latter product, it is smooth on $M$ as well. Theorem \ref{thm:forwardmapping} is proved. 
\end{proof}

We finally make the adjustments needed to incorporate restrictions to certain Lie-algebra valued elements, proving Theorem \ref{thm:forwardmappingFisher}.

\begin{proof}[Proof of Theorem \ref{thm:forwardmappingFisher}]   
    The proof of (1) follows directly from Theorem \ref{thm:forwardmapping} and the fact that when $\Phi\in C^\infty(M, \Cm^{n\times n})$, then $\Theta(\Phi,\Phi)$ is a smooth matrix field on $\Cm^{n\times n}$.

    On to the proof of (2), suppose that $\Phi$ is $\mathfrak{g}$-valued. Equation \eqref{eq:formulaendo} allows us to write
    \[I_{\Theta(\Phi,\Phi)}(f)=\int_{0}^{\tau}u^{-1}_{\Phi}f u_{\Phi}\,dt=\int_{0}^{\tau}\text{Ad}_{u_{\Phi}^{-1}}(f)\,dt\]
    where $\text{Ad}_{g}(f)=gfg^{-1}$ is the Adjoint representation. The map $I_{\Theta(\Phi,\Phi)}^*$ can be easily computed using \eqref{eq:adjoint} to obtain
    \[I_{\Theta(\Phi,\Phi)}^*(h)=\int_{S_{x}}\text{Ad}_{u_{\Phi}}((h/\tau)^{\sharp})(x,v)\,dS_{x}.\]
    But the Adjoint representation preserves $\mathfrak{g}$ and thus $\mathbb{N}_{\Phi}$ maps $C^{\infty}(M,\mathfrak{g})$ into itself. 
    In fact, since $\text{Ad}_{g}$ for $g\in G\subset U(n)$ is unitary with respect to the Frobenius inner product
    we may $F$-orthogonally split $\C^{n\times n}=\mathfrak{g}\oplus \mathfrak{g}^{\perp}$ and from the expressions above we see that also
    \begin{align*}
	\mathbb{N}_{\Phi}: C^{\infty}(M,\mathfrak{g}^{\perp})\to C^{\infty}(M,\mathfrak{g}^{\perp}).	
    \end{align*}
\end{proof}

\subsection{Isomorphism property - proof of Theorem \ref{thm:isomorphism}} \label{sec:isomorphism}

Let us denote $N_\Theta := I_\Theta^* I_\Theta$. As previously pointed out in Remark \ref{remark:themuaffair}, unlike the case where $L^2_\mu$ (for $\mu$ the symplectic measure from Sec. \ref{ssec:linearization}) is chosen as co-domain for $I_\Theta$, $N_\Theta$ is a pseudo-differential operator on $M^{\text{int}}$ which does {\bf not} extend to any simple neighbourhood of $M$. Understanding such an operator will require taking care of interior and boundary behavior separately. The interior behavior is well-known and holds in a broad range of cases, while the boundary behavior makes use of the recent results of \cite{M}. The range of applicability of \cite{M} is geodesic disks of constant curvature, and although what follows could apply to this class of surfaces, we will restrict to the Euclidean disk for simplicity.

\subsubsection{Interior behavior} In the interior, we now show that $N_\Theta$ is a classical elliptic $\Psi$DO of order $-1$, and this actually holds for any simple manifold of dimension $d\ge 2$. Indeed, from the above calculation \eqref{eq:Nexpr}, we first write
\begin{align*}
    N_\Theta f(x) = \int_{S_x} \int_0^{\tau(x,v)} {\mathcal N}_\Theta (x,\exp_x(tv)) f(\exp_x(tv)) j(x,v,t)\ dt\ dS_{x}(v),
\end{align*}
where 
\begin{align}
    {\mathcal N}_\Theta(x,\exp_x(tv)) := \frac{R(x,v) R^{-1} (\varphi_t(x,v))+R(x,-v) R^{-1} (\varphi_{-t}(x,-v)}{\tau(\psi(x,v)) j(x,v,t)},  
    \label{eq:SchwartzN}
\end{align}
and $j(x,v,t)$ denotes the Jacobian of the exponential map $S_x \times (0,\epsilon) \ni (v,t) \to \exp_x (tv) \in M$. The Schwarz kernel of $N_\Theta$ is then ${\mathcal N}_\Theta (x,y)$. Expansions for small $t$ give
\begin{align*}
    \frac{1}{j(x,v,t)} = t^{-d+1} + O(t^{-d+2}), \qquad R(x,v)R^{-1}(\varphi_t(x,v)) = id_{N\times N} + t \Theta(x) + O(t^2),
\end{align*}
and thus the part of the Schwarz kernel that contributes to the principal symbol is given by 
\begin{align*}
    \frac{2}{d_g(x,y)^{d-1} \ell(x,y)} id_{N\times N},
\end{align*}
where $\ell$ denotes the length of the maximal geodesic passing through $(x,y)$. 

\subsubsection{Boundary behavior} We now focus on the case of the {\bf Euclidean disk}, where $g=e$, $d=2$ and the geodesic flow takes the form $\varphi_t(x,v) = (x+tv,v)$. We now recall the theory described in the case $\Theta = 0$, as outlined in \cite{M}. Consider $x = (\rho\cos\omega, \rho\sin\omega)$ polar coordinates on the unit disk, and define\footnote{The $4\pi$ factor is not directly incorporated in the definition of $\mL$ in \cite{M}, though it helps avoid a proliferation of constants here, and only changes the results of \cite{M} by powers of $4\pi$.} the unbounded operator
\begin{align*}
    \mL := (4\pi)^{-2} [- \left( (1-\rho^2) \partial_\rho^2 + (\rho^{-1} - 3\rho) \partial_\rho + \rho^{-2} \partial_\omega^2 \right) + 1],
\end{align*}
with domain $C^\infty(M)$. Then $\mL$ is essentially self-adjoint on $L^2(M)$ with known (pure point) spectral decomposition
\begin{align*}
    \{Z_{n,k},\ n\in \Nm_0,\ 0\le k\le n\}, \qquad \lambda_{n} = (4\pi)^{-2} (n+1)^{2}. 
\end{align*}
The eigenfunctions are (Zernike) polynomials, hence smooth on $M$. We then define the Hilbert scale $\{\wtH^s(M)\}_{s\ge 0}$ by 
\begin{align}
    \wtH^s = \wtH^s(M) := \left\{ f = \sum_{n,k} f_{n,k} \widehat{Z_{n,k}},\quad (4\pi)^{-2s}\sum_{n=0}^\infty (n+1)^{2s} \sum_{k=0}^n |f_{n,k}|^2 < \infty \right\},
    \label{eq:wtH0}
\end{align}
where the hat denotes $L^2$-normalization. It is then proved in \cite[Lemma 3]{M} that $\cap_{s\ge 0} \wtH^s = C^\infty(M)$. Moreover, following \cite[Lemmas 13-14]{M}, there exists $\alpha>3/2$ and $\ell>2$ such that for any $u\in C^\infty(M)$ and $s\in \Nm_0$, we have 
\begin{align}
    \|u\|_{\wtH^{2s}} = \|\mL^s u\|_{L^2(M)} \lesssim \|u\|_{C^{2s}} \lesssim \|u\|_{\wtH^{\alpha + 2s\ell}},  
    \label{eq:nottame}
\end{align}
where for $k\in \Nm_0$, we define the $C^k$ norm $\|u\|_{C^k} = \sup_{x\in M} \sum_{|\alpha|\le k} |\partial^\alpha u(x)|$. Therefore, the topological dual of $C^{\infty}(M)$ equipped with the family of semi-norms $\{\|\cdot\|_{\wtH^s}\}_{s\in \Nm_0}$ coincides with that of $C^\infty(M)$ equipped with the family of $C^k(M)$ norms, the latter being the space of {\em supported distributions} $\dot C^{-\infty}(M)$. 

As a result, $\mL$ can be extended by duality to $\dot C^{-\infty} (M)$ through the pairing $\langle \mL u, \phi \rangle := \langle u, \mL \phi\rangle$ (if by $\langle\cdot,\cdot\rangle$ we denote the $(\dot C^{-\infty}(M), C^\infty(M))$ pairing). An element $u\in \dot C^{-\infty} (M)$ will be said to be in $L^2(M)$ if there exists a constant $C$ such that for any $\phi\in C^\infty(M)$, $|\langle u, \phi \rangle| \le C \|\phi\|_{L^2(M)}$. Definition \eqref{eq:wtH0} may then be extended to $s\in \R$, and each space can be identified as 
\begin{align}
    \wtH^s = \{ u\in \dot C^{-\infty}(M), \quad \mL^{s/2} u\in L^2 \}, \qquad \|u\|_{\wtH^s} := \|\mL^{s/2}u\|_{L^2}.
    \label{eq:wtH}
\end{align}

As this Sobolev scale is not the classical one (it is modeled after an elliptic operator whose ellipticity degenerates at the boundary), we state a few facts which are reminiscent of the traditional scales: 

\begin{Lemma} \label{lem:scale}
    The scale $\{\wtH^s\}_{s\in \R}$ satisfies the following:
    \begin{enumerate}
	\item[(a)] Using $L^2$ as pivot space, for every $s\ge 0$, we have $(\wtH^s)' = \wtH^{-s}$.
	\item[(b)] For any $s,t \in \R$ such that $t<s$, the injection $\wtH^s\subset \wtH^t$ is compact. 
	\item[(c)] For any $0\le s<t$ and $\theta\in [0,1]$, we have $[\wtH^t, \wtH^s]_\theta = \wtH^{\theta s + (1-\theta)t}$.     
    \end{enumerate}
\end{Lemma}

\begin{proof} The definition \eqref{eq:wtH0} makes each $\wtH^s$ isomorphic to a weighted $\ell^2$ space. Then (a) follows directly from the fact that for any sequence of positive numbers $\{\lambda_n\}_n$, 
    \begin{align*}
	\sum_{n\in \Nm} u_n \overline{v_n} \le \left( \sum_{n\in\Nm} \lambda_n^2 u_n^2 \right)^{1/2} \left( \sum_{n\in\Nm} \lambda_n^{-2} v_n^2 \right)^{1/2}.
    \end{align*}
    Then (b) is an immediate consequence of the fact that for any sequence $\{\lambda_n\}_n$ decreasing to zero, the operator $T_\lambda\colon \ell^2\to \ell^2$ given by $\{u_j\}_{j\in \Nm} \mapsto \{\lambda_j u_j\}_{j\in \Nm}$ is compact. 

    Finally, (c) follows readily from the general complex interpolation result \cite[Proposition 2.2]{T}, bearing in mind that $\wtH^s$ is nothing but the domain space ${\mathcal D}(\mL^{s/2})$.
\end{proof}

Furthermore, we have that for any $s\in \R$ and any $u\in \wtH^s$, $\|N_0 u\|_{\wtH^{s+1}} = \|u\|_{\wtH^s}$. Moreover, the following identity is given in \cite[Theorem 11]{M}
\begin{align}
    \mL N_0^2 = id|_{C^\infty(M)},
    \label{eq:id}
\end{align}
and this equality extends to $\dot C^{-\infty}(M)$ by density. Therefore, $N_0$ is an isomorphism of $C^\infty(M)$ (in fact, a bijection of $\dot C^{-\infty}(M)$), and the work below will imply that this remains true for $N_\Phi$, by showing that $N_\Phi$ is a relatively compact perturbation of $N_0$ on the $\wtH^s$ scale.  

Morally, the $\wtH^s$ scale behaves like the usual Sobolev scale in the interior of $M$ (while allowing for faster radial oscillations near the boundary). This is summarized in Lemma \ref{lem:scale_equiv} below, in stark contrast with \eqref{eq:nottame}. Here and below, we write $U\Subset M^{int}$ for a set $U$ which is relatively compact in $M^{int}$. If $U$ is open, we have the natural operators of extension-by-zero $e_U\colon C^\infty_c(U)\to C^\infty(M)$ and restriction $r_U \colon C^\infty(M) \to C^\infty(U)$, which extend by duality to $e_U = r_U^t\colon {\mathcal E}'(U)\to \dot C^{-\infty}(M)$ and $r_U = e_U^t \colon \dot C^{-\infty}(M) \to \mD'(U)$. We also have $r_U e_U = id|_{ {\mathcal E}'(U)}$, and $\mL e_U = e_U \mL$ (where $\mL$, being a differential operator, will be viewed either as continuous on ${\mathcal E}'(U)\to {\mathcal E}'(U)$ or $\dot C^{-\infty}(M) \to \dot C^{-\infty}(M)$). 

\begin{Lemma} \label{lem:scale_equiv}
    Fix an open set $U\Subset M^{int}$ and an integer $p\ge 0$. Then for any $u\in {\mathcal E}'(U)$, we have that $u\in H^{2p}(U)$ if and only if $e_U u \in \wtH^{2p}$. Moreover there exist constants $C_1(U,p)$ and $C_2(U,p)$ such that 
    \begin{align}
	C_1 \|u\|_{H^{2p}(U)} \le \|e_U u\|_{\wtH^{2p}} \le C_2 \|u\|_{H^{2p}(U)}, \qquad \forall u\in H^{2p}(U) \cap {\mathcal E}'(U).
    \end{align}    
\end{Lemma}

\begin{proof}   
    We then have 
    \begin{align*}
	\|e_U u\|_{\wtH^{2p}} = \|\mL^p e_U u\|_{L^2(M)} = \|\mL^p u\|_{L^2(U)} \le C \|u\|_{H^{2p}(U)},
    \end{align*}
    where the last inequality comes from the fact that $\mL^p$ is a differential operator of order $2p$. For the other inequality, notice that for any $u\in {\mathcal E}'(U)$ and any $p\in \Nm_0$, we have $e_U u = N_0^{2p} \mL^p e_U u$, and upon applying $r_U$ we obtain $u =  r_U N_0^{2p} e_U \mL^p u$. We now claim that there is a constant such that 
    \begin{align}
	\|r_U N_0^{2p} e_U v\|_{H^{2p}(U)}\le \|v\|_{L^2(U)}, \qquad \forall v\in L^2(U).
	\label{eq:claim}
    \end{align}
    In that case, we write
    \begin{align*}
	\|u\|_{H^{2p}(U)} &= \|r_U N_0^{2p} e_U \mL^p u\|_{H^{2p}(U)} \\
	&\le C \|\mL^p u\|_{L^2(U)} = C \|\mL^p e_U u\|_{L^2(M)} = C\|e_U u\|_{\wtH^{2p}},
    \end{align*}
    completing the proof of the lemma. 
    
    To prove \eqref{eq:claim}: given $U'$ an open set such that $U\Subset U'\Subset M^{int}$, define $e_{U,U'} \colon {\mathcal E}'(U) \to {\mathcal E}'(U')$ and $r_{U',U} \colon \mD'(U')\to \mD'(U)$ the operators of extension by zero and restriction. With $\chi\in C_c^\infty(U')$ equal to $1$ in a neighborhood of $U$, the operators $r_U N_0^{2p} e_U$ and $r_{U',U} \chi r_{U'} N_0^{2p} e_{U'} \chi e_{U,U'}$ agree. The operator $\chi r_{U'} N_0^{2p} e_{U'} \chi$ is a properly supported element of $\Psi^{-2p} (U')$ and thus by \cite[Theorem 4.7]{GS}, 
    \begin{align*}
	\chi r_{U'} N_0^{2p} e_{U'} \chi \colon L^2_{loc}(U') \to H^{2p}_{loc} (U')
    \end{align*}
    is continuous. In particular, there exists $U''\Subset U'$ and a constant $C$ such that
    \begin{align*}
	\| r_{U',U} \chi r_{U'} N_0^{2p} e_{U'} \chi w \|_{H^{2p} (U)} \le C \|r_{U',U''} w\|_{L^{2}(U'')}, \qquad \forall w\in {\mathcal E}'(U'). 
    \end{align*}
    Applying this inequality to $w = e_{U,U'} v$ for some $v\in {\mathcal E}'(U)$ yields the result.     
\end{proof}

Everything we have done in this section so far generalizes straighforwardly to $\mathbb{C}^N$-valued functions. We may define $\wtH^s(M; \mathbb{C}^N)$ as in \eqref{eq:wtH0} by making the coefficients $f_{n,k}$ to be valued in $\mathbb{C}^N$ with $|f_{n,k}|^2$ the standard Euclidean norm. This scale corresponds to a Sobolev scale with respect to $\mL$ acting on each scalar component. Now denoting $\wtH^s = \wtH^s(M; \mathbb{C}^N)$, Lemmas \ref{lem:scale} and \ref{lem:scale_equiv} still hold true with minor modifications. We now turn to the study of $N_\Theta$, and write $N_\Theta = N_0 + K_\Theta$, where the 'unattenuated' normal operator $N_0$ is thought of as acting diagonally on each component of a $\mathbb{C}^N$-valued function.

\begin{Lemma}\label{lem:PsiDO} For any open set $U\Subset M^{int}$, the following hold. 
    
    (i) The operator $r_U N_\Theta e_U$ is an elliptic element of $\Psi^{-1}(U)$. 
    
    (ii) The operator $r_U K_\Theta e_U $ belongs to $\Psi^{-2}(U)$.
\end{Lemma}

\begin{proof} Fix an open set $U\Subset M^{int}$. For $f\in C_0(U)$ extended by zero outside of $U$, we may write
    \begin{align*}
	r_U N_\Theta e_U f(x) = \int_{S_x} \int_0^\infty A(x,v,t) f(x+tv)\ dt\ dS_{x}(v), \qquad x\in U,
    \end{align*}
    where $A(x,v,t) := \frac{R(x,v)R^{-1} (x+tv,v)+R(x,-v)R^{-1}(x+tv,-v)}{\tau(\psi(x,v))}  \chi(x+tv)$ for $(x,v,t) \in \mD_U$ with 
    \begin{align*}
	\mD_U := \{ (x,v,t),\ (x,v)\in SU,\ t\in \R \},
    \end{align*}
    and where $\chi\in C_c^\infty(M^{int})$ is equal to $1$ on $U$. Then $A\in C^\infty(\mD_U)$ and by \cite[Lemma B.1]{DPSU}, $r_U N_\Theta e_U$ is a classical $\Psi$DO of order $-1$ on $U$ with full symbol $\sigma(x,\xi)\sim \sum_{k=0}^\infty \sigma_k(x,\xi)$, where
    \begin{align*}
	\sigma_k(x,\xi) = \pi \frac{i^k}{k!} \int_{S_x U} \partial_t^k A(x,v,0) \delta^{(k)} (\langle\xi,v\rangle)\ dS_{x}(v).
    \end{align*}
    The principal symbol of $N_\Theta$ is thus given by 
    \begin{align*}
	\sigma_{0} (x,\xi) = 2\pi \int_{S_x} \frac{\delta (\langle \xi, v\rangle)}{\tau(x,v) + \tau(x,-v)}\ dS_{x}(v)\ id_{N\times N} = \frac{4\pi}{|\xi|} \frac{1}{\tau(x,\hat\xi_\perp) + \tau(x,-\hat\xi_\perp)}\ id_{N\times N}.
    \end{align*}
    We also notice that $\sigma_0$ actually does not depend on $\Theta$, in other words, $r_U K_\Theta e_U = r_U (N_\Theta-N_0)e_U \in \Psi^{-2} (U)$. Hence the result.  
\end{proof}

The next lemma is in essence the reason why $N_\Theta$ is a relatively compact perturbation of $N_0$ on the $\wtH^s$ scale.  
\begin{Lemma}\label{lem:compact}
    For any $s\ge 0$, the operators $\mL K_\Theta$ and $K_\Theta\mL$ are $\wtH^{s}\to \wtH^{s}$ bounded. 
\end{Lemma}

\begin{proof} It is enough to prove boundedness for $s = 2p$ with $p\in \Nm_0$, and the general case follows from Lemma \ref{lem:scale}.(c) and the interpolation result \cite[Proposition 2.1]{T}. 

    An important observation is that since $\Theta$ is compactly supported inside $M^{int}$, there exists $\delta>0$ such that for any $x_0\in \partial M$, if $x,y\in B_\delta(x_0)\cap M$, then $\mK_\Theta (x,y) =0$. Indeed, if $\delta$ is so small that $B_{\delta}(x_0)$ does not intersect the support of $\Theta$, and by convexity of the set $B_\delta(x_0)\cap M$, the geodesic segment $[x,y]$ is completely included outside the support of $\Theta$, thus in \eqref{eq:SchwartzN}, writing $y = \text{Exp}_x (tv)$ for some $t,v$, we have that $R(x,v) R^{-1} (\varphi_t(x,v)) = id_{N\times N}$ and hence ${\mathcal N}_0 (x,y) = {\mathcal N}_\Phi (x,y)$ there. 

    Let us then cover $M$ by open balls $\{U_i\}_i$ of small enough diameter that if $U_i \cap U_j\ne \emptyset$ and if either intersects $\partial M$, then $U_i\cup U_j\subset B_\delta (x_0)$ for some $x_0 \in \partial M$. In this scenario, $\mK_\Phi (x,y) = 0$ for any $x\in U_i$ and $y \in U_j$. Consider $\{\psi_i\}_i$ a locally finite partition of unity subordinated to $\{U_i\}_i$, and write $K_\Theta = \sum_{i,j} K_{ij}$ with $K_{ij} (x,y) = \psi_i(x)K(x,y)\psi_j(y)$. Denote by $S_i\Subset U_i$ the support of $\psi_i$. By the comment above, $K_{ij}$ is trivial whenever $U_i\cap U_j\ne \emptyset$ and either set intersects $\partial M$ and we may assume that the non-trivial terms arise either from (I) $U_i\cap U_j = \emptyset$, or (II) $U_i\cap U_j \ne \emptyset$ and $U_i \cup U_j \Subset M^{int}$. 
    
    In case (I), then $\mK_{ij}, \mK_{ji} \in C^\infty (M\times M)$, since these are supported away from the diagonal and the corner of $M\times M$. In particular for any $p\in \Nm$, the Schwartz kernel of $\mL^p K_{ij}$ and $\mL^p K_{ji}$ belongs to $C^\infty (M\times M)$ as well as those of $K_{ji} \mL^p$ and $K_{ij} \mL^p$ by duality. Then for any $p,q$, the Schwartz kernel of $\mL^q K_{ij} \mL^p$ belongs to $C^\infty(M\times M)$, thus $\mL^q K_{ij} \mL^p$ is $L^2\to L^2$ bounded. In particular, $\mL^p K_{ij} \mL$ and $\mL^p \mL K_{ij}$ are $L^2\to L^2$ bounded, which is equivalent to $K_{ij}\mL$ and $\mL K_{ij}$ being $L^2\to \wtH^{2p}$ bounded, and in particular, $\wtH^{2p}\to \wtH^{2p}$ bounded.
    
    In case (II), take open sets $U, U'$ such that $S_i \cup S_j \subset U \Subset U' \Subset M^{int}$. Then from the composition calculus of $\Psi$DO's and Lemma \ref{lem:PsiDO}.(ii), $K_{ij} \mL$ and $\mL K_{ij}$ are properly supported elements of $\Psi^0(U')$, and thus by \cite[Theorem 4.7]{GS}, we have $\mL K_{ij}, K_{ij}\mL\colon H^s_{loc}(U')\to H^{s}_{loc}(U')$ for all $s$. In particular, there exists $V\Subset U'$ and a constant $C$ such that for every $v\in {\mathcal E}'(U)$, $\|\mL K_{ij} v\|_{H^{2p}(U)} \le C \|v\|_{H^{2p}(V)}$. Using Lemma \ref{lem:scale_equiv}, this gives 
    \begin{align*}
	\|e_U \mL K_{ij} v\|_{\wtH^{2p}} \lesssim \|\mL K_{ij} v\|_{H^{2p}(U)} \lesssim \|v\|_{H^{2p}(V)} \lesssim \|e_V v\|_{\wtH^{2p}},
    \end{align*}
    similarly for $K_{ij} \mL$.
    
    On to the proof, for $v\in \dot C^{-\infty}(M)$, we write $\mL K_{\Theta} v = \sum_{i,j} \mL K_{ij} v_j$, where $v_j = \chi_j v$ and where $\chi_j \in C_c^\infty(U_j)$ is equal to $1$ on $S_j$. Then 
    \begin{align*}
	\|\mL K_\Theta v\|_{\wtH^{2p}} \le \sum_{(I)} \|\mL K_{ij} v_j\|_{\wtH^{2p}} + \sum_{(II)} \|\mL K_{ij} v_j\|_{\wtH^{2p}}.
    \end{align*} 
    From the work above, each term involving $v_j$ is $\lesssim \|v_j\|_{\wtH^{2p}}$, which by Leibniz's rule is bounded by $C \|v\|_{\wtH^{2p}}$. The proof for $\mL K_\Theta$ is identical.
\end{proof}

Since $K_\Theta$ is $L^2\to L^2$ self-adjoint and $\mL$ is essentially $L^2\to L^2$ self-adjoint, the transpose of $\mL K_\Theta\colon \wtH^{s} \to \wtH^{s}$ is $K_\Theta \mL\colon \wtH^{-s}\to \wtH^{-s}$, and the transpose of $K_\Theta\mL\colon \wtH^{s} \to \wtH^{s}$ is $\mL K_\Theta \colon \wtH^{-s}\to \wtH^{-s}$, both of which are then bounded by virtue of Lemma \ref{lem:compact}. A consequence of the previous lemma is also that $K_\Theta = \mL^{-1} \circ \mL K_\Phi\colon \wtH^{s}\to \wtH^{s+2}$ is bounded for every $p\in \Nm_0$, and thus that $N_\Theta = N_0 + K_\Theta$ is $\wtH^{s}\to \wtH^{s+1}$ bounded for all $s\ge 0$. Dualizing, the operator $N_\Theta\colon \wtH^{-s-1}\to \wtH^{-s}$ is bounded for all $s\ge 0$. 

We now prove the main theorem of this section. 

\begin{theorem}\label{thm:iso} 
    For all $s\ge 0$, the operator $N_\Theta \colon \wtH^{s} \to \wtH^{s+1}$ is a Hilbert space isomorphism. As a consequence, the operator $N_\Theta\colon C^\infty(M,\mathbb{C}^N)\to C^\infty(M,\mathbb{C}^N)$ is a Fr\'echet space isomorphism. 
\end{theorem}

\begin{proof} 
    We know that $N_\Theta\colon L^2(M,\mathbb{C}^N) \to L^2(M,\mathbb{C}^N)$ is self-adjoint by construction, and injective \cite{PSUGAFA}, and in particular, injective on $\wtH^s$ for any $s\ge 0$. We now prove that this is also true for negative $s$. Indeed for $s<0$, if $u\in \wtH^{s}$ satisfies $0 = N_\Theta u = N_0 u + K_\Theta u$, composing with $\mL^{1/2}$, we obtain the equation $u = - \mL^{1/2} K_\Theta u$. Now from Lemma \ref{lem:compact}, we have that $\mL^{1/2} K_{\Theta} = \mL^{-1/2}\circ \mL K_\Theta \colon \wtH^t\to \wtH^{t+1}$ is continuous for all $t\in \R$, and thus by bootstrapping, $u\in C^\infty(M,\mathbb{C}^N)$. Finally by injectivity of $N_\Theta$ on $C^\infty(M,\mathbb{C}^N)$, we obtain that $N_\Theta$ is injective on $\wtH^{s}$ for any $s \in \R$.    

    On to the surjectivity, fix $s\ge 0$: given $f\in \wtH^{s+1}$, $u\in \wtH^{s}$ solves $N_\Theta u = f$ if and only if $u$ solves $N_0 u + K_\Theta u = f$. Upon composing by $\mL^{1/2}$, this is equivalent to solving for $u\in \wtH^{s}$
    \begin{align}
	u + \mL^{1/2} K_\Phi u = \mL^{1/2} f \in \wtH^{s}.
	\label{eq:fred}
    \end{align}
    As mentioned above the operator $\mL^{1/2} K_\Theta\colon\wtH^{s}\to \wtH^{s+1}$ is bounded, hence $\wtH^{s}\to \wtH^{s}$ compact. As a result, the bounded operator $Id + \mL^{1/2} K_\Theta = \mL^{1/2} N_\Theta\colon \wtH^s \to \wtH^s$ has closed range.     
    Finally, the Hilbert-space adjoint of $\mL^{1/2} N_\Theta\colon \wtH^{s}\to \wtH^{s}$ is $\mL^{-s} N_\Theta \mL^{1/2}\mL^{s}$ and thus, 
    \begin{align*}
	\ran \left(\mL^{1/2} N_\Theta|_{\wtH^s}^{\wtH^s}\right) = \overline{\ran \left(\mL^{1/2} N_\Theta|_{\wtH^s}^{\wtH^s}\right)} = \left(\ker \left(\mL^{-s} N_\Theta \mL^{1/2}\mL^{s}|_{\wtH^s}^{\wtH^s} \right)\right)^\perp.
    \end{align*}
    The latter kernel is directly related to $\ker N_\Theta|^{\wtH^{-s}}_{\wtH^{-s-1}}$, which was proved above to be trivial. As a result, $\mL^{1/2} N_\Theta\colon \wtH^{s}\to \wtH^{s}$ is an isomorphism, and so is $N_\Theta\colon \wtH^{s}\to \wtH^{s+1}$.
\end{proof}


%

\section*{Acknowledgements}
The authors would like to thank the anonymous referees
for their constructive comments that improved the
quality of this paper. F.M. was supported by NSF grant DMS-1814104 and NSF CAREER grant DMS-1943580. R.N. was supported by the European Research Council under ERC grant No. 647812 (UQMSI). G.P.P. was supported by the Leverhulme trust and EPSRC grant EP/R001898/1.



\begin{thebibliography}{4}
\bibitem{AN19} K. Abraham, and R. Nickl, {\it On statistical Cald\'eron problems}, Math. Stat. Learn. \textbf{2} (2019) 165--216.

\bibitem{ALS13} S. Agapiou, S. Larsson, A.M. Stuart, {\it Posterior contraction rates for the Bayesian approach to linear ill-posed inverse problems.} Stochastic Process. Appl. \textbf{123} (2013) 3828--3860.

\bibitem{BR11} G. Bal, and K. Ren, {\it Multi-source quantitative photoacoustic tomography in a diffusive regime.}
Inverse Problems \textbf{27} (2011). 

\bibitem{BU10} G. Bal, and G. Uhlmann, {\it Inverse diffusion theory of photoacoustics.} Inverse Problems \textbf{26} (2010).

\bibitem{BGLFS17} A. Beskos, M. Girolami, S. Lan, P. Farrell, and A. Stuart. {\it Geometric MCMC for infinite-dimensional inverse problems.} Journal of Computational Physics \textbf{335} (2017) 327--351.

\bibitem{CN13} I. Castillo, R. Nickl, {\it Nonparametric Bernstein von Mises theorems in Gaussian white noise}, Annals of Statistics \textbf{41}  (2013) 1999--2028

\bibitem{CN14} I. Castillo, R. Nickl, {\it On the Bernstein von Mises phenomenon for nonparametric Bayes procedures}, Annals of Statistics \textbf{42} (2014) 1941-1969

\bibitem{CR19} I. Castillo, V. Rockova, {\it Uncertainty quantification for Bayesian CART}, arXiv preprint 1910.07635 (2019)

\bibitem{CR15} I. Castillo, J. Rousseau, {\it A Bernstein-von Mises theorem for smooth functionals in semiparametric models}, Annals of Statistics \textbf{43} (2015) 2353--2383

\bibitem{CvdP20} I. Castillo, S. van der Pas, {\it Multiscale Bayesian survival analysis}, arXiv preprint 2005.02889 (2020)

\bibitem{CZ95} K.L. Chung, Z.X. Zhao, {\it From Brownian motion to Schr\"odinger's equation.} Springer 1995.

\bibitem{CMPS16} P.R. Conrad, Y.M. Marzouk, N.S. Pillai, and A. Smith, {\it Accelerating asymptotically exact MCMC for computationally intensive models via local approximations.} Journal of the American Statistical Association \textbf{516} (2016) 1591--1607.

\bibitem{CRSW13} S.L. Cotter, G.O. Roberts, A.M. Stuart, D. White, {\it MCMC methods for functions: modifying old algorithms to make them faster}, Statist. Sci. {\bf 28} (2013) 424-446.

\bibitem{DPSU} N.S. Dairbekov, G.P. Paternain, P. Stefanov, and G. Uhlmann. {\it The boundary rigidity problem in the presence of a magnetic field}, Advances in Mathematics {\bf 216}, no. 2 (2007): 535-609.

\bibitem{DS16} M. Dashti, A. Stuart, The Bayesian approach to inverse problems. In: {\it Handbook of Uncertainty Quantification}, Editors R. Ghanem, D. Higdon and H. Owhadi, Springer (2016).

\bibitem{D02} R.M. Dudley, {\it Real analysis and probability}. CUP 2002

\bibitem{E} G. Eskin, {\it On non-Abelian Radon transform,} Russ. J. Math. Phys.  {\bf 11}  (2004) 391--408.

\bibitem{UQ} R. Ghanem, D. Higdon and H. Owhadi (eds.), {\it Handbook of uncertainty quantification. I-III.} Springer, Cham, 2017. 

\bibitem{GvdV17} S. Ghosal, A. van der Vaart, {\it Fundamentals of non-parametric Bayesian inference}, CUP 2017

\bibitem{GT98} D. Gilbarg, N.S. Trudinger, {\it Elliptic partial differential equations of second order}, Springer 1998.

\bibitem{GN16} E. Gin\'e, R. Nickl, {\it Mathematical foundations of infinite-dimensional statistical models}. CUP 2016

\bibitem{GK19} M. Giordano, H. Kekkonen, {\it Bernstein-von Mises theorems and uncertainty quantification for linear inverse problems}. SIAM J. on Uncert. Quant. \textbf{8} (2020) 342--373

\bibitem{GN19} M. Giordano, R. Nickl, {\it Consistency of Bayesian inference with Gaussian process priors in an elliptic inverse problem}, Inverse Problems \textbf{36} (2020).

\bibitem{GS} A. Grigis, J. Sj\"ostrand, {\it Microlocal analysis for differential operators: an introduction}. Vol. 196. Cambridge University Press (1994).

\bibitem{Gr} G. Grubb, {\it Fractional Laplacians on domains, a development of H\"ormander's theory of $\mu$-transmission pseudodifferential operators,} Adv. Math. {\bf 268} (2015) 478--528.

\bibitem{GVD20} Gugushvili, S., van der Vaart, A, and Yan, D.; Bayesian linear inverse problems in regularity scales. {\it Ann. Inst. Henri Poincar\'e Probab. Stat.} \textbf{56} (2020) 2081--2107.

\bibitem{HSV14} M. Hairer, A. M. Stuart, S. Vollmer, {\it Spectral gaps for a Metropolis-Hastings algorithm in infinite dimensions.} Annals of Applied Probability {\bf 24} (2014) 2455-2490.

\bibitem{Hetal} A. Hilger, I. Manke, N. Kardiljov, M. Osernberg, H. Markotter, J. Banhart, {\it Tensorial neutron tomography of threedimensional magnetic vector fields in bulk materials.} Nature Communications (2018).

\bibitem{VSS13} V.H. Hoang, C. Schwab, A.M. Stuart, {\it Complexity analysis of accelerated MCMC methods for Bayesian inversion.} Inverse Problems \textbf{29} (2013). 

\bibitem{IlMo} J. Ilmavirta and F. Monard, {\it Integral geometry on manifolds with boundary and applications}, The Radon Transform: the first 100 years and beyond. Radon series on computations and applied mathematics {\bf 22} (2019). Editors: R. Ramlau, O. Scherzer.

\bibitem{KKL01} A. Kachalov, Y. Kurylev, M. Lassas, {\it Inverse Boundary Spectral Problems}, Chapman Hall, 2001

\bibitem{KKSV00} J. Kaipio, V. Kolehmainen, E. Somersalo, M. Vauhkonen, {\it Statistical inversion and Monte Carlo sampling methods in electrical impedance tomography.} Inverse Problems \textbf{16} (2000)

\bibitem{KS04} J. Kaipio, and E. Somersalo, {\it Statistical and Computational Inverse Problems.} Springer 2004.

\bibitem{KLS16} H. Kekkonen, M. Lassas, S. Siltanen, {\it Posterior consistency and con- vergence rates for Bayesian inversion with hypoelliptic operators.} Inverse Problems \textbf{32} (2016).

\bibitem{KVV11} B. Knapik, A.W. van der Vaart, H. van Zanten, {\it Bayesian inverse problems with Gaussian priors}, Ann. Statist. {\bf 39} (2011) 2626-2657.

\bibitem{L1812} P.S. Laplace, {\it Th\'eorie analytique des probabilit\'es}, Courcier 1812

\bibitem{LC86} L. Le Cam, {\it Asymptotic methods in statistical decision theory}, Springer 1986

\bibitem{M} F. Monard, {\it Functional relations, sharp mapping properties and regularization of the X-ray transform on disks of constant curvature}, SIAM J. Math. Anal. 52-6 (2020), pp. 5675--5702.

\bibitem{MNP} F. Monard, R. Nickl, G. P. Paternain, {\it Efficient Nonparametric Bayesian inference for X-ray transforms}, Ann. Statist. {\bf 47} (2019) 1113--1147.

\bibitem{MNP19} F. Monard, R. Nickl and G.P. Paternain, {\it Consistent Inversion of Noisy Non-Abelian X-Ray Transforms}, Comm. Pure Appl. Math. \textbf{74} (2021), 1045-1099.

\bibitem{N1} R. Nickl, {\it Bernstein - von Mises theorems for statistical inverse problems I: Schr\"odinger equation,} J. Eur. Math. Soc. {\bf 22} (2020) 2697--2750.

\bibitem{NR20} R. Nickl, and K. Ray, {\it Nonparametric statistical inference for drift vector fields of multi-dimensional diffusions,} Ann. Statist. {\bf 48} (2020) 1383-1408.

\bibitem{NS19} R. Nickl, and J. S\"ohl, {\it Bernstein - von Mises theorems for statistical inverse problems II: Compound Poisson processes,} Electronic J. Stat. \textbf{13} (2019) 3513--3571

\bibitem{NvdGW18} R. Nickl, S. van de Geer, S. Wang, {\it Convergence rates for penalised least squares estimators in PDE constrained regression problems,} SIAM J. Uncert. Quant. \textbf{8} (2020) 374-413.

\bibitem{NW20} R. Nickl, and S. Wang, {\it On polynomial time computation of high-dimensional posterior measures by Langevin-type algorithms,} preprint (2020). 

\bibitem{No} R. Novikov, {\it On determination of a gauge field on $\mathbb{R}^d$ from its non-Abelian Radon transform along oriented straight lines}, J. Inst. Math. Jussieu {\bf 1} (2002) 559--629.

\bibitem{Novikov_nonabelian} R.\ Novikov, {\it Non-Abelian Radon transform and its applications}, R. Ramlau, O. Scherzer. The Radon Transform: The First 100 Years and Beyond, pp. 115--128, 2019. 

\bibitem{PSUGAFA} G. P. Paternain, M. Salo, G. Uhlmann, {\it The attenuated ray transform for connections and Higgs fields}, Geom. Funct. Anal. {\bf 22} (2012) 1460--1489. 

\bibitem{PSU_book} G. P. Paternain, M. Salo, G. Uhlmann, {\it Geometric inverse problems in 2D}, text in preparation.

\bibitem{PU} L. Pestov, G. Uhlmann, {\it Two dimensional compact simple Riemannian manifolds are boundary distance rigid}, Ann. of Math. {\bf 161} (2005), 1089--1106.

\bibitem{R08} Rei\ss, M., {\it Asymptotic equivalence for nonparametric regression with multivariate and random design.}  Ann. Statist. {\bf 36} (2008) 1957--1982. 

\bibitem{R13} K. Ray, {\it Bayesian inverse problems with non-conjugate priors}, Electronic J. Stat. {\bf 7} (2013) 2516--2549.

\bibitem{R17} K. Ray, {\it Adaptive Bernstein-von Mises theorems in Gaussian white noise}, Annals of Statistics \textbf{45} (2017) 2511--2536. 

\bibitem{RHL14} L. Roininen, J.M. Huttunen, S. Lasanen, {\it Whittle-Mat\'ern priors for Bayesian statistical inversion with applications in electrical impedance tomography} Inverse Probl. Imaging \textbf{8} (2014)

\bibitem{Sales17} M. Sales et al, {\it Three Dimensional Polarimetric Neutron Tomography of Magnetic Fields,}  Nat. Sc. Rep. (2018)


\bibitem{Sharafutdinov} V.A. Sharafutdinov, {\it Integral Geometry of Tensor Fields}, VSP (1994).


\bibitem{SU} Mikko Salo and Gunther Uhlmann, {\em the attenuated ray transform on simple surfaces}, Journal of Diff. Geom., 2011, 88, pp. 161--187. 

\bibitem{S10} A. Stuart, {\it Inverse problems: a Bayesian perspective.} Acta Numer. {\bf 19} (2010) 451-559.


\bibitem{T} M.E. Taylor, {\it Partial differential equations. 1, Basic theory}. Springer (1996).

\bibitem{vdV91} A.W. van der Vaart, {\it On differentiable statistical functionals}, Annals of Statistics \textbf{19} (1991) 178--204.

\bibitem{vdV98} A.W. van der Vaart, {\it Asymptotic statistics}, CUP (1998).

\bibitem{vdVvZ08} A.W. van der Vaart, H. van Zanten, {\it Rates of contraction of posterior distributions based on Gaussian process priors.} Annals of Statistics \textbf{36} (2008) 1435--1463.

\bibitem{vM31} R. von Mises, {\it Wahrscheinlichkeitsrechnung}, Deuticke, Vienna (1931)




\end{thebibliography}


\end{document}